\documentclass[a4paper,11pt]{article} 

\usepackage{etoolbox}
\newtoggle{arxiv} 
\toggletrue{arxiv} 

\makeatletter

\iftoggle{arxiv}{
 \usepackage[margin=25mm]{geometry}
}{
\renewcommand\normalsize{%
    \@setfontsize\normalsize{11.7}{14pt plus .3pt minus .3pt}%
    \abovedisplayskip 10\p@ \@plus4\p@ \@minus4\p@
    \abovedisplayshortskip 6\p@ \@plus2\p@
    \belowdisplayshortskip 6\p@ \@plus2\p@
    \belowdisplayskip \abovedisplayskip}
\renewcommand\small{%
    \@setfontsize\small{9.5}{12\p@ plus .2\p@ minus .2\p@}%
    \abovedisplayskip 8.5\p@ \@plus4\p@ \@minus1\p@
    \belowdisplayskip \abovedisplayskip
    \abovedisplayshortskip \abovedisplayskip
    \belowdisplayshortskip \abovedisplayskip}
\renewcommand\footnotesize{%
    \@setfontsize\footnotesize{8.5}{9.25\p@ plus .1pt minus .1pt}
    \abovedisplayskip 6\p@ \@plus4\p@ \@minus1\p@
    \belowdisplayskip \abovedisplayskip
    \abovedisplayshortskip \abovedisplayskip
    \belowdisplayshortskip \abovedisplayskip}
\setlength\parindent    {30\p@}
\setlength\textwidth    {412\p@}
\setlength\textheight   {570\p@}
\paperwidth=210mm
\paperheight=260mm
\ifdefined\pdfpagewidth
\setlength{\pdfpagewidth}{\paperwidth}
\setlength{\pdfpageheight}{\paperheight}
\else
\setlength{\pagewidth}{\paperwidth}
\setlength{\pageheight}{\paperheight}
\fi
\calclayout
}

\makeatother

\usepackage[dvipsnames]{xcolor} 
\usepackage{amsmath,amsthm,amssymb}
\usepackage{enumitem}
\usepackage{tikz}
\usepackage{hyperref}
\hypersetup{colorlinks, linkcolor=blue}
\usepackage{amsfonts,marvosym,amscd}
\usepackage{mathtools}

\usepackage{mathscinet}
\usepackage{wasysym}
\usepackage{stmaryrd}
\usepackage{graphicx}
\usepackage{psfrag}
\usepackage{datetime}
\usepackage{tikz-cd} 
\usepackage{adjustbox}
\usepackage{subcaption}
\usepackage[mathcal]{euscript} 
\usepackage{mathrsfs} 
\usepackage{ytableau}
\ytableausetup{boxsize=.5em}
\ytableausetup{centertableaux}
\usepackage{verbatim}
\usepackage[normalem]{ulem}

\usepackage{tabstackengine}
\setstackgap{L}{.75\baselineskip}

\input xy
\xyoption{all}
\usepackage[all,knot]{xy}

\newcommand{\ca}{{\mathcal A}}
\newcommand{\cA}{{\mathcal A}}
\newcommand{\cb}{{\mathcal B}}
\newcommand{\cc}{{\mathcal C}}

\newcommand{\cd}{{\mathcal D}}
\newcommand{\cD}{{\mathcal D}}
\newcommand{\ce}{{\mathcal E}}

\newcommand{\cf}{{\mathcal F}}

\newcommand{\cg}{{\mathcal G}}
\newcommand{\ch}{{\mathcal H}}
\newcommand{\cH}{{\mathcal H}}

\newcommand{\cK}{{\mathcal K}}

\newcommand{\cm}{{\mathcal M}}

\newcommand{\cn}{{\mathcal N}}

\newcommand{\cp}{{\mathcal P}}

\newcommand{\cR}{{\mathcal R}}

\newcommand{\cS}{{\mathcal S}}
\newcommand{\ct}{{\mathcal T}}
\newcommand{\cT}{{\mathcal T}}

\newcommand{\pvd}{\operatorname{pvd}\nolimits}

\newcommand{\Gr}{\operatorname{Gr}\nolimits}

\newcommand{\Hom}{\operatorname{Hom}\nolimits}

\newcommand{\Imm}{\operatorname{Im}\nolimits}
\newcommand{\Iso}{\operatorname{Iso}\nolimits}

\newcommand{\Ker}{\operatorname{Ker}\nolimits}
\newcommand{\Kerr}{\operatorname{Ker}\nolimits}

\newcommand{\Ext}{\operatorname{Ext}\nolimits}

\renewcommand{\mod}{\operatorname{mod}\nolimits}

\newcommand{\perf}{\operatorname{perf}\nolimits}
\newcommand{\proj}{\operatorname{proj}\nolimits}

\newcommand{\Z}{\operatorname{\mathbb{Z}}\nolimits}
\newcommand{\QQ}{\operatorname{\mathbb{Q}}\nolimits}
\newcommand{\C}{\operatorname{\mathbb{C}}\nolimits}

\newcommand{\R}{\operatorname{\mathbb{R}}\nolimits}
\newcommand{\F}{\operatorname{\mathbb{F}}\nolimits}

\newcommand{\aut}{\operatorname{Aut}\nolimits}

\newcommand{\rank}{\operatorname{rk}\nolimits}

\newcommand{\mt}{\operatorname{\bf{mt}}}

\newcommand{\hdf}{\mathscr{H}_1} 
\newcommand{\hdff}{\mathscr{H}_2} 
\newcommand{\hdfx}{\mathscr{L}_1}
\newcommand{\hdffx}{\mathscr{L}_2} 
\newcommand{\hallcy}{\ch_{n\mathrm{CY}}} 
\newcommand{\halluncy}{\ch_{1\mathrm{CY}}} 
\newcommand{\hallminusuncy}{\ch_{(-1)\mathrm{CY}}} 
\newcommand{\rulings}{\mathcal{R}} 
\newcommand{\rpoly}{P} 
\newcommand{\rpolyfun}{\mathscr{P}} 
\newcommand{\ltanglecat}{\mathsf{LTangle}}
\newcommand{\spancat}{\mathsf{Span}} 
\newcommand{\vectcat}{\mathsf{Vect}}
\newcommand{\typeacat}[1]{\mathscr A_{#1}} 
\newcommand{\degfin}{degreewise finite} 

\newcommand{\defword}[1]{\textbf{#1}} 

\usetikzlibrary{decorations.pathmorphing}
\usetikzlibrary{decorations.pathreplacing}
\usetikzlibrary{positioning}
\usetikzlibrary{calc}

\makeatletter
\tikzset{
    dot diameter/.store in=\dot@diameter,
    dot diameter=3pt,
    dot spacing/.store in=\dot@spacing,
    dot spacing=10pt,
    dots/.style={
        line width=\dot@diameter,
        line cap=round,
        dash pattern=on 0pt off \dot@spacing
    }
}
\makeatother

\newtheorem{theorem}{Theorem}[section]
\newtheorem{thmdef}{Theorem/Definition}[section]
\newtheorem{lemma}[theorem]{Lemma}
\newtheorem{proposition}[theorem]{Proposition}
\newtheorem{corollary}[theorem]{Corollary}

\theoremstyle{definition}
\newtheorem{definition}[theorem]{Definition}

\newtheorem{example}[theorem]{Example}

\theoremstyle{remark}
\newtheorem{remark}[theorem]{Remark}

\begin{document}

\iftoggle{arxiv}{
\title{Counting in Calabi--Yau categories,\\
with applications to Hall algebras and knot polynomials}
}{
\title[Counting in Calabi--Yau categories]{Counting in Calabi--Yau categories,\\
with applications to Hall algebras and knot polynomials}
}

\author{Mikhail Gorsky and Fabian Haiden}

\date{}

\newcommand{\Addresses}{{
  \setlength{\parindent}{0pt}
  \bigskip
  \footnotesize
    MG: \textsc{Faculty of Mathematics, University of Vienna, Oskar-Morgenstern-Platz 1, 1090 Vienna, Austria} \par\nopagebreak
    \textit{E-mail:} \texttt{mikhail.gorskii@univie.ac.at} 
	\medskip
    
	FH: \textsc{Centre for Quantum Mathematics, Department of Mathematics and Computer Science, University of Southern Denmark, Campusvej 55, 5230 Odense, Denmark} \par\nopagebreak
	\textit{E-mail:} \texttt{fab@sdu.dk} 
}}

\maketitle

\begin{abstract}
We show that homotopy cardinality --- a priori ill-defined for many dg-categories, including all periodic ones --- has a reasonable definition for even-dimensional Calabi--Yau (evenCY) categories and their relative generalizations (under appropriate finiteness conditions).

As a first application we solve the problem of defining an intrinsic Hall algebra for {\degfin} pre-triangulated dg-categories in the case of oddCY categories. 
We compare this definition with To\"en's derived Hall algebras (in case they are well-defined) and with other approaches based on extended Hall algebras and central reduction\iftoggle{arxiv}{, including a construction of Hall algebras associated with Calabi--Yau triples of triangulated categories}. For a category equivalent to the root category of a 1CY abelian category $\ca$, the algebra is shown to be isomorphic to the Drinfeld double of the twisted Ringel--Hall algebra of $\ca$, thus resolving in the Calabi--Yau case the long-standing problem of realizing the latter as a Hall algebra intrinsically defined for such a triangulated category.

Our second application is the proof of a conjecture of Ng--Rutherford--Shende--Sivek, which provides an intrinsic formula for the ruling polynomial of a Legendrian knot $L$, and its generalization to Legendrian tangles, in terms of the augmentation category of $L$.
\end{abstract}

\setcounter{tocdepth}{2} 
\tableofcontents

\section{Introduction}

Calabi--Yau categories and functors are abundant in algebraic geometry, topology, representation theory, symplectic and contact geometry. 
We propose a formalism for counting objects in {\degfin} dg-categories satisfying Calabi--Yau-type conditions. 
It provides a new approach to existing invariants and constructions, such as:
\begin{itemize}
    \item Homotopy cardinality of evenCY categories,
    \item Drinfeld doubles of Hall algebras of 1CY categories,
    \item (Semi-)derived Hall algebras of oddCY categories,
    \item $\Z/2m$-graded ruling polynomials of Legendrian knots,
    \item The leading part of the HOMFLY-PT polynomial of a topological knot,
\end{itemize}
but can also be applied in far more general settings, as we hope to demonstrate. 

The theory can be developed at various categorical levels, each containing the previous one as a special case.

\subsection{Level 0: Homotopy cardinality and evenCY categories}
\label{subsec: level 0}

A natural way of defining the cardinality of a (finite) groupoid, $\cg$, is as the number of isomorphism classes, $x$, counted with weight $|\Hom(x,x)|^{-1}$.
This generalizes to $\infty$-groupoids/spaces, $X$, whose \textit{homotopy cardinality} is defined as
\begin{equation}\label{hcard}
    \sum_{x\in\pi_0(X)}\prod_{i=1}^\infty|\pi_i(X,x)|^{(-1)^i}    \in\QQ_{\geq 0}
\end{equation}
assuming that $\pi_i(X)$ is trivial for large $i$ and finite for all $i$.
This notion was considered by Kontsevich~\cite{kontsevich_3TQFT} (under the name \textit{homotopy Euler characteristic}), Quinn~\cite{quinn95} (under the name \textit{homotopy order}) and Baez--Dolan~\cite{BaezDolan01}.

As a special case, suppose that $\ca$ is a dg-category  with finitely many isomorphism classes of objects over a finite field $\F_q$, then we can apply~\eqref{hcard} to its space of objects, $\ca^\sim$, i.e. the maximal Kan complex in the dg-nerve of $\ca$. 
Then $\pi_1(\ca^\sim,x)=\mathrm{Aut}(x)\subset\Ext^{0}(x,x)$ and $\pi_i(\ca^\sim,x)=\Ext^{1-i}(x,x)$ for $i\geq 2$, thus~\eqref{hcard} becomes
\begin{equation}\label{hcard_dg}
    \sum_{x\in\pi_0(\ca^\sim)}\frac{q^{-\langle x,x\rangle_{< 0}}}{|\mathrm{Aut}(x)|}, 
\end{equation}
where
\begin{equation}\label{truncatedEuler}
    \langle x,y\rangle_{< 0}\coloneqq\sum_{i<0}(-1)^i\dim_{\F_q}\Ext^{i}(x,y)
\end{equation}
is the \textit{truncated Euler pairing}.

For~\eqref{truncatedEuler} and thus~\eqref{hcard_dg} to be well-defined, one needs $\Ext^{i}(x,y)$ to be trivial for $i\ll 0$.
In particular, this assumption fails
for \textit{periodic} dg-categories, i.e. those with $\Ext^i(x,y)\cong\Ext^{i+N}(x,y)$ for some fixed period $N$.

Our basic observation is that there is a good replacement of~\eqref{hcard_dg} for those dg-categories over $\F_q$ which are 
\begin{enumerate}
    \item \textit{\degfin}: $\dim \Ext^i(x,y)<\infty$ for all $x, y$ and for all $i\in\Z$ (but possibly $\Ext^i(x,y)\neq 0$ for infinitely many $i$),
    \item \textit{$n$-Calabi--Yau}: $\Ext^i(x,y)\cong\left(\Ext^{n-i}(y,x)\right)^\vee$ for all $x, y$, for some \textit{even} $n$.
\end{enumerate}
Note that the second condition implies the first.

To find this replacement, suppose that $\ca$ satisfies the above conditions for some $n \geq 0$ but also $\Ext^i(x,y)=0$ for $|i|\gg 0$, then
\[
\langle x,x\rangle_{<0}=\langle x,x\rangle -\langle x,x\rangle_{\geq 0}=\langle x,x\rangle -\langle x,x\rangle_{0,\ldots,n}-\langle x,x\rangle_{<0},
\]
hence
\begin{equation}\label{truncatedEulerCY}
    \langle x,x\rangle_{<0}=\frac{\langle x,x\rangle-\langle x,x\rangle_{0,\ldots,n}}{2},   
\end{equation}
where we use the notation $\langle x,y\rangle_I=\sum_{i\in I}(-1)^i\dim_{\F_q}\Ext^{i}(x,y)$ for $I\subset\Z$, generalizing~\eqref{truncatedEuler}, and $\langle x,y\rangle=\langle x,y\rangle_{\Z}$ is the Euler pairing.
The term $\frac{\langle x,x\rangle}{2}$ is still ill-defined in the general {\degfin} case, however it turns out we can just drop it and work with a kind of ``twisted homotopy cardinality''. This twisting by (one half of) the Euler pairing is familiar from constructions of Hall algebras.
Our replacement for \eqref{hcard_dg} is then
\[
\sum_{x\in\pi_0(\ca^\sim)}\frac{q^{\frac{1}{2}\langle x,x\rangle_{0,\ldots,n}}}{|\mathrm{Aut}(x)|}.
\]
In a similar way, for $n < 0$, we replace \eqref{hcard_dg} by 
\[
\sum_{x\in\pi_0(\ca^\sim)}\frac{q^{-\frac{1}{2}\langle x,x\rangle_{n+1,\ldots,-1}}}{|\mathrm{Aut}(x)|}.
\]

Going beyond the case of finite fields, we expect that one could still make sense of $\eqref{hcard_dg}$, not as a rational number, but as an element in a (localized) ring of motives as in~\cite{KS08}, where $q$ is replaced by the motive of the affine line. This requires the structure of a constructible stack on the set of objects of $\ca$. We do not pursue this direction here.

\subsection{Level 1: Calabi--Yau functors and Hall algebras}

To get more mileage out of the above idea, one considers even-dimensional (weak right) \textit{relative}  CY structures, whose definition was sketched by To\"en in~\cite[Section 5.3]{ToenDAG} and whose theory was further developed by Brav--Dyckerhoff~\cite{BravDyckerhoffRelCY}. 
An $n$CY structure on a functor $f:\cb\to\ca$ of $\mathbf k$-linear dg-categories includes an $(n-1)$CY structure on $\ca$ and ensures the existence of a long exact sequence
\[
\begin{tikzcd}
\cdots \arrow[r] & 
\Ext^i_{\cb}(x,y) \arrow[r]
& \Ext^i_{\ca}(f(x),f(y)) \arrow[r]
\arrow[d, phantom, ""{coordinate, name=Z}]
& \left(\Ext^{n-i-1}_{\cb}(y,x)\right)^{\vee} \arrow[dll,
rounded corners,
to path={ -- ([xshift=2ex]\tikztostart.east)
|- (Z) [near end]\tikztonodes
-| ([xshift=-2ex]\tikztotarget.west)
-- (\tikztotarget)}] \\
& \Ext^{i+1}_{\cb}(x,y) \arrow{r} \arrow[r]
& \cdots 
\end{tikzcd}
\]
reminiscent of the long exact sequence in cohomology of an oriented compact manifold with boundary. (See Definition~\ref{def:CYfunctor} for the precise version of the definition that we will use.)
Assuming $n$ is even, $\mathbf k=\F_q$, and suitable finiteness conditions, we define the ``integral over $\cb$'' for a function $\alpha$ on $\pi_0(\ca^\sim)$ as
\begin{equation}\label{IntRelCY}
    \hdfx (f)(\alpha)\coloneqq\sum_{x\in\pi_0(\cb^\sim)}\frac{|\mathrm{Aut}(f(x))|^{\frac{1}{2}}}{|\mathrm{Aut}(x)|} q^{\frac{1}{2}\gamma(f,x)}\alpha(f(x)),
\end{equation}
where 
\begin{equation}
    \gamma(f,x)\coloneqq \begin{cases} \langle x,x\rangle_{0,\ldots,n-1}+\mathrm{rk}\left(\Ext^0(x,x)^\vee\to\Ext^n(x,x)\right), & \mbox{for} \, n \geq 0;\\
    -\langle x,x\rangle_{n,\ldots,-1}+\mathrm{rk}\left(\Ext^0(x,x)^\vee\to\Ext^n(x,x)\right), & \mbox{for} \, n < 0;
    \end{cases}
\end{equation}
and the map $\Ext^0(x,x)^\vee\to\Ext^n(x,x)$ comes from the long exact sequence above.
In the above, $\alpha$ should be thought of as ``half-density'' on $\ca^\sim$. (See Subsection~\ref{subsec_hcpercy} for some justification for this terminology.) 
Our main result about~\eqref{IntRelCY} is a functoriality property for \textit{CY spans}, i.e. CY functors of the form $f=(f_1,f_2):\cb\to\ca_1\times\ca_2$.
For such a span, satisfying finiteness conditions, we get a linear map
\[
\hdfx(f):\hdfx(\ca_1)\to \hdfx(\ca_2)
\]
where $\hdfx(\ca_i)\coloneqq\C\pi_0(\ca_i^\sim)$ is the space of ``half-densities'' and $\hdfx(\ca_i)\hookrightarrow\hdfx(\ca_i)^\vee$ via the standard inner product.

Given a pair of CY spans $(f_1,f_2):\cb_1\to \ca_1\times\ca_2$ and $(g_1,g_2):\cb_2\to \ca_2\times\ca_3$, the \textit{composed span} $\cb\to\ca_1\times\ca_3$ obtained from the diagram
\begin{equation*}
    \begin{tikzcd}
           &       & \cb   \arrow[dl]\arrow[dr]&       & \\
           & \cb_1 \arrow[dl]\arrow[dr]&       & \cb_2 \arrow[dl]\arrow[dr]& \\
     \ca_1 &       & \ca_2 &       & \ca_3, 
    \end{tikzcd}
\end{equation*}
where the square is a homotopy cartesian square of dg-categories, has a natural CY structure. 
Thus, we may consider a category whose morphisms are CY spans.
For fixed $n$ and $q$ denote by $\spancat^{n\mathrm{CY}}$ the category with objects the {\degfin} dg-categories over $\F_q$ with $(n-1)$CY structure and morphisms the spans $(f_1,f_2)$, up to equivalence, admitting an $n$CY structure and so that $\pi_0(f_1^\sim)$ has finite fibers.

\begin{theorem}\label{intro_FunctorialityThm}
Let $n$ be even, then $\hdfx$ is a functor from $\spancat^{n\mathrm{CY}}$ to the category of vector spaces and linear maps over $\C$.
\end{theorem}

This is Theorem~\ref{thm:functoriality} in the main text.

\subsubsection{Analogies}

Before discussing applications, we will briefly highlight some connections and analogies with geometry and mathematical physics. These will not be used in the rest of the paper, but served as inspiration for us and perhaps as a useful source of intuition for the reader.

Following Kontsevich--Soibelman~\cite{KS_Ainfty} and Pantev--To\"en--Vaqui\'e--Vezzosi~\cite{PTVV13}, one can think of $n$CY categories as non-commutative spaces with $(2-n)$-shifted symplectic structure.
Here, we interpret $\Hom^\bullet(x,x)[1]$ as the ``tangent space'' to the object $x$.
Furthermore, motivated by the main results of~\cite{BravDyckerhoffRelCYII} we can think of CY functors as inclusions of Lagrangian subspaces.
In particular, CY spans are analogous to Lagrangian correspondences, the morphisms in Weinstein's symplectic ``category''~\cite{weinstein_sympcat}.
Thus, $\spancat^{n\mathrm{CY}}$ is a non-commutative analog of the symplectic category and our functor $\hdfx$ is a kind of prequantization procedure.

A fundamental observation in supergeometry, made by Khudaverdian in~\cite{khudaverdian04}, is that while supermanifolds, $M$, with \textit{even} symplectic structure have a canonical volume form (Liouville's theorem), the story is different in the odd case. 
What makes sense instead is the integral
\begin{equation}\label{semidensityInt}
\int_Ls    
\end{equation}
where $s$ is a half-density on $M$ and $L\subset M$ is a Lagrangian submanifold.
These structures play a prominent role in the Batalin--Vilkovisky formalism in mathematical physics.
We suggest that the above integral is analogous to~\eqref{IntRelCY}, where $L\hookrightarrow M$ corresponds to the evenCY functor $f:\cb\to\ca$ and $s$ to $\alpha$.
Vaguely speaking (we do not know how to make this more precise), both \eqref{IntRelCY} and \eqref{semidensityInt} are well-defined for the same formal reason.

\subsubsection{Application: Hall algebras for oddCY categories}

The most elementary version of the Hall algebra construction takes as input an essentially small abelian category, $\ca$, for which both $\Hom(x,y)$ and $\Ext^1(x,y)$ are finite for any $x,y\in\ca$.
The \textit{Hall algebra} is then the $\Z$-algebra (ring) with basis the set of isomorphism classes of objects in $\ca$ and product
\begin{equation}
    z\cdot x\coloneqq\sum_y\left|\left\{x'\subseteq y\mid x'\cong x, y/x'\cong z\right\}\right|y
\end{equation}
which is associative and unital. If the Euler pairing
\[
\langle x, y \rangle = \sum_{i \geq 0} (-1)^i \dim_{\F_q} \Ext^i(x, y)
\]
on the category $\ca$ is well-defined, then the natural object to study is the \emph{twisted} Hall algebra $(\ch(\ca)_{\scriptscriptstyle{\mathrm{tw}}}, \ast)$, which is the same vector space as the Hall algebra (after extension of scalars to $\QQ(q^{\frac{1}{2}})$) but with the twisted multiplication
\[
z \ast x = q^{\frac{1}{2} \langle z, x \rangle} z \cdot x.
\]

Examples of suitable abelian categories $\ca$ are categories of representations of quivers over finite fields.
By a celebrated result of Ringel~\cite{ringel_hall}, the twisted Hall algebra $\ch(\mod \F_q Q)_{\scriptscriptstyle{\mathrm{tw}}}$
recovers the nilpotent part of the quantum group associated with $Q$ when the quiver $Q$ is of Dynkin type. This construction was instrumental in Lusztig's discovery of (dual) canonical bases \cite{Lusztig}, and plays an important role in representation theory to this day. Hall algebras of other categories have been related to automorphic forms \cite{Kapranov1997}, Donaldson--Thomas invariants \cite{KS08}, skein algebras \cite{HaidenSkeinHall}, etc.

Ringel's result and the follow-up work in the 1990s \cite{Green, Kapranov1997, Kapranov, PengXiao1997, Xiao} suggested that the full quantum group (of type ADE) should be somehow obtained from the $\Z/2$-graded derived category $\cd_2(\mod \F_q Q)$ of the quiver $Q$, also known as its \emph{root category}. Explicitly, the quantum group may be obtained from its Borel part by the \emph{reduced Drinfeld double} construction, and so it was expected that one should be able to define some Hall algebra of the root category which would recover the Drinfeld double of $\ch(\mod \F_q Q)_{\scriptscriptstyle{\mathrm{tw}}}$.
The root category is triangulated, and it thus became an important natural problem to define Hall algebras of triangulated categories more generally.

The problem was partially solved by To\"{e}n in~\cite{ToenDerivedHall}.
The input of the construction is a pre-triangulated dg-category, $\cc$, satisfying suitable finiteness conditions.
One of the key ingredients in To\"{e}n's definition is the notion of homotopy cardinality~\eqref{hcard}, applied to classifying spaces of exact triangles in $\cc$.
Concretely, the \emph{derived Hall algebra} $\cd\ch(\cc)$ of $\cc$ is then the $\mathbb Q$-algebra with basis $\{m_x\}$ parameterized by the isomorphism classes of objects in the homotopy category of $\cc$ and product
\begin{equation}\label{KShallprod}
m_z\cdot m_x=\left(\prod_{i=0}^\infty|\Ext^{-i}(z,x)|^{(-1)^{i+1}}\right)\sum_{\delta\in\Ext^1(z,x)}m_{C(\delta)},
\end{equation} 
where $C(\delta) \coloneqq \mathrm{Cone}\left(\delta:z[-1]\to x\right)$.
The algebra $\cd\ch(\cc)$ is well-defined assuming all $\Ext^i(x,y)$ are finite and trivial for $i$ sufficiently small.
(In fact, the formula \eqref{KShallprod} is taken from~\cite{KS08} and is obtained from To\"en's by rescaling the basis.) 
Assuming  all $\Ext^i(x,y)$ are finite and trivial for $|i| \gg 0$, the Euler pairing on $\cc$ is well-defined, and so one can  define a natural twisted version of $\cd\ch(\cc)$ just as in the abelian case. The definition applies for the bounded derived category of representations of a Dynkin quiver, and the resulting twisted derived Hall algebras were proved to be isomorphic to twisted Grothendieck rings of certain monoidal categories of representations of quantum affine algebras and related to graded quiver varieties \cite{HernandezLeclerc}. They also contain $\ch(\mod \F_q Q)_{\scriptscriptstyle{\mathrm{tw}}}$ as subalgebras. This suggests that twisted derived Hall algebras are the correct object to study. 

Unfortunately, \eqref{KShallprod} is ill-defined for periodic derived categories since the finiteness conditions clearly fail in this generality. Thus, a different approach was required to define (twisted) Hall algebras of root categories and recover Drinfeld doubles of twisted Hall algebras of abelian categories.
The first partial remedy was found by Bridgeland in his seminal paper \cite{bridgeland_quantum}. In this work and in further papers developing its ideas \cite{Gorsky2013, Gorsky2018, Gorsky_thesis, LuPeng, Yanagida}, an extended version  of a derived Hall algebra of a triangulated category $\cT$ has been defined under different finiteness conditions. Such extended versions are associative and make sense in some situations where the formula \eqref{KShallprod} is not well-defined. The original paper \cite{bridgeland_quantum} uses this approach to realize entire quantum groups as a version of extended Hall algebras of root categories defined via $2$-periodic complexes with projective components, and the subsequent works generalize and modify Bridgeland's construction in different directions always using certain Quillen exact categories, e.g. categories of complexes, as the starting point. In these extended Hall algebras of root categories, acyclic complexes play the role of the Cartan part of the quantum group; in particular, they form a commutative subalgebra.

As noted in \cite{BurbanSchiffmannEllipticHall}, if $\cA$ is $1$-Calabi--Yau as an abelian category --- the main example being the category of coherent sheaves on an elliptic curve --- one can avoid considering an extended version, i.e. one expects the non-extended twisted Hall algebra of the root category of $\cA$ to be well-defined, ideally  using only the structure of the triangulated root category and not any extra information depending on its abelian subcategory $\ca$. We start from this observation and extend this to a much more general class of triangulated and dg-categories.

Our main contribution is to show that when a {\degfin} pre-triangulated dg-category is assumed to be Calabi--Yau of \textit{odd} dimension, then it has a naturally defined Hall algebra. 

\begin{theorem}
\label{thm:hallintro}
Let $\cc$ be an essentially small pre-triangulated dg-category over a finite field $\F_q$ which is {\degfin} and $n$-Calabi--Yau for some odd $n$.
Then the 
$\C$-vector space with basis $\{u_x\}$ parameterized by  isomorphism classes of objects $x\in\cc$ has an associative unital algebra structure $\hallcy(\cc)$ with product
\begin{equation}\label{hallprodintro}
u_z\cdot u_x=\begin{cases} \sum_{\delta\in\Ext^1(z,x)}q^{\frac{1}{2}(\langle z,x\rangle_{1,\ldots, n-1}-r_1(\delta)-r_n(\delta))} u_{C(\delta)} & \text{for }n\geq 1 \\ 
\sum_{\delta\in\Ext^1(z,x)}q^{\frac{1}{2}(-\langle z,x\rangle_{n,\ldots, 0}-r_1(\delta)-r_n(\delta))} u_{C(\delta)} & \text{for }n\leq -1
\end{cases}
\end{equation}
where
\[
r_i(\delta)\coloneqq \rank\left(\Ext^{i-1}(x,x)\oplus\Ext^{i-1}(z,z)\xrightarrow{(a,b)\mapsto a\delta + (-1)^{i-1} \delta b
}\Ext^i(z,x)\right).
\]
\end{theorem}

Xiao--Xu \cite{XiaoXu08} proved that the product formulas for derived Hall algebras make sense and define associative algebras for arbitrary triangulated categories $\ct$ satisfying the finiteness conditions, without assuming the existence of enhancements. We note that the formulas \eqref{hallprodintro} also make sense for triangulated categories. We have not tried to prove the associativity of the Hall algebra $\hallcy(\ct)$ for an arbitrary $n$CY triangulated category $\ct$ with odd $n$, although various comparisons with other constructions defined in the triangulated setting suggest that it should hold either with our formulas or with some strict generalizations of those. We note, however, that the associativity certainly holds if we assume that $\ct$ is algebraic, i.e. admits a dg-enhancement, but the $n$CY property is required only on the triangulated level (which is a priori weaker than the assumption of Theorem \ref{thm:hallintro}). This can be proved using 3-step twisted complexes in an enhancement of $\ct$ and the associated spectral sequences (we omit the proof). 

In case of positive $n$, natural examples are given by root (or bounded derived) categories of categories of coherent sheaves \cite{BravDyckerhoffRelCY}, categories of matrix factorizations and more general singularity categories, higher generalized cluster categories,  and certain Fukaya categories. For negative $n$, the prototypical example is given by the stable category of a finite-dimensional symmetric algebra, which is (-1)-Calabi--Yau. Further examples come from dg-stable categories of non-positively graded finite-dimensional symmetric algebras, including negative cluster categories of hereditary algebras recently gaining interest in the context of positive non-crossing partitions \cite{Brightbill2020, Brightbill2021, CSPP, IyamaJin}.

The analogy with the discussion in Section~\ref{subsec: level 0} concerning twisted homotopy cardinality goes as follows. Assuming that $\Ext^i(x,y)=0$ for $|i|\gg 0$ in an oddCY {\degfin} category $\cc$, 
both the derived Hall algebra and the algebra in Theorem \ref{thm:hallintro} are well-defined, and so is the Euler pairing $\langle x, y \rangle$. By suitably rescaling the basis of the former in order to work with half-densities, we observe that these two algebras are in fact related by the twist  
by $q^{\frac{1}{2} \langle z, x \rangle}$.
We can thus indeed consider the algebra in Theorem \ref{thm:hallintro} as \emph{the} replacement of the twisted To\"en's derived Hall algebra of $\cc$ for oddCY categories which are not necessarily locally homologically finite. As discussed above, we believe the twisted derived Hall algebra --- or our replacement --- to be a more fundamental object than its untwisted counterpart.

The formula~\eqref{hallprodintro} for the Hall algebra product simplifies further if $n=1$, see Section~\ref{subsec:hallOddCY}. In this case, the product reads simply
\begin{equation}
u_z\cdot u_x=\sum_{\delta\in\Ext^1(z,x)}q^{-r_1(\delta)}u_{C(\delta)}.
\end{equation}
This includes all  $\Z/2m$-graded categories with $2m \mid n-1$, in particular all $\Z/2$-graded categories, since they can equivalently be seen as $1$CY categories.

As a special case, we consider root categories $\cd_2(\ca)$ of $n$-Calabi--Yau abelian categories. We show that our intrinsically defined algebra of the root category of $\ca$ matches a central reduction of the algebra defined in \cite{bridgeland_quantum, Gorsky2013, Gorsky_thesis, LuPeng}. Thanks to results of \cite{LuPeng, Yanagida}, this implies the following. 

\begin{theorem}
[= Corollary \ref{cor:DD}]
\label{thm:DD_intro}
Let $\ct$ be a triangulated category equivalent to the root category $\cd_2(\ca)$ of a $1$CY abelian category $\ca$. 
Then $\halluncy(\cd_2(\ca))$ is isomorphic to the Drinfeld double of the twisted Hall algebra $\ch(\ca)_{\scriptscriptstyle{\mathrm{tw}}}$ of $\ca$. 
\end{theorem}

This recovers the case of the categories of coherent sheaves on elliptic curves, where Hall algebras and their Drinfeld doubles were shown to be related to many interesting topics including double affine Hecke algebras, symmetric functions, link invariants, Fukaya  categories, etc. Theorem \ref{thm:DD_intro} establishes that $\halluncy(\cd_2(\ca))$ provides a ``global'' construction, expected since the mid-1990s, of the Drinfeld double $\ch(\ca)_{\scriptscriptstyle{\mathrm{tw}}}$ as a Hall algebra of the triangulated root category: its definition does not refer to the abelian category $\ca$ in any way, unlike the algebra of \cite{bridgeland_quantum, Gorsky2013, Gorsky_thesis, LuPeng} or the constructions in more recent works \cite{chen2023derived, zhang2022hall}. Thanks to our definition depending only on the root category, we recover (in the $1$CY case) the derived invariance of the 
Drinfeld double of $\ch(\ca)_{\scriptscriptstyle{\mathrm{tw}}}$, first proved by Cramer \cite{Cramer} by fairly \emph{ad hoc} methods.

To\"en \cite{ToenDerivedHall} proved that the Hall algebra of the heart of a $t$-structure on $\cc$ embeds as a subalgebra into the derived Hall algebra $\cd\ch(\cc)$, provided both algebras are well-defined. As our algebras play the role of twisted derived Hall algebras in the CY setting, we have a counterpart of this result for the twisted Hall algebra of the heart $\ca \subseteq H^0(\cc)$. In fact, 
a bit more holds for $\hallcy(\cc)$: The (appropriately twisted) Hall algebra of every proper abelian subcategory in the sense of \cite{Jorgensen} or, more generally, of a suitably embedded Quillen exact subcategory $\ce \subseteq H^0(\cc)$ naturally embeds as a subalgebra into $\hallcy(\cc)$.

\iftoggle{arxiv}{
Finally, we compare our intrinsically defined Hall algebras $\hallcy(\cc)$ of pretriangulated dg-categories or their homotopy categories with other constructions: a variation of derived Hall algebras valid under a certain ``odd-periodicity'' condition \cite{XuChen}; another generalization of Bridgeland's algebra, defined for Frobenius exact categories \cite{Gorsky2018}; and its variation from \cite{Gorsky2024} using  Verdier quotients of triangulated categories as an input. In all the cases, we show that in the algebraic $n$CY case, the algebra $\hallcy(\ct)$
is isomorphic either to a suitably twisted algebra in question, or to its certain central reduction. 

As an aside, we note that the construction using Verdier quotients allows one to associate an extended Hall algebra with any $n$CY triangulated category $\ct$ arising from an $(n+1)$-CY triple in the sense of \cite{IyamaYang, IyamaYang2, jin2019reductions}. For $n$ odd, one can perform a twist followed by a central reduction and recover, as mentioned above, the algebra 
$\hallcy(\ct)$. For $n$ even, we do not see any twist making such a central reduction possible, but the extended version is still interesting and admits functorial properties. As a special case, this construction associates an algebra, which can be thought of as an extended Hall algebra of the cluster category, with any Jacobi-finite quiver with potential, and this algebra is invariant under mutations and admits a braid group action via spherical twists. We expect it to be interesting to study such algebras in more detail. Further, the general construction associated with CY triples suggests that our definition involving spans could allow for a functorial generalization yielding intrinsic ``extended'' Hall algebras from relative Calabi--Yau functors. We hope to address this in future work.
} 

The previous work~\cite{HaidenSkeinHall} of the second author suggested that the intrinsic Hall algebra of the $\Z/2m$-graded Fukaya category of a surface should be well-defined precisely when that category is 1CY, which eventually led us to the above definition.
Using the formalism developed here, $\Z/2m$-graded variants of the main results of~\cite{HaidenSkeinHall,HaidenFlagsTangles} can be formulated.
Details will appear in future work.

\subsection{Level 2: 2-spans and invariants of Legendrian tangles}

At this level, (1-)spans are replaced by 2-spans, i.e. coherent diagrams of dg-categories of the following form:
\begin{equation}\label{diag:2span}
\begin{tikzcd}
     & \cb_1 \arrow[dl,"f_{11}"']\arrow[dr,"f_{21}"] & \\
    \ca_1 & \cc \arrow[u,"g_1"]\arrow[d,"g_2"] & \ca_2 \\
     & \cb_2 \arrow[ul,"f_{21}"]\arrow[ur,"f_{22}"'] &
\end{tikzcd}
\end{equation}
Such a diagram can be folded along the vertical axis to a coherent square, and Christ--Dyckerhoff--Walde~\cite{CDW_complexes} have defined CY structures on squares (more generally cubical diagrams) of dg-categories. This notion is also contained, as a special case, in earlier work of Shende--Takeda~\cite{ShendeTakeda}.
In Section~\ref{subsec:cy2spans} we discuss a variant of their definition adapted for our purposes.
We then show that the operations of vertical composition and whiskering can be performed on CY 2-spans.

As a categorification of the functor $\hdfx$ taking CY 1-spans to linear maps we define a 2-functor $\hdffx$ which linearizes 2-spans (but not 1-spans) by sending a 2-span of the form~\eqref{diag:2span} to
\[
\hdfx\left(*\leftarrow \cc\to \cb_1{\times}_{\ca_1{\times} \ca_2} \cb_2\right)(1)\in\hdfx(\cb_1{\times}_{\ca_1{\times} \ca_2} \cb_2).
\]
where the vector space on the  right-hand side is by definition the set of 2-morphisms in the linearized category.

The formalism described above is applied in the setting of Legendrian tangles and their augmentation categories.

\subsubsection{Application: Invariants of Legendrian knots and tangles}

An immersed curve in $\R^3=J^1(\R)$ is \textit{Legendrian} if it is everywhere tangent to the standard contact distribution $\xi=\mathrm{Ker}(dz-ydx)$.
A \textit{Legendrian link}, $L$, is an embedding of a disjoint union of $S^1$'s as Legendrian curves.
These are commonly studied via their \textit{front}, which is their projection to the $xz$-plane.
For generic $L$, the front looks locally like the graph of a function $z(x)$ except for two types of singularities: crossings ($\times$) and cusps (left $\prec$ and right $\succ$).
Conversely, from any diagram in the plane locally modelled on functions $z(x)$, crossings, and cusps one can recover the corresponding Legendrian link by $y=dz/dx$.
For our purposes, we can think of Legendrian links as such planar diagrams.

Invariants of Legendrian links typically depend on a choice of \textit{grading structure} on $L$, which can be defined as follows.
Let $N\geq 0$, then a $\Z/N$-grading on $L$ is a function $\mu$ from the set of \textit{strands} of $L$ (segments between cusps) to $\Z/N$ such that at any cusp, one has
\[
\mu(\text{lower strand})=\mu(\text{upper strand})+1.
\]
In particular, a $\Z/2$-grading is essentially a choice of orientation of $L$. 
For $N>2$, the obstruction for finding a $\Z/N$-grading comes from the rotation numbers of the components of $L$.

For a $\Z/N$-graded Legendrian link $L$, Chekanov and Pushkar~\cite{cp_4conj} define the \textit{$\Z/N$-graded ruling polynomial}, $P_L(z)$, which is a Laurent polynomial with non-negative integer coefficients.
We recall the definition of $P_L(z)$ in Subsection~\ref{subsec:ltangles}.

A more sophisticated invariant of a $\Z/2m$-graded Legendrian link ($N=2m$ is now even) is its Chekanov--Eliashberg dg-algebra. 
The set of augmentations of this dg-algebra is the set of objects of a naturally defined $A_\infty$-category, the \textit{augmentation category} of $L$~\cite{bourgeois_chantraine,NRSSZ}.
Following~\cite{NRSSZ}, we will recall an alternative definition of this category in terms of acyclic complexes with complete flags and denote it by $\cc_1(L)$. 
This is a $\Z/2m$-graded dg-category over the chosen base field $\mathbf k$, which we will take to be a finite field.

Important for our considerations is that the functor of $\mathbf k$-linear dg-categories
\begin{equation}\label{leglinkcyfunctor}
    \cc_1(L)\to\mathrm{Loc}_1(L)
\end{equation}
to the category $\mathrm{Loc}_1(L)$ of rank 1 local systems on $L$ carries a 2CY structure. 
This is essentially \textit{Sabloff duality} in the Legendrian knot theory literature~\cite{SabloffDuality,EES09,NRSSZ}, established in terms of CY structures on functors in upcoming work of Ma--Sabloff~\cite{MaSabloff} and (partial results) in upcoming work of Chen~\cite{ZChen}.
We give an independent proof of the existence of such a structure using gluing arguments.

\begin{theorem}\label{thm:introLinvariants}
Let $L\subset \R^3$ be a $\Z/2m$-graded Legendrian link with ruling polynomial $\rpoly_L(z)$ and augmentation category $\cc_1(L)$ over $\F_q$.
Then
\[
\rpoly_L\left(q^{\frac{1}{2}}-q^{-\frac{1}{2}}\right)=\sum_{x\in\pi_0(\cc_1(L)^\sim)}\frac{q^{\frac{1}{2}\gamma(x)}}{|\mathrm{Aut}(x)|}
\]
where
\[
\gamma(x)\coloneqq\gamma(\cc_1(L)\to\mathrm{Loc}_1(L),x) = \langle x,x\rangle_{0,1}+\mathrm{rk}\left(\Ext^0(x,x)^\vee\to\Ext^2(x,x)\right).
\]
In the case when $L$ has a single component, we have
\[
\gamma(x)=2\dim\Ext^0(x,x)-\dim\Ext^1(x,x)-1
\]
as follows from inspection of the long exact sequence of~\eqref{leglinkcyfunctor}.
\end{theorem}

This proves a conjecture of Ng--Rutherford--Shende--Sivek~\cite[Corollary 20 of Conjecture 19]{NRSS17}, who prove the above theorem in the case $m=0$.
The case $m=1$ is of particular interest, since the $\Z/2$-graded ruling polynomial is part of the HOMFLY-PT polynomial of the underlying topological knot~\cite{rutherford06}.
We will formulate and prove a generalization of Theorem~\ref{thm:introLinvariants} for Legendrian tangles, Theorem~\ref{thm:tangleFunctorIso} in the main text.
As we will show, augmentation categories of Legendrian tangles are naturally part of 2CY 2-spans. 
The strategy of the proof is then to check the equality for basic tangles (those with a single cusp or crossing) and that both sides are compatible with composition of tangles --- which uses composition of 2-spans.

The identity between the $\Z$-graded ruling polynomial and the homotopy cardinality of the augmentation category, the main result of~\cite{NRSS17}, is applied in the works~\cite{pan17,CSLMMPT} to derive obstructions to the existence of exact Maslov-0 Lagrangian cobordisms between Legendrian knots and links. 
(We thank Dan Rutherford for bringing this to our attention.) Similarly, we expect that obstructions to the existence of exact $\Z/2m$-graded Lagrangian cobordisms can be deduced from Theorem~\ref{thm:introLinvariants}.

Murray~\cite{murray23} generalizes the identity of~\cite{NRSS17} to colored ruling polynomials, which count higher rank augmentations. 
Using the ideas in this paper one can presumably formulate and prove $\Z/2m$-graded variants of this.
In another direction, one could generalize everything from knots/tangles to graphs, c.f.~\cite{ABS22}.

\subsection*{Conventions}

We use Kan complexes as a model for $\infty$-groupoids/homotopy types/spaces.
(Co)limits are by default the homotopical ones. 
In particular \textit{cartesian square} (pullback diagram) usually means \textit{homotopy cartesian square}.
The ground field is usually denoted by $\mathbf k$ and the dual vector space is denoted by $V^\vee$.
For a dg-category $\cc$ we write $\Hom_{\cc}(x,y)$ for the chain complex (with differential of degree $+1$) of morphisms from $x$ to $y$ and $\Ext^i_\cc(x,y)$ for its $i$-th cohomology space. All the Calabi-Yau structures are by default the weak ones.

\subsection*{Acknowledgements}

We thank Bernhard Keller and Vivek Shende for inspiring discussions. We are especially grateful to Maxim Kontsevich for explaining to each of us \iftoggle{arxiv}{Example \ref {ex:Kontsevich}}{his extended Hall algebra construction} during our separate stays at IHES in 2018--2019, which eventually led to this collaboration. We thank Merlin Christ and Dan Rutherford for useful comments on a preliminary draft of the paper.
FH is supported by the VILLUM FONDEN, VILLUM Investigator grant 37814 and Sapere Aude grant 3120-00076B from the
Independent Research Fund Denmark (DFF). MG was supported by the French ANR grant CHARMS (ANR-19-CE40-0017). This paper is partly a result of the ERC-SyG project Recursive and Exact New Quantum Theory (ReNewQuantum) which received funding from the European Research Council (ERC) under the European Union's Horizon 2020 research and innovation programme under grant agreement No 810573, and of the project Refined invariants in combinatorics, low-dimensional topology and geometry of moduli spaces (REFINV) that has received funding from ERC under the European Union’s Horizon 2020 research and innovation programme under grant agreement No.\ 101001159.

\section{Homotopy cardinality and Calabi--Yau structures}
\label{sec_hcardcy}

The main goal of this section is to construct the functor $\hdfx$ from the category of CY-spans in {\degfin} dg-categories over $\F_q$ to the category of vector spaces and linear maps, see Section~\ref{subsec_hcpercy}.
In Sections~\ref{subsec_htpycard} and~\ref{subsec_relcy} we review definitions and results around homotopy cardinality and CY structures on functors, respectively.

\subsection{Linearization of functors between \texorpdfstring{$\infty$}{infinity}-groupoids}
\label{subsec_htpycard}

The homotopy cardinality assigns a rational number to any sufficiently finite homotopy type. 
Going one categorical level up, we would like to assign a vector space over $\mathbb Q$ to each homotopy type and a linear map to each map (more generally: span) of homotopy types, at least under suitable finiteness hypotheses.
This was achieved by To\"en in~\cite[Section 2]{ToenDerivedHall}, see also~\cite[Section 3]{Dyckerhoff2018}, and in the first part of this subsection we recall how this is done.
In the second part of this subsection we make several observations with an eye towards a variant for dg-categories with Calabi--Yau structure.

\begin{definition}
A homotopy type, $X$, is \defword{locally finite} if for any $x\in X_0$, $\pi_i(X,x)$ is finite for $i\geq 1$ and trivial for $i\gg 0$.
Furthermore, $X$ is \defword{finite} if it is locally finite and $\pi_0(X)$ is finite.
A map $f:X\to Y$ of locally finite homotopy types is \defword{proper} if the fibers of the map $\pi_0(f):\pi_0(X)\to\pi_0(Y)$ are finite. 
\end{definition}
If $f:X\to Y$ is a map of spaces, then the long exact sequence of homotopy groups shows that the homotopy fibers of $f$ are again locally finite, and moreover finite if $f$ is proper.

Given a space, $X$, we write $\QQ^{\pi_0(X)}$ for the algebra of $\QQ$-valued functions on $\pi_0(X)$ and $\QQ\pi_0(X)$ for the algebra of \textit{finitely supported} $\QQ$-valued functions on $\pi_0(X)$.
Given a map $f:X\to Y$ of locally finite homotopy types there is a fiberwise summation map $f_!:\QQ\pi_0(X)\to\QQ\pi_0(Y)$ defined by
\begin{align*}
f_!(\alpha)(y)&\coloneqq\sum_{x\in\pi_0(X_y)}\alpha(\iota(x))\prod_{i=1}^\infty\left|\pi_i(X_y,x)\right|^{(-1)^i} 
=\sum_{\substack{x\in\pi_0(X) \\ f(x)=y}}\alpha(x)\prod_{i=1}^\infty\left(\frac{|\pi_i(X,x)|}{|\pi_i(Y,y)|}\right)^{(-1)^i}
\end{align*} 
where $X_y$ is the homotopy fiber of $f$ over $y\in Y$ and $\iota:\pi_0(X_y)\to\pi_0(X)$.
The equality follows from the long exact sequence of homotopy groups (see~\cite[Lemma 2.3]{ToenDerivedHall}). 
From the second formula it is clear that $(f\circ g)_!=f_!\circ g_!$.
There is also the map $f^*:\QQ^{\pi_0(Y)}\to\QQ^{\pi_0(X)}$ given by pullback along $\pi_0(f)$, and this restricts to a map $f^*:\QQ{\pi_0(Y)}\to\QQ{\pi_0(X)}$ if $f$ is proper. We have $(f\circ g)^* = g^* \circ f^*$.

\begin{proposition}[{\cite[Lemma 2.6]{ToenDerivedHall}}]
\label{prop:PushPull}
Suppose
\[
\begin{tikzcd}
    X \arrow[d,"f"']\arrow[r,"g"] & Y \arrow[d,"q"] \\
    Z \arrow[r,"p"'] & W 
\end{tikzcd}
\]
is a cartesian square of locally finite homotopy types with $p$ proper.
Then $g$ is proper and
\[
f_!\circ g^*=p^*\circ q_!.
\]
\end{proposition}

We make two remarks.
First, the constructions $f_!$ and $f^*$ can be combined into a single functor whose compatibility with composition becomes a reformulation of the above proposition. 
Let $\spancat$ be the category whose objects are locally finite homotopy types and morphisms from $X$ to $Z$ are diagrams $X\xleftarrow{f} Y\xrightarrow{g} Z$ where $Y$ is locally finite and $f$ is proper, up to equivalence, i.e. up to isomorphism of $Y$ in the homotopy category compatible with $f$ and $g$.
Composition in this category is the usual composition of spans (defined via pullback).
Proposition~\ref{prop:PushPull} then ensures that we have a functor from $\spancat$ to the category of vector spaces which sends $X$ to $\QQ\pi_0(X)$ and $X\xleftarrow{f} Y\xrightarrow{g} Z$ to the linear map $g_!\circ f^*$. 
In particular, this functor sends the span $*\leftarrow X\to *$ to the map $\QQ\to\QQ$ given by multiplication with the homotopy cardinality of $X$.

Second, for any locally finite homotopy type $X$ there is a canonical \textit{weighted counting measure} $\mu_X\in\QQ^{\pi_0(X)}$ given by
\[
\mu_X(x)\coloneqq\prod_{i=1}^\infty|\pi_i(X,x)|^{(-1)^i}
\]
and its corresponding inner product on $\QQ\pi_0(X)$ defined by
\[
\langle \alpha,\beta\rangle\coloneqq \sum_{x\in\pi_0(X)}\alpha(x)\beta(x)\mu_X(x).
\]
We define $\hdf(X)\coloneqq \C\pi_0(X)$ and think of its standard basis vectors as the normalizations of those in $\QQ\pi_0(X)$ considered previously, i.e. differing by a factor of $\sqrt{\mu_X(x)}$.
Elements of $\hdf(X)$ can be thought of as ``half-densities'' and the natural inner product on $\hdf(X)$ is the standard one without the factor $\mu_X(x)$. (Here, one can replace $\C$ by any field containing the quadratic extensions of $\QQ$.)
In this new basis, the induced maps $f_!:\hdf(X)\to\hdf(Y)$ and $f^*:\hdf(Y)\to\hdf(X)$ take the form
\begin{align*}
    f_!(\alpha)(y)\sqrt{\mu_Y(y)}&=\sum_{\substack{x\in\pi_0(X) \\ f(x)=y}}\alpha(x)\sqrt{\mu_X(x)} \\
    \frac{f^*(\alpha)(x)}{\sqrt{\mu_X(x)}}&=\frac{\alpha(f(x))}{\sqrt{\mu_Y(f(x))}}.
\end{align*}
\begin{definition}\label{def_lingroupoidspan}
    For a span $X\xleftarrow{f} Y\xrightarrow{g} Z$ of locally finite homotopy types where $f$ is proper, define the linear map
    \begin{equation*}
    \hdf(f,g)\coloneqq g_!\circ f^*:\hdf(X)\to\hdf(Y),
    \end{equation*}
    explicitly:
    \[
    \hdf(f,g)(\alpha)(z)=\sum_{\substack{y\in\pi_0(Y) \\ g(y)=z}}\frac{\alpha(f(y))\mu_Y(y)}{\sqrt{\mu_X(f(y))\mu_Z(z)}}.
    \]
\end{definition}
It turns out that in the setting of even dimensional Calabi--Yau spans of dg-categories, the map $\hdf(f,g)$ is still well-defined, and the normalization with square-roots is crucial, see Section~\ref{subsec_hcpercy}.

\subsection{Calabi--Yau structures on dg-functors}
\label{subsec_relcy}

In this subsection we review \textit{(relative) Calabi--Yau structures} on functors between {\degfin} dg-categories as in~\cite{keller2021introduction}. This is a slight generalization of the locally proper case, as developed by Brav--Dyckerhoff~\cite{BravDyckerhoffRelCY}. On the other hand, the cyclic invariance property plays no role in our considerations and we will therefore not impose it, which also simplifies the discussion: we work with \emph{weak right}  Calabi-Yau structures.

We fix throughout a base field $\mathbf k$.
A dg-category $\ca$ over $\mathbf k$ is \defword{\degfin} if $\Ext^i_{\ca}(x,y)$ is finite dimensional for any $x,y\in\ca$, $i\in\Z$. 
The diagonal ${\ca}^{\mathrm{op}}\otimes\ca$-module, given on objects by $(x,y)\mapsto\Hom_{\ca}(x,y)$, is denoted by $\ca$.
The $\mathbf k$-dual of the diagonal bimodule, $\ca^\vee$, is the $\left(\ca^{\mathrm{op}}\otimes \ca\right)^{\mathrm{op}}=\ca^{\mathrm{op}}\otimes \ca$-module given on objects by $(x,y)\mapsto\mathrm{Hom}_{\ca}(y,x)^\vee$.

\begin{definition}
    For $n\in\Z$, a \defword{(weak right) $n$-Calabi--Yau ($n$CY) structure}  on a {\degfin} dg-category $\ca$ is an equivalence (in the derived category of bimodules) 
    \[
    \phi:\ca^\vee\longrightarrow\ca[n].
    \]
    Thus, $\phi$ is a map of bimodules inducing isomorphisms
    \[
    \left(\Ext^{n-i}_{\ca}(y,x)\right)^\vee\longrightarrow\Ext^i_{\ca}(x,y)
    \]
    for all $x,y\in\ca$, $i\in\Z$.
\end{definition}

\begin{definition}\label{def:CYfunctor}
Suppose $f:\cb\to\ca$ is a dg-functor between {\degfin} dg-categories, and $\phi:\ca^\vee\to\ca[n-1]$ is an $(n-1)$CY structure on $\ca$ with chosen inverse up to homotopy $\phi^{-1}$.
A \defword{(weak relative right) $n$-Calabi--Yau ($n$CY) structure} on $f$ is given by a bicartesian square of $\cb^{\mathrm{op}}\otimes \cb$-modules of the form
\[
\begin{tikzcd}
    \cb \arrow[r,"f"]\arrow[d] & f^*\ca \arrow[d] \\
    0 \arrow[r] & \cb^\vee[1-n]
\end{tikzcd}
\]
where the right vertical map is the composition
\[
\begin{tikzcd}
f^*\ca \arrow[r,"f^*\phi^{-1}"] & f^*(\ca^\vee[1-n])
\arrow[r,"\sim"] & (f^*\ca)^\vee[1-n]
\arrow[r,"f^\vee"] & \cb^\vee[1-n],
\end{tikzcd}
\]
the second arrow being the canonical equivalence.
\end{definition}

Note that (bi)cartesian squares are by definition not required to be commutative, but only coherent, i.e. commute up to homotopy which is part of the data of the diagram.
For the above diagram this is the data of chain map whose boundary is the composition of the right vertical and top horizontal arrows.
If the square is bicartesian, then we get a long exact sequence
\begin{equation}
\label{relcyLES}
\begin{tikzcd}
\cdots \arrow[r] & 
\Ext^i_{\cb}(x,y) \arrow[r]
& \Ext^i_{\ca}(f(x),f(y)) \arrow[r]
\arrow[d, phantom, ""{coordinate, name=Z}]
& \left(\Ext^{n-i-1}_{\cb}(y,x)\right)^{\vee} \arrow[dll,
rounded corners,
to path={ -- ([xshift=2ex]\tikztostart.east)
|- (Z) [near end]\tikztonodes
-| ([xshift=-2ex]\tikztotarget.west)
-- (\tikztotarget)}] \\
& \Ext^{i+1}_{\cb}(x,y) \arrow{r} \arrow[r]
& \cdots 
\end{tikzcd}
\end{equation}
as in the introduction.

\begin{example}\label{ex:an1cy}
An important example for our purposes is the functor $\typeacat{n}\to \typeacat{1}^{n+1}$.
Here, $\typeacat{n}$ denotes the bounded derived category of any $A_n$-type quiver.
More concretely, let $A$ be the algebra of $n\times n$-matrices and let $B\subseteq A$ be the subalgebra of upper-triangular matrices.
Thus, $B$ is the path algebra of an $A_n$-type quiver with all arrows pointing in the same direction.
Then we have a short exact sequence of $B$-bimodules
\[
0\longrightarrow B \xrightarrow{\begin{pmatrix}\iota & \pi \end{pmatrix}} A\oplus \mathbf k^n\cong A^\vee\oplus (\mathbf k^\vee)^n\xrightarrow{\begin{pmatrix}\iota^\vee \\ \pi^\vee\end{pmatrix}} B^\vee\longrightarrow 0
\]
where $\iota:B\to A$ is the inclusion, $\pi:B\to\mathbf k^n$ is the projection to diagonal matrices, the identification $A\cong A^\vee$ comes from the pairing $(a,b)\mapsto -\mathrm{tr}(ab)$, and the identification $\mathbf k\cong\mathbf k^\vee$ is the canonical one.
This defines a 1CY structure on $(\iota,\pi)$, thought of as a dg-functor between dg-categories with one object. (The required homotopy is the zero map.)
Passing to bounded derived categories, we get a 1CY structure on the functor $\typeacat{n}\to\typeacat{1}^{n+1}$.
We observe that the connecting map in the long exact sequence~\eqref{relcyLES} vanishes in these example, since the bimodules appearing in the above short exact sequence are all concentrated in a single degree.
\end{example}

In what follows, we will extensively use the standard pasting laws, which we will now state explicitly for reference. 

\begin{lemma} \label{lem:pasting-law}
Suppose we a given a coherent diagram in an $\infty$-category $\mathcal{C}$ of the following shape:
\[
\begin{tikzcd}
A \arrow[r]\arrow[d] & B \arrow[r]\arrow[d] & C \arrow[d] \\
D \arrow[r] & E \arrow[r] & F
\end{tikzcd}
\]
\begin{enumerate}
\item If the left square is cocartesian in $\mathcal{C}$, then the right square is cocartesian if and only if the outer rectangle is a cocartesian.
\item If the right square is cartesian in $\mathcal{C}$, then the left square is cartesian if and only if the outer rectangle is cartesian.
\item Assume that $\mathcal{C}$ is stable. Then if one of the squares is bicartesian, then the other one is bicartesian if and only if so is the outer rectangle.
\end{enumerate}
\end{lemma}

\begin{proof}
Part (1) is \cite[Lemma 4.4.2.1]{Lurie_HTT} and part (2) is its dual. Part (3) follows from (1) and (2) combined with the fact that in a stable $\infty$-category (homotopy) cartesian and cocartesian squares coincide with each other and so with bicartesian squares.
\end{proof}

A \defword{CY span} is a functor of the form $f=(f_1,f_2):\cb\to\ca_1\times\ca_2$ equipped with a CY structure.
Thus, both $\ca_1$ and $\ca_2$ have an $(n-1)$CY structure and the defining bicartesian square can equivalently be written as
\[
\begin{tikzcd}
    \cb \arrow[r,"f_1"]\arrow[d,"f_2"'] & f^*\ca_1 \arrow[d] \\
    f^*\ca_2 \arrow[r] & \cb^\vee[1-n]
\end{tikzcd}
\]
where the CY structure on $\ca_2$ is now the opposite one (to make the square commute instead of anti-commute).

CY spans can be composed in the following way.

\begin{proposition}\label{prop:CYSpanComp}
Suppose
\begin{equation}\label{diag:SpanComposition}
    \begin{tikzcd}
           &       & \cb   \arrow[dl,"p_1"']\arrow[dr,"p_2"]&       & \\
           & \cb_1 \arrow[dl,"f_1"']\arrow[dr,"f_2"]&       & \cb_2 \arrow[dl,"f_3"']\arrow[dr,"f_4"]& \\
     \ca_1 &       & \ca_2 &       & \ca_3 
    \end{tikzcd}
\end{equation}
is a diagram of dg-categories and functors such that
\begin{enumerate}
    \item $\ca_1,\ca_2,\ca_3$ are equipped with $(n-1)$CY structures,
    \item the spans $\cb_1\to\ca_1\times\ca_2$ and $\cb_2\to\ca_2\times\ca_3$ are $n$CY,
    \item the square is cartesian,
\end{enumerate}
then the span $\cb\to\ca_1\times\ca_3$ is naturally $n$CY.
\end{proposition}

\begin{proof}
The proof is essentially the same as in the locally proper case, see \cite[Theorem 3.15]{christ2023relative}. 
We sketch the proof here, since we will need the diagram below later.

By assumption, we get four bicartesian squares fitting into a coherent ``window diagram'' of $\cb$-bimodules
\begin{equation}\label{diag:window}
\begin{tikzcd}
\cb \arrow[r]\arrow[d]& p_1^*\cb_1 \arrow[r]\arrow[d]& p_1^*f_1^*\ca_1 \arrow[d]\\
p_2^*\cb_2 \arrow[r]\arrow[d]& p_1^*f_2^*\ca_2 \arrow[r]\arrow[d]& p_1^*\cb_1^\vee[1-n] \arrow[d] \\
p_2^*f_4^*\ca_3 \arrow[r] & p_2^*\cb_2^\vee[1-n] \arrow[r] & \cb^\vee[1-n]
\end{tikzcd}
\end{equation}
where the lower right square is equivalent to the shifted $\mathbf k$-dual of the top left square since $\ca_2\cong \ca_2^\vee[1-n]$.
To see that the upper left square is bicartesian, we can argue as follows: After possibly replacing $\cb_1$ and $\cb_2$ by quasi-equivalent categories, we may assume that $f_1$ and $f_2$ are fibrations (with respect to the usual model structure on dg-categories) and $\cb$ is the 1-categorical pullback. Then the diagonal bimodules form a short exact sequence of complexes $0\to\cb\to p_1^*\cb_1\oplus p_2^*\cb_2\to p_1^*f_2^*\ca_2\to 0$, hence the square is bicartesian.
By Lemma \ref{lem:pasting-law}, we can conclude that the outer square is bicartesian and thus provides the CY structure on the span $\cb\to\ca_1\times\ca_3$.
\end{proof}

As a special case where $\ca_1$ and $\ca_2$ are the final dg-category we obtain:

\begin{corollary}\label{cor:CYspancomp}
Suppose $\ca$ is $(n-1)$CY and $f_i:\cb_i\to\ca$ are $n$CY functors, $i=1,2$, then $\cb_1{\times}_{\ca}\cb_2$ carries an induced $n$CY structure.
\end{corollary}

\subsection{Homotopy cardinality for evenCY dg-functors} 
\label{subsec_hcpercy}

The following construction passes from dg-categories to Kan complexes (spaces).
\begin{definition}
    Let $\cc$ be a dg-category, then we denote by $\cc^\sim$ its \textit{core}, the $\infty$-groupoid obtained by discarding all non-equivalences from $\cc$, viewed as an $(\infty,1)$-category. 
    More concretely, the Kan complex $\cc^\sim$ is obtained by taking the dg-nerve of $\cc$ (see~\cite[Section 1.3.1]{HA}) and passing to the sub-complex of equivalences.
\end{definition}
The homotopy groups of $\cc^\sim$ are given by 
\begin{equation*}
    \pi_i(\cc^\sim,x)=\begin{cases}
        \mathrm{Aut}_\cc(x) & i=1 \\
        \mathrm{Ext}^{-i+1}_\cc(x,x) & i\geq 2
    \end{cases}
\end{equation*}
which follows from~\cite[Remark 1.3.1.12]{HA}.

In the remainder of this subsection all dg-categories are linear over $\mathbf k=\F_q$ and {\degfin}. 
Note that since we are not requiring $\cc$ to be locally proper, $\cc^\sim$ is generally not locally finite, but its truncation to a 1-groupoid, $\tau_{\leq 1}\cc^\sim$, is. 

For the moment, we stay in the world of locally proper categories.
Suppose $f:\cb\to\ca$ is a functor of locally proper dg-categories such that the induced functor $f^\sim: \cb^\sim \to\ca^\sim$ is proper, i.e. the map on equivalence classes of objects, $\pi_0(\cb^\sim)\to\pi_0(\ca^\sim)$, has finite fibers. 
We apply Definition~\ref{def_lingroupoidspan} to the span $\ca^\sim\leftarrow\cb^\sim\to*$, which gives a map $\hdf(f^\sim,*):\hdf(\ca^\sim)\to\hdf(*)=\C$,
\begin{equation}\label{LagrIntZ}
\hdf(f^\sim,*)(y)=\sum_{\substack{x\in\pi_0(\cb^\sim) \\ f(x)=y}}\frac{|\mathrm{Aut}(y)|^{\frac{1}{2}}}{|\mathrm{Aut}(x)|}q^{\frac{1}{2}\langle y,y\rangle_{<0}-\langle x,x\rangle_{<0}}.
\end{equation}
The above formula does not make sense beyond the left locally homologically finite case.
Suppose however that $f$ admits a compatible $n$CY structure for some \textit{even} $n$, then the long exact sequence~\eqref{relcyLES} implies
\[
\langle f(x),f(x)\rangle_{<0}=\langle x,x\rangle_{<0}-\langle x,x\rangle_{\geq n}+\mathrm{rk}\left(\Ext^{n}(x,x)^\vee\to\Ext^{0}(x,x)\right).
\]
For $n \geq 0$, we also have
\[
-\langle x,x\rangle_{\geq n}=-\langle x,x\rangle+\langle x,x\rangle_{0,\ldots,n-1}+\langle x,x\rangle_{<0},
\]
and thus
\begin{equation}\label{rewriteForPeriodic}
\begin{split}
    \langle f(x),f(x)\rangle_{<0} -2\langle x,x\rangle_{<0}=&-\langle x,x\rangle+\langle x,x\rangle_{0,\ldots,n-1} \\
    &+\mathrm{rk}\left(\Ext^{n}(x,x)^\vee\to\Ext^{0}(x,x)\right);
\end{split}
\end{equation}
we can also rewrite the formula in a similar way for $n < 0$.
If we drop the term $\langle x,x\rangle$, as in the introduction, we can then re-write~\eqref{LagrIntZ} in a way which makes sense for {\degfin} categories.

\begin{definition}\label{def:LagrInt}
\begin{enumerate}
    \item 
    If $\cb$ is a {\degfin} dg-category which is $(n-1)$CY, define $\hdfx(\cb)\coloneqq\hdf(\cb^\sim)=\C\pi_0(\cb^\sim)$. 
    
    \item
    Suppose $f:\cb\to\ca$ is a functor of {\degfin} dg-categories which is $n$CY for some even $n$ and with $f^\sim$ proper.
    Define a functional $\hdfx(f,*):\hdfx(\ca)\to\C$
    \begin{equation}\label{LagrIntPer}
    \hdfx(f,*)(y)\coloneqq\sum_{\substack{x\in\pi_0(\cb^\sim) \\ f(x)=y}}\frac{|\mathrm{Aut}(y)|^{\frac{1}{2}}}{|\mathrm{Aut}(x)|}q^{\frac{1}{2}\gamma(f,x)} 
    \end{equation}
    where
    \begin{equation*}
        \gamma(f,x)\coloneqq \begin{cases} \langle x,x\rangle_{0,\ldots,n-1}+\mathrm{rk}\left(\Ext^n(x,x)^\vee\to\Ext^0(x,x)\right), & \mbox{for} \, n \geq 0;\\
    -\langle x,x\rangle_{n,\ldots,-1}+\mathrm{rk}\left(\Ext^n(x,x)^\vee\to\Ext^0(x,x)\right), & \mbox{for} \, n < 0
    \end{cases}
    \end{equation*}

    \item 
    More generally, suppose that $(f_1,f_2):\cb\to\ca_1\times\ca_2$ is an $n$CY span of {\degfin} dg-categories for some even $n$ and with $f_1^\sim$ proper. 
    Define a linear map $\hdfx(f_1,f_2):\hdfx(\ca_1)\to\hdfx(\ca_2)$ by
    \begin{equation}\label{CYSpanLinearization}
        \hdfx(f_1,f_2)(y_1)\coloneqq \sum_{\substack{x\in\pi_0(\cb^\sim) \\ f_1(x)=y_1}}\frac{|\mathrm{Aut}(y_1)|^{\frac{1}{2}}|\mathrm{Aut}(f_2(x))|^{\frac{1}{2}}}{|\mathrm{Aut}(x)|}q^{\frac{1}{2}\gamma((f_1,f_2),x)} f_2(x).
    \end{equation}
\end{enumerate}
\end{definition}

\begin{remark}
The long exact sequence~\eqref{relcyLES} shows that
\begin{align*}
&\mathrm{rk}\left(\Ext^n(x,x)^{\vee}  {\to}\Ext^0(x,x)\right)= \\
&\qquad\qquad\dim\Ext^0(x,x)-\mathrm{rk}\left(\Ext^0(x,x)\to\Ext^0(f(x),f(x))\right). 
\end{align*}
This means that $\gamma(f,x)$ and, therefore, $\hdfx(f,*)$ depend only on the functor $f$ and not on the choice of CY structure on $f$.
One could even define $\hdfx(f,*)$ without the assumption on the existence of a CY structure, but we do not expect this to be a good definition in general, in the sense that, for example, Theorem~\ref{thm:functoriality} fails without this assumption.
\end{remark}

\begin{remark}
The map $\hdfx(f_1,f_2)$ is a replacement for $\hdf(f,g)=g_!\circ f^*$ discussed in Section~\ref{subsec_htpycard}.
However, we do not have individual replacements for $g_!$ and $f^*$, only their composition.
In particular, we do not construct a theory with transfer in the sense of~\cite{DK_highersegal}.
\end{remark}

Let $\spancat^{n\mathrm{CY}}$ be the category whose objects are {\degfin} $(n-1)$CY categories and whose morphisms from $\ca_1$ to $\ca_2$ are spans $(f_1,f_2):\cb\to\ca_1\times\ca_2$, up to equivalence of $\cb$, admitting an $n$CY structure extending the one on $\ca_1\times\ca_2$ and such that $f_1^\sim$ is proper.
Proposition~\ref{prop:CYSpanComp} ensures that composition is well-defined in this category.

\begin{theorem}\label{thm:functoriality}
    For even $n$, $\hdfx$ is a functor from $\spancat^{n\mathrm{CY}}$ to the category of vector spaces over $\C$.
\end{theorem}

The proof of the theorem will be based on the following basic lemma.

\begin{lemma}\label{lem:DefectAdditivity}
Let $H:\cc\to\ca$ be a cohomological functor from a stable $\infty$-category $\cc$ to an abelian category $\ca$.
Suppose
\[
\begin{tikzcd}
A \arrow[r]\arrow[d] & B \arrow[r]\arrow[d] & C \arrow[d] \\
D \arrow[r] & E \arrow[r] & F
\end{tikzcd}
\]
is a coherent diagram in $\cc$ such that both squares are bicartesian (and thus so is the big rectangle) and consider the connecting maps
\[
\delta:H(F[-1])\to H(A),\qquad \delta':H(E[-1])\to H(A),\qquad \delta'':H(F[-1])\to H(B)
\]
then the sequence
\[
0\longrightarrow\Imm(\delta')\longrightarrow\Imm(\delta)\longrightarrow\Imm(\delta'')\longrightarrow 0
\]
is exact.
\end{lemma}

\begin{proof}
Follows immediately from the long exact sequences of the three bicartesian squares and the fact that both $\delta'$ and $\delta''$ factor through $\delta$.
\end{proof}

\begin{remark}
    The above lemma also makes sense and is true if $\cc$ is a triangulated category. The required pasting lemma for homotopy cartesian squares in triangulated categories is proven in~\cite{ChristensenFrankland}.
\end{remark}

Given $\cc,\ca,H$ as in Lemma~\ref{lem:DefectAdditivity} and a bicartesian square
\begin{equation}\label{bicartSquare}
\begin{tikzcd}
A \arrow[r]\arrow[d] & B \arrow[d]  \\
C \arrow[r] & D 
\end{tikzcd}
\end{equation}
we call the class of $\Imm\left(\delta:H(D[-1])\to H(A)\right)$ in $K_0(\ca)$ the \defword{defect} of the square, denoted $\delta(B{\times}_DC)$.
As a corollary of the lemma, the defect is additive with respect to pasting of homotopy cartesian squares.
In the sequel, $\ca$ will be the category $\vectcat^{\mathrm{fd}}(\mathbf k)$ of finite-dimensional vector spaces over a field $\mathbf k$, $\cc=D(\ca)$ its derived category, and $H=H^0$.
In particular, the defect is then essentially a non-negative integer after the identification $K_0(\vectcat^{\mathrm{fd}}(\mathbf k)) \overset\sim\to \mathbb{Z}$, the rank of the connecting map.

There is another --- closely related --- kind of defect of a pullback square of pointed spaces $(A,a)$, $(B,b)$, $(C,c)$, $(D,d)$, fitting into a square as above, which is given by $\left|\Imm(\pi_2(D,d)\to\pi_1(A,a)\right|$, the size of the image of the connecting map (assuming this is finite).
This is multiplicative with respect to pasting of pullback squares as in Lemma~\ref{lem:DefectAdditivity}, as can be seen by inspecting the long exact sequences. It gives a measure for the failure of pullback to commute with the truncation to 1-groupoids, $\tau_{\leq 1}$, as is made precise in the following.

\begin{lemma}\label{lem:truncPullDefect}
Given a pullback square of spaces of the form~\eqref{bicartSquare}, there is a short exact sequence of groups
\[
0\to \Imm\left(\pi_2(D,d)\to\pi_1(A,a)\right)\to\pi_1(A,a)\to\pi_1\left(A_1,a\right)\to 0
\]
where
\[
A_1\coloneqq(\tau_{\leq 1}B)\times_{\tau_{\leq 1}D}(\tau_{\leq 1}C)
\]
and we write $a$ also for the image of $a$ under the functor $A\to A_1$ coming from the universal property of the pullback.
\end{lemma}

\begin{proof}
The long exact sequences of the two pullback squares give a diagram
\[
\begin{tikzcd}
\pi_2(D,d) \arrow[r]\arrow[d] & \pi_1(A,a) \arrow[r]\arrow[d] & \pi_1(B\times C,(b,c)) \arrow[d,equal]\arrow[r] & \pi_1(D,d) \arrow[d,equal] \\
0 \arrow[r] & \pi_1(A_1,a) \arrow[r] & \pi_1(B\times C,(b,c)) \arrow[r] & \pi_1(D,d) 
\end{tikzcd}
\]
and the claim follows from the four lemma.
\end{proof}

\begin{proof}[Proof of Theorem~\ref{thm:functoriality}]
We are given a diagram of dg-categories as in~\eqref{diag:SpanComposition} and want to show that
\begin{equation}
    \hdfx(f_1\circ p_1,f_4\circ p_2)=\hdfx(f_3,f_4)\hdfx(f_1,f_2).
\end{equation}
The idea of the proof is to consider the corresponding diagram of 1-groupoids obtained by replacing a dg-category $\cc$ by the groupoid of isomorphisms in its homotopy category, $\tau_{\leq 1}\cc^\sim$.
If the square in the diagram of 1-groupoids were cartesian, and thus also a cartesian square of (truncated) spaces, then Proposition~\ref{prop:PushPull} would tell us that 
\[
\hdf(\tau_{\leq 1}(f_1\circ p_1)^\sim,\tau_{\leq 1}(f_4\circ p_2)^\sim)=\hdf(\tau_{\leq 1}f_3^\sim,\tau_{\leq 1}f_4^\sim)\hdf(\tau_{\leq 1}f_1^\sim,\tau_{\leq 1}f_2^\sim).
\]
However, by Lemma~\ref{lem:truncPullDefect}, the pullback of spaces differs from the pullback of 1-groupoids in that the latter's $\pi_1$'s are larger by the factor whose base $q$ logarithm is the defect of the bicartesian square of complexes
\[
\begin{tikzcd}
\Hom_{\cb}(y,y) \arrow[d]\arrow[r] & \Hom_{\cb_1}(y_1,y_1) \arrow[d] \\
\Hom_{\cb_2}(y_2,y_2) \arrow[r] & \Hom_{\ca_2}(x_2,x_2).
\end{tikzcd}
\]
which is 
\[
\delta(y)\coloneqq \mathrm{rk}\left(\Ext^{-1}_{\ca_2}(x_2,x_2)\to\Ext^0_{\cb}(y,y)\right)
\]
where
\[
y\in\cb,\qquad y_1=p_1(y)\in\cb_1,\qquad y_2=p_2(y)\in\cb_2,\qquad x_2=f_2(y_1)\in\ca_2.
\]
To complete the proof it thus suffices to show that $\delta(y)$ is precisely cancelled by the extra power of $q^{\frac{1}{2}}$ which appears in $\hdfx$ but not $\hdf\circ\tau_{\leq 1}$, i.e. that
\[
\gamma(y)\coloneqq\gamma((f_1\circ p_1,f_2\circ p_2),y)=\gamma((f_1,f_2),y_1)+\gamma((f_3,f_4),y_2)+2\delta(y).
\]

To show the above identity of ranks we apply Lemma~\ref{lem:DefectAdditivity} to the diagram~\eqref{diag:window}, evaluated at $(y,y)\in\cb^{\mathrm{op}}\otimes \cb$.
Let $\delta_{11},\delta_{12},\delta_{21},\delta_{22}$ be the defects of the upper-left, upper-right, lower-left, and lower-right squares respectively.
Then $\delta_{11}=\delta(y)$. Assume that $n \geq 0$, 
then the long exact sequence of the upper left square shows that
\[
\delta_{22}=\delta_{11}-\langle y,y\rangle_{0,\ldots,n-1}+\langle y_1,y_1\rangle_{0,\ldots,n-1}+\langle y_2,y_2\rangle_{0,\ldots,n-1}
\]
using the fact that $\ca_2$ is oddCY and thus $\langle x_2,x_2\rangle_{0,\ldots,n-1}=0$.
Furthermore, the defects of the upper-right, lower-left, and big squares are
$\delta_{21}=\gamma((f_1,f_2),y_1)-\langle y_1,y_1\rangle_{0,\ldots,n-1}$, $\delta_{12}=\gamma((f_3,f_4),y_2)-\langle y_2,y_2\rangle_{0,\ldots,n-1}$, and $\gamma(y)-\langle y,y\rangle_{0,\ldots,n-1}$ respectively.
Putting this all together we get
\begin{align*}
    \gamma(y) &= \langle y,y\rangle_{0,\ldots,n-1}+\delta_{11}+\delta_{12}+\delta_{21}+\delta_{22} \\
    &= 2\delta_{11}+\delta_{12}+\langle y_1,y_1\rangle_{0,\ldots,n-1}+\delta_{21}+\langle y_2,y_2\rangle_{0,\ldots,n-1} \\
    &= 2\delta(y)+\gamma((f_1,f_2),y_1)+\gamma((f_3,f_4),y_2).
\end{align*}
For $n < 0$, the proof is completed by a similar computation.
\end{proof}

\section{Hall algebras}
\label{sec:hall}

Throughout this section, all the abelian, exact, and triangulated categories are assumed to be $\Hom$-finite, essentially small, idempotent complete, and linear over $\mathbf k = \mathbb{F}_q$. This assumption implies, in particular, the Krull--Schmidt property. 

For a triangulated category $\ct$ we write $\Ext^i(x, y) := \Hom(x[-i], y)$ for $x, y \in \ct, i \in \mathbb{Z}$. We write $C(\delta)$ for the isomorphism class of the cone of a morphism $\delta$ in a triangulated or in a dg-category. 

\subsection{Reminder on Hall algebras}

Let $\ca$ be a  $\mathbf k$-linear abelian category with finite $\Hom$- and $\Ext^1$-spaces. 
The Hall algebra $\ch(\ca)$ is a  $\QQ$-vector space with a basis $\{m_x \mid x \in \Iso(\ca)\}$ parameterized by isomorphism classes of objects in $x \in \ca$ endowed with the product
\begin{align*}
m_z \cdot m_x = \sum_{y \in \Iso(\ca)} \frac{|\Ext^1(z, x)_y|}{|\Hom(z, x)|} m_y, 
\end{align*}
where $\Ext^1(z, x)_y \subseteq \Ext^1(z, x)$ is a subset of equivalence classes of short exact sequences $0 \to x \to y' \to z \to 0$ with $y' \cong y$. Denote the isomorphism class of the middle term of any short exact sequence representing an extension $\delta  \in \Ext^1(z, x)$ by $\mt(\delta)$. The multiplication can be rewritten as
\begin{align*}
m_z \cdot m_x = \frac{1}{|\Hom(z, x)|} \sum_{\delta  \in \Ext^1(z, x)} m_{\mt(\delta)}. 
\end{align*}

Hall algebras were introduced by Ringel~\cite{ringel_hall}, who also proved that $\ch(\ca)$ is an associative unital algebra, with unit given by $m_0$.

Ringel worked in a different basis, of the form $\{\frac{1}{|\aut(x)|} \cdot m_x\}$. One could interpret $m_x$'s as measures and elements of the rescaled basis as finitely supported functions on $\Iso(\ca)$. We choose to follow the conventions used in \cite{bridgeland_quantum}. 

The algebra $\ch(\ca)$ is naturally graded by the Grothendieck group $K_0(\ca)$. In fact, it is graded by the monoid version of $K_0(\ca)$ called the \emph{Grothendieck monoid} which we denote $M_0(\ca)$ \cite{BerensteinGreenstein}. Any 2-cocycle (i.e. bilinear form) on the grading group $K_0(\ca)$ with values in $\QQ$ defines a \emph{twist} of the Hall algebra. Similarly, we can extend the scalars and work over $\C$ or some intermediate field extensions of $\QQ$ and consider 2-cocycles with values in these fields. The most important example of a twist comes from the Euler pairing on the category $\ca$. Assume 
that for any $x, y \in \ca$ we have $\Ext^i(x, y) = 0$ for $i \gg 0$.  (We recall that higher extension groups in essentially small abelian categories can be defined either as $\Ext^i$-groups in the bounded derived category $\cd^b(\ca)$ where $x, y$ are considered as objects in the heart of a natural $t$-structure, or equivalently and similarly to $\Ext^1$ as equivalence classes of exact sequences of length $(i+2)$, so-called \emph{Yoneda extensions}). Then we have a pairing 
$\langle x, y \rangle \coloneqq \sum_{i=0}^{\infty} (-1)^i \dim \Ext^i(x, y)$
which descends to a $\Z$-valued bilinear form on $K_0(\ca)$. The \emph{twisted Hall algebra}  $(\ch(\ca)_{\scriptscriptstyle{\mathrm{tw}}}, \ast)$ is then defined to be the same $\C$-vector space as  $(\ch(\ca) \otimes_{\QQ} \C, \cdot)$ endowed with the multiplication 
\begin{align*}
m_z \ast m_x \coloneqq q^{\frac{1}{2} \langle z, x \rangle} m_z \cdot m_x = \frac{q^{\frac{1}{2} \langle z, x \rangle}}{|\Hom(z, x)|} \sum_{\delta  \in \Ext^1(z, x)} m_{\mt(\delta)}. 
\end{align*}
One may also consider $(\ch(\ca)_{\scriptscriptstyle{\mathrm{tw}}}, \ast)$ over $\QQ(q^{\frac{1}{2}})$. The distinction will not be important for us.

All the same definitions and properties remain valid if we work with Quillen exact categories instead of abelian categories \cite{Hubery}. These are axiomatizations of extension closed subcategories of abelian categories. 

Given a triangulated category $\ct$, the same structure constants as in the abelian case fail to produce an associative algebra. Under the following (fairly restrictive) finiteness condition, a modification is given by derived Hall algebras.

We say that a triangulated category $\ct$ is \textit{left} 
\textit{locally homologically finite} if for each pair of objects $x, y \in \ct$,
we have $\Ext^{-i}(x, y) = 0$ 
for $i \gg 0$. 
This condition
ensures the existence of an associative unital algebra associated with $\ct$ and remembering some basic homological information. Recall that whenever we write $\delta \in \Ext^1(z,x)$, we interpret it as a morphism $z[-1] \to x$.

\begin{definition}
Suppose $\ct$ is left locally homologically finite. The \textit{derived Hall algebra} $\cd\ch(\ct)$ is a  $\QQ$-vector space with a basis $\{m_x \mid x \in \Iso(\ct)\}$ parameterized by isomorphism classes of objects in $x \in \ct$, endowed with the product
\begin{align*}
m_z \cdot m_x &= \underbrace{\left(\prod_{i = 0}^{\infty} | \Ext^{-i}(z, x)|^{(-1)^{i+1}}\right)}_{q^{-\langle z, x \rangle_{\leq 0}}}
\sum_{\delta \in \Ext^1(z, x)} m_{C(\delta)}.
 \end{align*}
\end{definition}

Derived Hall algebras are associative and unital whenever they are well-defined. This is due to To\"en \cite{ToenDerivedHall} in the case of algebraic triangulated categories (i.e. those admitting a dg enhancement), and to Xiao--Xu \cite{XiaoXu08} in the general case. The structure constants we wrote are obtained from those in these papers by rescaling and were first considered by Kontsevich--Soibelman \cite{KS08}. Just as in the abelian case, the rescaling can be interpreted as an isomorphism between algebra structures on the space of finitely supported functions and on the space of measures on $\Iso(\ct)$. 
We chose the latter convention, with $m_x$'s standing for measures.

Just as Hall algebras of abelian categories, the derived Hall algebra $\cd\ch(\ct)$ is graded by the Grothendieck group $K_0(\ct)$ and can be twisted (possibly after taking extension of scalars) by any bilinear form on this grading group.

\subsection{Intrinsic Hall algebras of oddCY categories}
\label{subsec:hallOddCY}

\subsubsection{Definition and structure constants}

The following observation goes back to T\"oen~\cite{ToenDAG}.

\begin{proposition}\label{propHallSpanCY}
Suppose $\cc$ is a {\degfin} $n$CY pre-triangulated dg-category over $\mathbf k$.
Then the functor $\cc^{\typeacat{2}}\to \cc^{\times 3}$ sending a triangle $A\to B\to C\to $ to the triple $(A,B,C)$ has a natural $(n+1)$CY structure extending the $n$CY structure on $\cc^{\times 3}$.
\end{proposition}

\begin{proof}
Note that $\typeacat{2}$ is smooth and proper, thus $\cc^{\typeacat{2}}$ is Morita equivalent to $\cc\otimes\typeacat{2}$.
By the well-known fact that $\typeacat{2}\to\typeacat{1}^{\times 3}$ is $1$CY, see Example~\ref{ex:an1cy}, we are thus looking for the statement that the tensor product of an $m$CY functor of proper categories with a {\degfin} $n$CY category is an $(m+n)$CY functor.
For smooth categories and their left CY structures, such a statement is proven in~\cite[Proposition 6.4]{BravDyckerhoffRelCY}, and for either smooth or proper categories in~\cite[Proposition 3.8]{christ2023relative}.
The
same proof extends to our setting: Take the defining bicartesian square of the $m$CY structure on the functor and tensor it with the fixed $n$CY category. The resulting square is bicartesian and provides the desired $n$CY structure. 
\end{proof}

\begin{thmdef}
\label{thm:algebra_assoc_main}
Let $\cc$ be a pre-triangulated {\degfin} dg-category over $\mathbf k$ which is $n$CY for some odd $n$.
If $\mu$ is the span $\cc\times\cc\leftarrow\cc^{\typeacat{2}}\to \cc$ then $\hdfx(\mu)$ defines an associative and unital product on $\hdfx(\cc)$.
We denote the resulting algebra by $\hallcy(\cc)$. 
\end{thmdef}

\begin{remark}
Instead of requiring $\cc$ to be pre-triangulated, it suffices for $\cc$ to be closed under extensions, but not necessarily shifts, in a pre-triangulated {\degfin} n$CY$ category $\cc'$. $\hallcy(\cc)$ is then defined as the subalgebra of the Hall algebra $\hallcy(\cc')$ of $\cc'$ supported on $\hdfx(\cc)$.
This generality includes dg categories whose pre-triangulated hulls are 
{\degfin} n$CY$ for some odd $n$, as well as
those exact dg categories in the sense of \cite{Chen1, Chen2} whose bounded derived categories, as introduced in \cite{Chen3}, are 
{\degfin} n$CY$ for some odd $n$. In the latter case, we set $\cc'$ to be the bounded derived category, and the easiest way to define $\hallcy(\cc)$ is indeed by setting it to be a subalgebra of the algebra defined via the span $\cc'\times\cc'\leftarrow(\cc')^{\typeacat{2}}\to \cc'$.
\end{remark}

\begin{proof}
Consider the diagram of dg-categories
\[
\begin{tikzcd}
\cc\times\cc\times\cc & \cc\times\cc^{\typeacat{2}} \arrow[l]\arrow[r] & \cc\times\cc \\
\cc^{\typeacat{2}}\times\cc \arrow[u]\arrow[d] & \cc^{\typeacat{3}} \arrow[l]\arrow[r]\arrow[u]\arrow[d] & \cc^{\typeacat{2}} \arrow[u]\arrow[d] \\
\cc\times\cc & \cc^{\typeacat{2}} \arrow[l]\arrow[r] & \cc
\end{tikzcd}
\]
where all squares are coherent and the upper right and lower left squares are cartesian.
We conclude that both compositions of spans $\mu\circ(\cc\times\mu)$ and $\mu\circ (\mu\times\cc)$ are equivalent to the span $\cc^{\times 3}\leftarrow\cc^{\typeacat{3}}\to \cc$, and all spans are $(n+1)$CY by Proposition \ref{propHallSpanCY}, so associativity follows from Theorem~\ref{thm:functoriality}.
The unit is given by the class of the zero object in $\cc$, considered as an element of $\hdfx(\cc)=\C\pi_0(\cc^\sim)$.
\end{proof}

Our next goal is to find a more explicit formula for the Hall algebra product. It will be useful to write the formulas in different bases. 

\begin{itemize}
 \item We denote the half-densities forming our usual basis of $\hdfx(\cc)$ by $h_x$, for $x \in \pi_0(\cc^\sim)$. 
 \item We introduce a basis $\left\{a_x \mid x \in \pi_0(\cc^\sim) \right\}$ by rescaling half-densities as follows:
\[
a_x :=  \sqrt{|\aut(x)|} \cdot h_x.
\]
 \item We introduce a basis $\left\{u_x \mid x \in \pi_0(\cc^\sim) \right\}$ by rescaling half-densities as follows:
\[
u_x :=  \sqrt{|\aut(x)|} q^{-\frac{1}{2}\langle x, x \rangle_0} \cdot h_x.
\]
\item If $\cc$ is (left) locally homologically finite, following discussion in Section~\ref{subsec_htpycard}, we have a basis of finitely supported functions (either $\C$- or $\QQ$-valued) of the form 
\[
\left\{\frac{1}{\sqrt{|\aut(x)|}} \cdot q^{-\frac{1}{2}\langle x, x \rangle_{< 0}} \cdot h_x\, \middle|\, x \in \pi_0(\cc^\sim)\right\}.
\]
\item Finally, if $\cc$ is (left) locally homologically finite, the basis of measures $\{m_x\}$ is obtained from the basis of finitely supported functions by multiplying its elements by the canonical weighted counting measure $|\aut(x)| \cdot q^{\langle x, x \rangle_{< 0}} $. Its relation with the basis of the half-densities is thus as follows:
\[
m_x :=  \sqrt{|\aut(x)|} q^{\frac{1}{2}\langle x, x \rangle_{< 0}} \cdot h_x.
\]
\end{itemize}

The formulas will look the nicest in the bases $\{a_x\}$ and $\{u_x\}$, which we will now establish.
Assume first that $n > 0$.
Note that the connecting map of the long exact sequence~\eqref{relcyLES} for the CY structure on $\cc^{\typeacat{2}}\to\cc^{\times 3}$ vanishes, since the same is true for  $\typeacat{2}\to\typeacat{1}^{\times 3}$.
Thus, the product  $\hdfx(\mu)$ is given by
\[
h_z\cdot h_x=\sum_{\substack{t=(x\to y\to z \to) \\ \in\pi_0((\cc^{\typeacat{2}})^\sim)}}\frac{\sqrt{|\aut(x)||\aut(y)||\aut(z)|}}{|\aut(t)|}q^{\frac{1}{2}\langle t, t\rangle_{0,\ldots,n}} h_y.
\]

In the basis $\{a_x\}$,
the above becomes
\begin{equation}\label{newHallProd_rescaled}
a_z\cdot a_x=\sum_{\substack{t=(x\to y\to z \to )\\ \in\pi_0((\cc^{\typeacat{2}})^\sim)}}\frac{|\aut(x)||\aut(z)|}{|\aut(t)|}q^{\frac{1}{2}\langle t, t\rangle_{0,\ldots,n}} 
a_y.    
\end{equation}
We will consider the ranks
\[
r_i(\delta)\coloneqq \rank\left(\Ext^{i-1}(z,z)\oplus\Ext^{i-1}(x,x)\xrightarrow{\begin{pmatrix}L_{\delta} & (-1)^{i-1} R_{\delta} 
\end{pmatrix}
}\Ext^{i}(z,x)\right)
\]
where $L_\delta$, $R_\delta$ are left and right multiplication by $\delta\in\Hom^1(z,x)$, respectively.

The long exact sequence of the fiber sequence of the map of Kan complexes $(\cc^{\typeacat{2}})^\sim\to (\cc\times\cc)^\sim$, $t\mapsto (x,z)$ which ends with
\[
\begin{tikzcd}
& \cdots \arrow[r] \arrow[d, phantom, ""{coordinate, name=Z1}] & \Ext^{-1}(x,x)\oplus\Ext^{-1}(z,z) \arrow[dll,
rounded corners,
to path={ -- ([xshift=2ex]\tikztostart.east)
|- (Z1) [near end]\tikztonodes
-| ([xshift=-2ex]\tikztotarget.west)
-- (\tikztotarget)}] \\
\Ext^0(z,x) \arrow[r] & \aut(t) \arrow[r] \arrow[d, phantom, ""{coordinate, name=Z2}] & \aut(x)\times\aut(z) \arrow[dll,
rounded corners,
to path={ -- ([xshift=2ex]\tikztostart.east)
|- (Z2) [near end]\tikztonodes
-| ([xshift=-2ex]\tikztotarget.west)
-- (\tikztotarget)}] \\
\Ext^1(z,x) \arrow{r} & \pi_0((\cc^{\typeacat{2}})^\sim) \arrow[r] & \pi_0((\cc\times\cc)^\sim)
\end{tikzcd}
\]
shows that
\[
\left|\left\{t=(x\to y\to z \to ) \in\pi_0((\cc^{\typeacat{2}})^\sim)\right\}\right|=\sum_{\delta\in\Ext^1(z,x)_y}\frac{|\aut(t)|}{|\aut(x)||\aut(z)|}q^{-\langle z,x\rangle_0+r_0(\delta)},
\]
where $x,y,z$ are fixed, $t=(x\to y\to z\xrightarrow{\delta} )$, and $\Ext^1(z,x)_y\subseteq \Ext^1(z,x)$ is the subset of those extension classes giving $y$.

To compute $\Ext^i(t,t)$ and $\Ext^i(y,y)$ we will model $t$ and $y$ by the 2-step (twisted) complex $z\xrightarrow{\delta}x$. 
Then $\Hom^\bullet(y,y)$ has underlying graded vector space
\[
\Hom^\bullet(z,x)\oplus\Hom^\bullet(z,z)\oplus\Hom^\bullet(x,x)\oplus\Hom^\bullet(x,z)
\]
with differential $\Hom^{i-1}(y,y)\to\Hom^i(y,y)$ given by
\[
\left(\begin{matrix}  
d & L_{\delta} & (-1)^{i-1} R_{\delta} & 0 \\
0 & d & 0 & (-1)^{i-1} R_{\delta} \\ 
0 & 0 & d & L_{\delta} \\
0 & 0 & 0 & d \end{matrix}  
\right)
\]
where $d$ is the differential in $\cc$ and $L_\delta$, $R_\delta$ and left and right multiplication by $\delta\in\Hom^1(z,x)$, respectively.

Note that the differential is compatible with the 3-step filtration
\[
\Hom^\bullet(z,x)\subset \Hom^\bullet(z,x)\oplus\Hom^\bullet(z,z)\oplus\Hom^\bullet(x,x)\subset \Hom^\bullet(y,y).
\]
Consider the spectral sequence $E_\bullet$ associated with this filtered complex. 
Its $E_1$ page is given by
\[
\Ext^\bullet(z,x)\oplus\Ext^\bullet(z,z)\oplus\Ext^\bullet(x,x)\oplus\Ext^\bullet(x,z)
\]
with differential
\[
D^i\coloneqq\left(\begin{matrix}  
0 & L_{\delta} & (-1)^{i-1} R_{\delta} & 0 \\
0 & 0 & 0 & (-1)^{i-1} R_{\delta} \\ 
0 & 0 & 0 & L_{\delta} \\
0 & 0 & 0 & 0 \end{matrix}  
\right)
\]
where now $L_{\delta}$ and $R_{\delta}$ are the induced maps on $\Ext$-groups.
Since the upper-right corner of this matrix vanishes, the spectral sequence converges at the $E_2=E_\infty$ page which is $\Ext^\bullet(y,y)$.
Thus, if we let $d_i(\delta)\coloneqq\mathrm{rk}D^i$, then 
\begin{equation}
\label{dim_ext_and_4x4matrix}
\dim\Ext^i(x \oplus z,  x \oplus z)  = \dim\Ext^i(y, y)  + d_i(\delta) + d_{i+1}(\delta).
\end{equation}

Since $\cc$ is $n$CY by assumption, 
the map
\[
\Ext^{i-1}(x,z) \xrightarrow{\begin{pmatrix}(-1)^{i-1} R_{\delta} \\
L_{\delta}
\end{pmatrix}
}
\Ext^{i}(z,z) 
\oplus 
\Ext^i(x,x)
\]
is dual to the map 
 \[
 \Ext^{n-i}(z,z)\oplus\Ext^{n-i}(x,x)\xrightarrow{\begin{pmatrix}L_{\delta} & (-1)^{n-i} R_{\delta} 
\end{pmatrix}
}\Ext^{n-i+1}(z,x),
\]
and so it has the same rank $r_{n-i+1}$.
We then have
\[
d_i(\delta) = r_i(\delta) + r_{n-i+1}(\delta).
\]

To compute $\Ext^\bullet(t,t)$ we start instead with the subcomplex
\[
\Hom^\bullet(z,x)\oplus\Hom^\bullet(z,z)\oplus\Hom^\bullet(x,x)
\]
and a similar argument then shows that 
\[
\dim\Ext^i(x,x)+\dim\Ext^i(z,z)+\dim\Ext^i(z,x)=\dim\Ext^i(t,t)+r_i(\delta)+r_{i+1}(\delta).
\]
In particular, since $n$ is odd and thus $\langle x,x\rangle_{0,\ldots,n}=\langle z,z\rangle_{0,\ldots,n}=0$, we get
\[
\langle t,t\rangle_{0,\ldots,n} = \langle z,x\rangle_{0,\ldots,n}-r_0(\delta)+r_{n+1}(\delta).
\]

Plugging all this into \eqref{newHallProd_rescaled}, we get

\begin{align*}
a_z\cdot a_x &= \sum_{\delta\in\Ext^1(z,x)}q^{\frac{1}{2}\langle z,x\rangle_{0,\ldots,n}-\langle z,x\rangle_0+\frac{1}{2}r_0(\delta)+\frac{1}{2}r_{n+1}(\delta)} a_{C(\delta)}\\
&= \sum_{\delta\in\Ext^1(z,x)}q^{\frac{1}{2}\langle z,x\rangle_{0,\ldots,n}-\langle z,x\rangle_0+\frac{1}{2}d_0(\delta)} a_{C(\delta)}.
\end{align*}

To summarize, we have proven the following.

\begin{proposition}
\label{prop:prod_intrinsic_positive}
Let $\cc$ be a pre-triangulated {\degfin} dg-category over $\mathbf k$ which is $n$CY for some odd $n > 0$. In the basis $\left\{a_x \mid x \in \pi_0(\cc^\sim) \right\}$ 
the product  
in $\hallcy(\cc)$
is given by the formula
\begin{align*}
a_z\cdot a_x = \sum_{\delta\in\Ext^1(z,x)}q^{\frac{1}{2}(\langle z,x\rangle_{0,\ldots, n}+r_0(\delta)+r_{n+1}(\delta)) - \langle z,x\rangle_{0}}a_{C(\delta)}.
\end{align*}
\end{proposition}

In the basis $\{u_x\}$, the multiplication formula is a little bit simpler. Indeed, we have
\begin{align*}
u_z\cdot u_x &= \sum_{\delta\in\Ext^1(z,x)}q^{\frac{1}{2}(\langle z,x\rangle_{0,\ldots, n}+r_0(\delta)+r_{n+1}(\delta)) - \langle z,x\rangle_{0} - \frac{1}{2}(\langle x,x\rangle_0+\langle z,z\rangle_0-\langle y,y\rangle_0)}u_{C(\delta)}\\
&= \sum_{\delta\in\Ext^1(z,x)}q^{\frac{1}{2}\langle z,x\rangle_{1,\ldots,n-1}+\frac{1}{2}d_0(\delta) - \frac{1}{2}\langle x \oplus z, x \oplus z\rangle_0 + \frac{1}{2} \langle y,y\rangle_0}u_{C(\delta)}\\
&= \sum_{\delta\in\Ext^1(z,x)}q^{\frac{1}{2}\langle z,x\rangle_{1,\ldots,n-1}-\frac{1}{2}d_1(\delta)}u_{C(\delta)}\\
&=\sum_{\delta\in\Ext^1(z,x)}q^{\frac{1}{2}(\langle z,x\rangle_{1,\ldots, n-1}-r_1(\delta)-r_n(\delta))}u_{C(\delta)},
\end{align*}
where we used that $-\langle x,z\rangle_{0} = \langle z,x\rangle_{n}$  
and the formula \eqref{dim_ext_and_4x4matrix}.
Here $\langle z,x\rangle_{1,\ldots, n-1}$ is by definition identically zero if $n = 1$. We proved the following.

\begin{proposition}
\label{prop:prod_intrinsic_positive_2}
Let $\cc$ be a pre-triangulated {\degfin} dg-category over $\mathbf k$ which is $n$CY for some odd $n > 0$. In the basis $\left\{u_x \mid x \in \pi_0(\cc^\sim) \right\}$ 
the product 
in $\hallcy(\cc)$
is given by the formula
\begin{align*}
u_z\cdot u_x = \sum_{\delta\in\Ext^1(z,x)}q^{\frac{1}{2}(\langle z,x\rangle_{1,\ldots, n-1}-r_1(\delta)-r_{n}(\delta))}u_{C(\delta)}.
\end{align*}
\end{proposition}

For $n < 0$, similar calculations show that the multiplication can be rewritten as follows.
\begin{proposition}
\label{prop:prod_intrinsic_negative}
Let $\cc$ be a pre-triangulated {\degfin} dg-category over $\mathbf k$ which is $n$CY for some odd $n < 0$. 
In the bases $\left\{[a_x] \mid x \in \pi_0(\cc^\sim) \right\}$ and $\left\{[u_x] \mid x \in \pi_0(\cc^\sim) \right\}$,
the product 
in $\hallcy(\cc)$
is given by the formulas
\begin{align*}
a_z\cdot a_x &= \sum_{\delta\in\Ext^1(z,x)}q^{\frac{1}{2}(-\langle z,x\rangle_{n+1,\ldots, -1}+r_0(\delta)+r_{n+1}(\delta)) - \langle z,x\rangle_{0}}a_{C(\delta)};\\
u_z\cdot u_x &= \sum_{\delta\in\Ext^1(z,x)}q^{\frac{1}{2}(-\langle z,x\rangle_{n,\ldots, 0}-r_1(\delta)-r_{n}(\delta))}u_{C(\delta)}.
\end{align*}
\end{proposition}

\begin{remark}
\label{rem:choice_of_dim_in_periodic_case}
We note that if the category $\cc$ is $\Z/2m$-graded, we can change the perspective slightly and consider it not as an $n$CY category, but as an $(n \pm 2m)$CY category, or as an $(n \pm 4m)$CY, etc. The multiplication formula depends on such a choice of the CY dimension, but the formulas for all the different choices are related by twists by $q^{\frac{k}{2}\langle z,x\rangle_{\Z/2m}}$, for $k \in \Z$.
\end{remark}

The general formulas have further simplifications in various cases.
\begin{itemize}
\item For $n=1$, the formula in the basis $\{u_x\}$ simplifies as follows:
\begin{align*}
u_z\cdot u_x  =\sum_{\delta\in\Ext^1(z,x)}q^{-\frac{1}{2} d_1(\delta)} u_{C(\delta)}=\sum_{\delta\in\Ext^1(z,x)}q^{- r_1(\delta)}u_{C(\delta)}.
\end{align*}
\item In the case of a $\Z/2$-graded category with $n=1$, the general formula in the basis $\{a_x\}$ simplifies to
\[
a_z\cdot a_x =\sum_{\delta\in\Ext^1(z,x)}q^{\frac{1}{2}\langle z,x\rangle_{0,1}-\langle z,x\rangle_0 +r_0(\delta)}u_{C(\delta)}.
\]
\item In the case of a $\Z/2m$-graded category with $2m \mid n-1$ and $n > 0$, the general formula simplifies to
\[
u_z\cdot u_x =\sum_{\delta\in\Ext^1(z,x)}q^{\frac{n-1}{4m}\langle z,x\rangle_{\Z/2m}-r_1(\delta)}u_{C(\delta)}.
\]
This generality includes all $\Z/2$-graded $n$CY pre-triangulated {\degfin} dg-categories with $n >0$.

Moreover, the form $\langle z,x\rangle_{\Z/2m} = \langle z,x\rangle_{0,\ldots, 2m-1}$ descends onto the Grothendieck group. Thus, we can twist the algebra by setting 
\[
u_z \ast u_x = q^{-\frac{n-1}{4m}\langle z,x\rangle_{\Z/2m}} u_z \cdot u_x.
\]
In the twisted algebra we have 
\begin{align*}
u_z \ast u_x  =\sum_{\delta\in\Ext^1(z,x)}q^{-r_1(\delta)}u_{C(\delta)}.
\end{align*}  
In fact, we can re-interpret such a twisted algebra as $\halluncy(\cc)$, since $\cc$ can alternatively be seen as a $1$CY category, as discussed in Remark \ref{rem:choice_of_dim_in_periodic_case}.

\item When $\cc$ is $\Z/2m$-graded with $2m \mid n + 1$ and $n < 0$, the same simplification and twist as in case of  $2m \mid n -1$ and $n > 0$ apply. In particular, in the algebra $\hallminusuncy(\cc)$ obtained from $\hallcy(\cc)$ by a natural twist, we have 
\begin{align*}
u_z \ast u_x  =\sum_{\delta\in\Ext^1(z,x)}q^{-r_1(\delta)}u_{C(\delta)}.
\end{align*}
\end{itemize}

\subsubsection{Functoriality}
The multiplication formulas in both Propositions~\ref{prop:prod_intrinsic_positive} and \ref{prop:prod_intrinsic_negative} depend only on the homotopy category $H^0(\cc)$. In other words, the algebra $\hallcy(\cc)$ is an invariant of the triangulated category $\ct := H^0(\cc)$ and does not depend of the choice of its $n$CY enhancement, though our proof of associativity relies on the existence of such an enhancement. To emphasize this, we will also write  $\hallcy(\ct)$ meaning the same algebra $\hallcy(\cc)$. In particular, we have the following immediate corollary of the multiplication formulas.

\begin{corollary}
A fully-faithful functor $F$ of 
triangulated categories $\ct$, $\ct'$ admitting enhancements by pre-triangulated {\degfin} dg-categories $\cc$, $\cc'$ over $\mathbf k$ of the same odd CY dimension $n$ induces an injective algebra homomorphism $\hallcy(\ct) \hookrightarrow \hallcy(\ct')$. If $F$ is an equivalence, this homomorphism is moreover an isomorphism.
In particular, a quasi-fully faithful functor (resp. a quasi-equivalence) of pre-triangulated {\degfin} dg-categories $\cc \to  \cc'$ over $\mathbf k$ of the same odd CY dimension $n$ induces an injective algebra homomorphism (resp. an algebra isomorphism) $\hallcy(\cc) \hookrightarrow \hallcy(\cc')$.
\end{corollary}

\subsubsection{Comparison with derived Hall algebras}

We now compare the algebra $\hallcy(\cc)$ with the derived Hall algebra of $\cc$, assuming they both are well-defined. This happens for example in the case when $\cc$ is generated by a spherical object. In this case, the derived Hall algebra $\cd\ch(\cc)$ was computed in \cite{KYZ}.

Let $\cc$ be a pre-triangulated {\degfin} dg-category over $\mathbf k$ which is $n$CY for some odd $n$ and such that for all $x, y \in \cc$, we have $\Ext^i(x,y)=0$ for $|i| \gg 0$, i.e. $\cc$ is locally proper. Note that the CY condition implies that the 
local properness
is equivalent to left local homological finiteness. The Euler pairing $\langle x, y \rangle = \langle x, y \rangle_{\Z}$ is well-defined on $\cc$. We can rewrite it as
\begin{equation}
\label{euler_pos}
 \langle z, x \rangle = \langle z,x\rangle_{< 0} - \langle x,z\rangle_{< 0} + \langle z,x\rangle_{0,\ldots,n}
\end{equation}
if $n > 0$
and 
\begin{equation}
\label{euler_neg}
 \langle z, x \rangle = \langle z,x\rangle_{<0} - \langle x,z\rangle_{<0} - \langle z,x\rangle_{n+1,\ldots,-1}
\end{equation}
if  $n < 0$.
Using these formulas, we can prove the following.

\begin{proposition}
\label{prop:derived_hall_comparison}
Let $\cc$ be a pre-triangulated  
locally proper dg-category over $\mathbf k$ which is $n$CY for some odd $n$.  
Then the algebra
$\hallcy(\cc)$ coincides with the twisted tensor product $(\cd\ch(\cc) \otimes_\mathbb{Q} \mathbb{C})_{\mathrm{tw}}$,
where the twist is given by $q^{\frac{1}{2}\langle z, x \rangle}$.
\end{proposition}

\begin{proof}
We go backwards and twist the multiplication in the intrinsic Hall algebra by $q^{-\frac{1}{2}\langle z, x \rangle}$. Using the above formulas \eqref{euler_pos} and \eqref{euler_neg}, we see that, for any odd $n \in \Z$, in the twisted algebra the multiplication reads 
\begin{equation}
\label{twisted_intrinsic}
a_z \ast a_x = \sum_{\delta\in\Ext^1(z,x)}q^{\frac{1}{2}\left(-\langle z,x\rangle_{< 0} +
\langle x,z\rangle_{<0} + d_0(\delta)\right) - \langle z,x\rangle_0}a_{C(\delta)}.
\end{equation}
Recall the formula \eqref{dim_ext_and_4x4matrix}. By taking the alternating sum of the expressions on both sides over all $i < 0$ (which is finite by the local properness assumption), we get
\[
\langle x \oplus z, x \oplus z \rangle_{< 0} = \langle y, y \rangle_{< 0} - d_0(\delta).
\]
Therefore, we can rewrite the twisted multiplication as follows:
\begin{equation}
\begin{aligned}
\label{eq:intermediate_calculation}
a_z \ast a_x &= \sum_{\delta\in\Ext^1(z,x)}q^{-\langle z,x\rangle_{\leq 0} +\frac{1}{2}\left(\langle z,x\rangle_{< 0} +
\langle x,z\rangle_{< 0} + d_0(\delta)\right)}a_{C(\delta)}\\
&=\sum_{\delta\in\Ext^1(z,x)}q^{-\langle z,x\rangle_{\leq 0}  - \frac{1}{2}(\langle x,x\rangle_{< 0}+\langle z,z\rangle_{< 0}-\langle y,y\rangle_{< 0})}a_{C(\delta)}.
\end{aligned}
\end{equation}

Recall that the basis $\{a_x\}$ is obtained from the basis of the half-densities by the rescaling
\[
a_x =  \sqrt{|\aut(x)|} \cdot  h_x.
\]
In other words, it is obtained from the basis of the measures $\{m_x\}$ by the rescaling 
\[
a_x =  q^{-\frac{1}{2}\langle x, x \rangle_{< 0}} m_x.
\]
Rescaling back to the basis of measures, we see that the formula \eqref{eq:intermediate_calculation} becomes
\[m_z \ast m_x = \sum_{\delta\in\Ext^1(z,x)}q^{-\langle z,x\rangle_{\leq 0}} m_{C(\delta)},
\]
which is precisely the formula in the derived Hall algebra. Since the twist is defined on the grading group, we can twist back and obtain the algebra $\hallcy(\cc)$, just in a rescaled basis.
This completes the proof. 
\end{proof}

\iftoggle{arxiv}{
For completeness, we also compare $\hallcy(\cc) = \hallcy(\ct)$ with a version of $\cd\ch(\ct)$ introduced in \cite{XuChen}, see also \cite{ShenChenXu}. There, it is defined for $\Hom$-finite idempotent complete triangulated categories $\ct$ which have an odd period $t > 1$, i.e. $[t] \cong \mbox{Id}$. The same definition and the same proof of associativity apply if this periodicity assumption is weakened by the following assumption: 
For every morphism $x \overset{\alpha}\to y$ and every $z$ in $\ct$, we have 
\begin{align}
\label{eq:odd_periodic1}
\rank \left(\Ext^0(y, z) \overset{\circ \alpha}\to \Ext^0(x, z)\right) &= \rank\left(\Ext^0(y, z[t]) \overset{\circ \alpha}\to \Ext^0(x, z[t])\right);\\
\label{eq:odd_periodic2}
\rank \left(\Ext^0(z, x) \overset{\alpha\circ}\to \Ext^0(z, y)\right) &= \rank\left(\Ext^0(z[t], x) \overset{\alpha\circ}\to \Ext^0(z[t], y)\right).
\end{align}
This assumption holds if the category is only ``odd periodic up to sign'', but the actual period is even.

Under this assumption, the algebra $\cd\ch'(\ct)$ can be defined over $\C$ in a basis $\{m'_x \mid x \in \Iso(\ct)\}$, playing the role of the basis of measures, with the product
\[
m'_z \cdot m'_x \coloneqq q^{\frac{1}{2}\langle z, x \rangle_{1,\ldots,t}}
\sum_{\delta \in \Ext^1(z, x)} m'_{C(\delta)}.
\]

Similarly to the proof of Proposition \ref{prop:derived_hall_comparison}, we can show the following. We omit the details.

\begin{proposition}
Let $\cc$ be a pre-triangulated {\degfin}  dg-category over $\mathbf k$ which is $n$CY for some odd $n$ and satisfies assumption the conditions on ranks \eqref{eq:odd_periodic1} and \eqref{eq:odd_periodic2}. Then the algebra
$\hallcy(\cc)$ is isomorphic to the twist of the algebra $\cd\ch'(\ct)$ by the form 
\begin{equation}
\begin{cases}
q^{\frac{1}{4}(-\langle z, x \rangle_{1,\ldots,t} + \langle x, z \rangle_{1,\ldots,t}) +\frac{1}{2}\langle z, x \rangle_{1,\ldots,n-1}} & \text{for }  n \geq 1;\\
q^{\frac{1}{4}(-\langle z, x \rangle_{1,\ldots,t} + \langle x, z \rangle_{1,\ldots,t}) - \frac{1}{2}\langle z, x \rangle_{n,\ldots,0}} & \text{for }  n \leq -1.
\end{cases}
\end{equation}
\end{proposition}
} 

\subsubsection{Exact subcategories and subalgebras}

\begin{lemma}
\label{lem:formula_with_vanishing_d_0}
Let $\cc$ be a pre-triangulated {\degfin} dg-category over $\mathbf k$ which is $n$CY for some odd $n$. Assume that for a pair of isomorphism classes of objects $z, x \in \pi_0(\cc^\sim)$, we have $d_0(\delta) = 0,$ for all $\delta \in \Ext^1(z, x)$.  
Then in the basis
$\left\{a_x \mid x \in \pi_0(\cc^\sim) \right\}$ 
the product in $\hallcy(\cc)$ is given by the formula 
\begin{equation}
a_z \cdot a_x = \begin{cases}
\sum_{\delta\in\Ext^1(z,x)}q^{\frac{1}{2}\langle z,x\rangle_{0,\ldots,n}-\langle z,x\rangle_{0}}a_{C(\delta)} & \text{for }n\geq 1; \\
\sum_{\delta\in\Ext^1(z,x)}q^{-\frac{1}{2}\langle z,x\rangle_{n+1,\ldots,-1}-\langle z,x\rangle_{0}}a_{C(\delta)}  & \text{for }n\leq -1.
\end{cases}
\end{equation}
\end{lemma}

\begin{proof}
This follows immediately from the formulas in Propositions \ref{prop:prod_intrinsic_positive} and \ref{prop:prod_intrinsic_negative}.
\end{proof}

Let $\ce$ be a full additive subcategory of $\ct = H^0(\cc)$. We say that it is \emph{extension-closed} in $\ct$ if for every triangle $x \overset{\beta}\to y \overset{\gamma}\to z \overset{\delta[1]}\to x[1]$, we have $y \in \ce$ whenever both $x \in \ce$ and $z \in \ce$. It is known that $\ce$ is extension-closed in $\ct$ and moreover in every such triangle $\beta$ is monic and $\gamma$ is epic if and only if the class of exact sequences  $x \overset{\beta}\to y \overset{\gamma} \to z$ defines a (Quillen) exact structure on $\ce$ (see \cite[Remark 2.18, Corollary 3.18, Proposition 3.22]{NakaokaPalu}). If these equivalent conditions hold, we say that the embedding $\ce \hookrightarrow \ct$ \emph{induces an exact structure on $\ce$}. 
We note that these equivalent conditions imply that for every such triangle, the maps $\Hom(z, t) \to \Hom(y, t)$ and $\Hom(t, x) \to \Hom(t, y)$ are injective for every $t \in \ce$. In particular, for $x, z, t \in \ce$ and every $\delta \in \Ext^1(z, x)$, both maps $\Ext^{-1}(t, z) \overset{L_{\delta}}\to \Ext^0(t, x)$ and $\Ext^{-1}(x, t) \overset{R_{\delta}}\to \Ext^0(z, t)$ are zero. By the $n$CY property, we also have that both maps $\Ext^{n}(x, t) \overset{R_{\delta}}\to \Ext^{n+1}(z, t)$ and $\Ext^{n}(t, z) \overset{L_{\delta}}\to \Ext^0(t, x)$ are zero. Combining this all together, we find that, in particular, in such exact categories, for $x, z \in \ce$ and every $\delta \in \Ext^1(z, x)$, we have $d_0(\delta) = 0$. We are now ready to show the following.

\begin{theorem}
\label{thm:exact_embedding}
Let $\cc$ be a pre-triangulated {\degfin} dg-category over $\mathbf k$ which is $n$CY for some odd $n$. Assume that $\ce$ is a full additive subcategory of $\ct = H^0(\cc)$ such that the embedding $\ce \hookrightarrow \ct$ induces an exact structure on $\ce$. Then this embedding induces an embedding of Hall algebras \[
(\ch(\ce)_{\scriptscriptstyle{\mathrm{tw}}}, \diamond) \hookrightarrow \hallcy(\ct), \,\,\, m_z \mapsto a_z,
\]
where $(\ch(\ce)_{\scriptscriptstyle{\mathrm{tw}}}, \diamond)$ is the twist of $(\ch(\ce) \otimes_{\QQ} \C, \cdot)$ given by
\begin{equation*}
m_z \diamond m_x \coloneqq 
\begin{cases}
q^{\frac{1}{2}\langle z, x \rangle_{0,\ldots,n}} m_z \cdot m_x  & \text{for }n\geq 1; \\
q^{-\frac{1}{2}\langle z, x \rangle_{n+1,\ldots,-1}} m_z \cdot m_x  & \text{for }n\leq -1.
\end{cases}
\end{equation*}
\end{theorem}

We first note that the vanishing of (ranks of) all the 4 connecting morphisms (to $\Ext^0$ and from $\Ext^n$) $L_{\delta}, R_{\delta}$ for all $x, y, t \in \ce$ and $\delta \in \Ext^1(z, x)$ implies that $q^{\frac{1}{2}\langle -, ? \rangle_{0,\ldots,n}}$ is indeed a bilinear form on $K_0(\ce)$, and so $(\ch(\ce)_{\scriptscriptstyle{\mathrm{tw}}}, \diamond)$ is a well-defined associative algebra.

\begin{proof}
Thanks to the vanishing of $d_0(\delta)$ for $x, z \in \ce$ and every $\delta \in \Ext^1(z, x)$, this follows immediately from comparing the definition of $(\ch(\ce)_{\scriptscriptstyle{\mathrm{tw}}}, \diamond)$ with formulas in Lemma \ref{lem:formula_with_vanishing_d_0}. 
\end{proof}

\begin{corollary}
\label{cor:embedding_same_twist}
Assume that we are in the setting of Theorem \ref{thm:exact_embedding} and moreover $n \geq 1$, and we have \begin{align*}
\Ext^i_{\ce}(-,?) &\cong \Ext^i_{\ct}(-,?) & \text{for } 0 \leq i \leq n;\\
\Ext^i_{\ce}(-,?) &= 0 & \text{for } i > n.
\end{align*}
Then the embedding $\ce \hookrightarrow \ct$ induces an embedding of Hall algebras \[
(\ch(\ce)_{\scriptscriptstyle{\mathrm{tw}}}, \ast) \hookrightarrow \hallcy(\ct), \,\,\, m_z \mapsto a_z.
\]
\end{corollary}

\begin{proof}
This follows from Theorem \ref{thm:exact_embedding} by noting that, by the assumption, for all $z, x \in \ce$ we have $\frac{1}{2}\langle z, x \rangle_{0,\ldots,n} = \frac{1}{2}\langle z, x \rangle_{\ce}$, and so $(\ch(\ce)_{\scriptscriptstyle{\mathrm{tw}}}, \ast) = (\ch(\ce)_{\scriptscriptstyle{\mathrm{tw}}}, \diamond)$.
\end{proof}

Following J{\o}rgensen \cite{Jorgensen}, we say that a full additive subcategory $\ca \subseteq \ct$ is a \emph{proper abelian subcategory} if it is abelian and a sequence $0 \to x \overset{\beta}\to y \overset{\gamma}\to z \to 0$ is short exact in $\ca$ if and only if there exists a triangle $x \overset{\beta}\to y \overset{\gamma}\to z \overset{\delta[1]}\to x[1]$ in $\ct$. In particular, the abelian exact structure on $\ca$ is induced by the embedding $\ca \hookrightarrow \ct$. 

\begin{corollary}
\label{cor:abelian_embedding}
Let $\cc$ be a pre-triangulated {\degfin} dg-category over $\mathbf k$ which is $n$CY for some odd $n$. Assume that $\ca$ is a proper abelian subcategory of $\ct = H^0(\cc)$. Then this embedding induces an embedding of Hall algebras \[
(\ch(\ca)_{\scriptscriptstyle{\mathrm{tw}}}, \diamond) \hookrightarrow \hallcy(\ct), \,\,\, m_z \mapsto a_z,
\]
and the left hand side coincides with $(\ch(\ce)_{\scriptscriptstyle{\mathrm{tw}}}, \ast)$ whenever the assumptions of Corollary~\ref{cor:embedding_same_twist} hold.
\end{corollary}

\begin{example}
The above statements apply in the following special cases.
\begin{itemize}
\item[(i)] If $\ca$ is the heart of a $t$-structure on $\ct$, then it is a proper abelian subcategory and Corollary \ref{cor:abelian_embedding} applies. This is the $n$CY counterpart of the heart property of derived Hall algebras established by To\"en \cite{ToenDerivedHall}.
\item[(ii)] If $\ce$ is a \emph{presilting} subcategory of $\ct$, meaning that $\Ext^i_{\ct}(\ce, \ce) = 0, \forall i > 0$, then Theorem~\ref{thm:exact_embedding} applies. By definition, the induced exact structure on $\ce$ is the split one, so the Hall algebra $\ch(\ce)$ is in fact commutative up to powers of $q^{\pm 1}$, and so all the twisted versions we considered are in fact commutative up to powers of $q^{\pm \frac{1}{2}}$. This case includes \emph{silting} subcategories, and so equivalently hearts of bounded weight structures \cite{Bondarko} also known as co-hearts of bounded co-$t$-structures \cite{Pauksztello} (see \cite{MSSS}).
\item[(iii)] An $n$CY abelian category $\ca$ embeds as a proper abelian subcategory into its \emph{$m$-periodic derived category} $\cd_m(\ca)$ as long as $m \geq n$, and so Corollary \ref{cor:abelian_embedding} applies for this embedding. We will discuss in detail the case $m = 2$ in the next subsection. 
\item[(iv)] The original motivation of \cite{Jorgensen} was to study negative Calabi--Yau categories. These cannot admit any $t$-structures with non-zero heart \cite{HJY}, but can still contain interesting abelian subcategories. It is proved in \cite{Jorgensen} that  for $n \leq -2$, the extension closure $\langle \mathscr{S} \rangle$ of any \emph{$(-n)$-orthogonal collection} $\mathscr{S}$ inside any triangulated category $\ct$ satisfying our general assumptions is in fact a proper abelian subcategory. Further, by definition, for each $z, x \in \langle \mathscr{S} \rangle$, we have $\Ext^i(z, x) = 0$ for all $n +1 \leq i \leq -1$. Thus, for odd $n \leq -3$, Corollary \ref{cor:abelian_embedding} applies and, moreover, $(\ch(\langle \mathscr{S} \rangle)_{\scriptscriptstyle{\mathrm{tw}}}, \diamond)$ is just $(\ch(\langle \mathscr{S} \rangle) \otimes_{\QQ} \C, \cdot)$.
This covers $n$-cluster categories of finite acyclic quivers $Q$ for $n \leq -3$. Such a category is an algebraic $n$CY category, and $\mod \mathbf{k} Q$ is an example of its proper abelian subcategory arising as  the extension closure  of a $(-n)$-orthogonal collection.
We refer the reader to \cite{Jorgensen} for all the definitions, details, and references.
\end{itemize}
\end{example}

\begin{remark}
\label{rem:exact_ext^-1}
It is important to note that we do not require $\Ext_{\cc}^{-1}$ to vanish for objects in the subcategory, only the connecting maps $L_{\delta}, R_{\delta}$. 
Examples where $\Ext_{\cc}^{-1}(x, y)$ does not necessarily vanish for objects $x, y$ in a proper abelian subcategory of an oddCY category will be given in the next subsection in the form of root categories of $1$CY abelian categories. 
\end{remark}

\begin{remark}
\label{rem:distinguished}
If the subcategory $\ca \subset \ct$ is not proper, but only distinguished in the sense of \cite{Linckelmann}, meaning that some of triangles with 2 outer terms belonging to $\ca$ do not correspond to short exact sequences in $\ca$, a suitable modification of results of \cite{FangGorsky} implies that the appropriately twisted Hall algebra of the subcategory is a subalgebra of a certain nontrivial flat degeneration of $\hallcy(\ct)$. 
\end{remark}

\subsection{2-periodic complexes, root categories, and Drinfeld doubles}
\label{sec:root_categories}

In this section we let $\ca$ be a $\Hom$-finite (essentially small) abelian category of global dimension $n < \infty$, linear over $\mathbf k = \F_q$. 

With $\ca$ we associate the category $\cc_2(\ca)$ of \emph{2-periodic} (or $\Z/2$-graded) complexes, whose objects are diagrams of the form
\[
X^{\bullet} = \xymatrix{X^0 \ar@<0.5ex>[r]^{d^0} & X^1 \ar@<0.5ex>[l]^{d^1}},
\]
with $X^0, X^1 \in \ca$, $d^0 \circ d^1 = d^1 \circ d^0 = 0$.

The category $\cc_2(\ca)$ is also an abelian category. One can define cohomology $H^i, \, i \in \Z/2$ in the usual way $H^i(X^{\bullet}) \coloneqq  \Kerr d^i/  \Imm d^{i-1}$.  One has natural definitions of acyclic complexes forming the full subcategory $\cc_{2, ac}(\ca)$, quasi-isomorphisms, the homotopy category, and thus one can define the \emph{2-periodic} (or \emph{$\Z/2$-graded}) derived category $\cd_2(\ca)$ just as in the $\Z$-graded case. In particular, for all $X^{\bullet}, Y^{\bullet} \in \cc_2(\ca)$, $ i \geq 0$, we have naturally defined maps $\Ext^i_{\cc_2(\ca)}(X^{\bullet}, Y^{\bullet}) \to \Ext^i_{\cd_2(\ca)}(X^{\bullet}, Y^{\bullet}).$ It was shown in \cite[Proposition 4.1.2, Proposition 4.1.9]{Gorsky_thesis} that these maps are bijective for  all $X^{\bullet}, Y^{\bullet} \in \cc_2(\ca)$, $i > n$.

The category $\cd_2(\ca)$ can be alternatively defined as the triangulated hull of the orbit category $\cd^b(\ca)/[2]$ in the sense of \cite{Keller2005}. The latter orbit category was first considered by Happel \cite{Happel} who coined the name \emph{root category}, since in the case of $\ca$ being the category of representations of an orientation of a Dynkin diagram $\Gamma$, the indecomposable objects of $\cd^b(\ca)/[2]$ are in bijection with the roots of the root system of the simple Lie algebra corresponding to $\Gamma$. In this case, $\cd^b(\ca)/[2]$ is already triangulated, and so coincides with its triangulated hull \cite{PengXiao1997}. Fu \cite{Fu} extended this approach to more general categories; following him, we sometimes also call $\cd_2(\ca)$ the root category of $\ca$ for arbitrary abelian category $\ca$ satisfying our general assumptions, and denote it by $\cR(\ca)$.

As $\ca$ is essentially small, its bounded derived category $\cd^b(\ca)$ admits a dg enhancement which we will denote by $\cd^b_{dg}(\ca)$ (see e.g. \cite{Krause}).
Consider the perfect derived category of the Laurent polynomial algebra $\perf \mathbf k[t^{\pm 1}]$, with the generator $t$ in degree $2$. The tensor product category $\cD^b_{dg}(\ca) \otimes \perf \mathbf k[t^{\pm 1}]$  gives an enhancement of the triangulated hull of the orbit category $\cd^b(\ca)/[2]$, i.e. of the $2$-periodic derived category $\cd_2(\ca)$. 

\begin{remark}
The standard reference on the fact that $\Z/2$-graded dg categories can be interpreted as categories enriched over complexes over $\mathbf k[t^{\pm 1}]$ is \cite{Dyckerhoff2011}.
When $\ca$ has enough projectives and finite global dimension, we have an equivalence $\cd^b_{dg}(\ca) \cong \perf(\ca)$; in this case, the suitable dg enhancement is discussed in detail in \cite{Stai, Saito} based on \cite{Gorsky2013, Zhao}.
\end{remark}

We now present a construction introduced in \cite{Gorsky_thesis} which generalizes \cite{bridgeland_quantum, Gorsky2013}. 

Consider the set $\Iso(\cc_2(\ca))$ as a commutative monoid with addition $[A] + [B] = [A \oplus B]$. We define a monoid $M'$ as the quotient of $\Iso(\cc_2(\ca))$ modulo the ideal of relations generated by relations of the form $[X^{\bullet}] + [Z^{\bullet}] = [Y^{\bullet}]$ for short exact sequences $X^{\bullet} \to Y^{\bullet} \to Z^{\bullet}$ with at least one of $X$ and $Z$ being acyclic. Further, we consider the localization $\widetilde{M}$ of the monoid $M'$ at the classes of acyclic complexes. 

We have commutative diagrams of monoids
\[
\begin{tikzcd}
   M_0(\cc_{2, ac}(\ca)) \arrow[r,>->] \arrow[d,->,]&  M' \arrow[r,->>,"p'"]\arrow[d,->,] & \mbox{Iso}(\cD_2(\ca)) \arrow[d,equal]\\  
    K_0(\cc_{2, ac}(\ca)) \arrow[r,>->] &  \widetilde{M} \arrow[r,->>,"p"] & \mbox{Iso}(\cD_2(\ca)) 
\end{tikzcd}
\]
and 
\[
\begin{tikzcd}
    M' \arrow[r,->>,"p'"]\arrow[d,->>,] & \mbox{Iso}(\cD_2(\ca)) \arrow[d,->>]\\  
   M_0(\cc_2(\ca)) \arrow[r,->>] & M_0(\cD_2(\ca)) = K_0(\cD_2(\ca)),
\end{tikzcd}
\]
where $M_0(-)$ denotes the Grothendieck monoid of an exact or a triangulated category in the sense of \cite{BerensteinGreenstein, enomoto2022grothendieck}, and of groups
\[
\begin{tikzcd}
    G(\widetilde{M}) = G(M') \arrow[r,->>]\arrow[d,->>,] & G(\Iso(\cD_2(\ca))) \arrow[d,->>]\\ 
   K_0(\cc_2(\ca)) \arrow[r,->>] & K_0(\cD_2(\ca)), 
\end{tikzcd}
\]
where $G(-)$ denotes the group completion of a monoid. 

We can think of $K_0(\cc_{2, ac}(\ca))$, (isomorphic to) the group of units of $M$, as the kernel of $p$.
In fact, $M'$ is the Grothendieck monoid of the exact structure on (the underlying additive category of) $\cc_2(\ca)$ obtained by pulling back the split exact structure on $\cd_2(\ca)$, and its group of units is trivial. $G(\widetilde{M}) = G(M')$ is the Grothendieck group of this exact structure, $G(\Iso(\cd_2(\ca)))$ is the Grothendieck group of the split exact structure on $\cd_2(\ca)$.
The category of acyclic complexes $\cc_{2, ac}(\ca)$ with its abelian exact structure is a Serre subcategory of $\cc_2(\ca)$ considered with this pullback exact structure. 

Finally, we consider the canonical inclusion map of Grothendieck groups given by the composition
\[
i_{ac}: K_0(\cc_{2, ac}(\ca)) \to \widetilde{M} \to G(\widetilde{M})
\]
and define $M$ to be the quotient of $\widetilde{M}$ by the ideal generated by $\Kerr(i_{ac})$.

Consider a pair of complexes $X^{\bullet}, Z^{\bullet} \in \cc_2(\ca)$. For every $\delta \in \Ext_{\cD_2(\ca)}^1(Z^{\bullet}, X^{\bullet})$, the pair $X^{\bullet}, Z^{\bullet}$ determines a canonical lift $\mt(\delta) \in M$ of the isomorphism class of $\mbox{Cone}(\delta)$ via $p$. 
Namely, if $\delta$ is represented by a roof $Z^{\bullet}[-1] \overset{\alpha[-1]}\leftarrow Z'^{\bullet}[-1] \overset{\beta}\to X^{\bullet}$, with $\alpha[-1]$ being a quasi-isomorphism, then $\mt(\delta) := [\mbox{Cone}(\beta)] - [\mbox{Cone}(\alpha[-1])]$. It is proved in \cite{Gorsky_thesis} that this does not depend on the roof.

We consider the ``truncated relative Euler pairing'', of sorts:
\[
\langle Z^{\bullet}, X^{\bullet} \rangle_{\cc_2/\cd_2, > 0} \coloneqq \sum_{i=1}^{\infty} (-1)^i \left(\dim \Ext^i_{\cc_2(\ca)}(Z^{\bullet}, X^{\bullet}) - \dim \Ext^i_{\cD_2(\ca)}(Z^{\bullet}, X^{\bullet})\right);
\]
and the Euler pairing   
\[
\langle -,?\rangle_{\cc_2} \coloneqq \sum_{i=0}^{\infty} (-1)^i \dim \Ext^i_{\cc_2(\ca)}(-, ?),
\]
the latter being well-defined if one of the arguments is acyclic and descending to the corresponding (quotients of) Grothendieck groups $K_0(\cc_2(\ca)) \times (K_0(\cc_{2, ac}(\ca))/\Kerr(i_{ac}))$ and $(K_0(\cc_{2, ac}(\ca))/\Kerr(i_{ac})) \times K_0(\cc_2(\ca))$. 

\begin{definition}
The Hall algebra $\cS\cd\ch_2(\ca)$ is a $\mathbb{Q}$-algebra with basis given by $M$ and product
\begin{align*}
[Z^{\bullet}] \cdot [X^{\bullet}] &= 
 \underbrace{\left(\frac{1}{|\Hom_{\cc_2(\ca)}(Z^{\bullet}, X^{\bullet})|}  
\prod_{i = 1}^{\infty}  \left(\frac{|\Ext^i_{\cc_2(\ca)}(Z^{\bullet}, X^{\bullet})|}{|\Ext^i_{\cD_2(\ca)}(Z^{\bullet}, X^{\bullet})|}\right)^{(-1)^{i+1}}\right)}_{q^{-\dim\Hom_{\cc_2(\ca)}(Z^{\bullet}, X^{\bullet}) - \langle Z^{\bullet}, X^{\bullet} \rangle_{\cc_2/\cd_2, > 0}}}
\sum_{\delta \in \Ext^1_{\cD_2(\ca)}(Z^{\bullet}, X^{\bullet})} [\mt(\delta)]
\end{align*} 
for $X^{\bullet}, Z^{\bullet}$ arbitrary representatives in $\cc_2(\ca)$ of elements in $M'/\langle \Kerr(i_{ac})\rangle$ (in the above infinite product all factors for $i>n$ are $1$ by \cite[Proposition 4.1.2, Proposition 4.1.9]{Gorsky_thesis}), and more generally,
\begin{gather*}
([Z^{\bullet}] \cdot [K^{\bullet}]^{-1})\cdot ([X^{\bullet}]\cdot [L^{\bullet}]^{-1})  =
q^{\dagger} \sum_{\delta \in \Ext^1_{\cD_2(\ca)}(Z^{\bullet}, X^{\bullet})} [\mt(\delta)] \cdot [K^{\bullet} \oplus L^{\bullet}]^{-1},\\
\dagger=-\dim\Hom_{\cc_2(\ca)}(Z^{\bullet}, X^{\bullet}) - \langle Z^{\bullet}, X^{\bullet} \rangle_{\cc_2/\cd_2, > 0}+\langle L^{\bullet} - X^{\bullet}, K^{\bullet} \rangle_{\cc_2(\ca)} + \langle K^{\bullet}, X^{\bullet} \rangle_{\cc_2(\ca)}
\end{gather*}
for $X^{\bullet}, Z^{\bullet}$ as above and $K^{\bullet}, L^{\bullet} \in \cc_{2, ac}(\ca)$.
\end{definition}

The following is proved in \cite{Gorsky_thesis}: 

\begin{itemize}
\item The algebra is well-defined, that is, the product does not depend on the choice of representatives $(Z^{\bullet}, K^{\bullet})$ and $(X^{\bullet}, L^{\bullet})$ in $\Iso(\cc_2(\ca)) \oplus \Iso(\cc_{2, ac}(\ca))$ of elements in $M$.
\item It is associative, unital, and is naturally graded by $K_0(\cc_2(\ca))$.
\item It is a free module over the twisted group algebra of 
$K_0(\cc_{2, ac}(\ca))/\Kerr(i_{ac})$,
where the twist is given by the Euler pairing and each choice of representatives of the quasi-isomorphism classes gives a basis. In fact, we have
\begin{equation}
\label{eq:prod_acyclic_sdh}
[K^{\bullet}] \cdot [X^{\bullet}] = q^{-\langle K^{\bullet}, X^{\bullet} \rangle_{\cc_2}} [K^{\bullet} \oplus X^{\bullet}]; \, \, [X^{\bullet}] \cdot [K^{\bullet}] = q^{-\langle X^{\bullet}, K^{\bullet} \rangle_{\cc_2}} [K^{\bullet} \oplus X^{\bullet}],
\end{equation}
for $K^{\bullet}$ being the class of an acyclic complex. 
\item The algebra is tilting invariant. 
\item If $\ca$ has enough projectives, the map $i_{ac}$ is injective and so $M = \widetilde{M}$ by results in \cite{bridgeland_quantum}. The algebra is isomorphic to the localization of the ordinary Hall algebra $\cH(\cc_2(\mathrm{proj} \cA))$ at the classes of contractible complexes. The latter was introduced by Bridgeland in his seminal work \cite{bridgeland_quantum}, and so is often called the \emph{Bridgeland--Hall algebra of $\ca$}.
\item If $\ca$ is hereditary, i.e. $n = 1$, then the map $i_{ac}$ is  injective and so $M = \widetilde{M}$ by results in \cite{LuPeng}. The algebra is isomorphic to the localization of a certain ideal quotient of the ordinary Hall algebra $\cH(\cc_2(\ca))$ at the multiplicative Ore set given by the classes of acyclic complexes. This localization was introduced in partial generality in \cite{Gorsky2013} and in full generality in \cite{LuPeng} under the name \emph{$\Z/2$-graded semi-derived Hall algebra of $\ca$}. This name explains our notation.
\end{itemize}

Such algebras, for $\ca$ hereditary, appeared in the context of the additive categorification of quantum groups. More precisely, their twisted versions were used to realize Drinfeld doubles of twisted and extended Hall algebras of $\ca$, see \cite{bridgeland_quantum, Gorsky_thesis, LuPeng, Yanagida}.
The twist is in fact defined for $\ca$ of arbitrary finite global dimension: We denote by $(\cS\cd\ch_{2, \scriptscriptstyle{\mathrm{tw}}}(\ca), \ast)$ the algebra obtained from $(\cS\cd\ch_2(\ca) \otimes_{\QQ} \C, \cdot)$ by twisting the multiplication by the rule
\begin{equation}
\label{eq:sdh_twist_dw}
[Z^{\bullet}] \ast  [X^{\bullet}] =  q^{\frac{1}{2}\langle Z^{\bullet}, X^{\bullet} \rangle_{\scriptscriptstyle{\mathrm{dw}}}} [Z^{\bullet}] \cdot [X^{\bullet}],
\end{equation}
where 
\[
\langle Z^{\bullet}, X^{\bullet} \rangle_{\scriptscriptstyle{\mathrm{dw}}} \coloneqq \langle Z^0, X^0 \rangle_\ca + 
\langle Z^1, X^1 \rangle_\ca.
\]
This bilinear form descends to the grading group of the algebra --- that is, to $K_0(\cc_2(\ca))$. 

Assume that $\cd^b_{dg}(\ca)$  admits a weak right $n$CY structure. This in particular implies that the triangulated category $\cd^b(\ca)$ is $n$CY and the abelian category $\ca$ is $n$CY. The category $\perf \mathbf k[t^{\pm 1}]$, with the generator $t$ in degree $2$, is $\Z/2$-graded and naturally $2$CY, thus it can be also seen as a $0$CY category. The tensor product category $\cD^b_{dg}(\ca) \otimes \perf \mathbf k[t^{\pm 1}]$, which enhances $\cd_2(\ca)$, thus admits a (weak right) $n$CY structure.

\begin{remark}
It is straightforward to check that for $Z^{\bullet}, X^{\bullet}$ concentrated in the same degree --- say, in degree 0 --- we have 
\[
\Ext^i_{\cd_2(\ca)}(Z^{\bullet}, X^{\bullet}) \cong \bigoplus_{j = i \mod 2; j \geq 0} \Ext^j_{\ca}(Z^0, X^0).
\]
This implies that if $\ca$ has global dimension at most 2, the full subcategory of $\cd_2(\ca)$ formed by complexes concentrated in degree $0$, which is equivalent to $\ca$ as an abelian category, is a proper abelian subcategory of $\cd_2(\ca)$, but if the global dimension of $\ca$ is larger, it is only a distinguished abelian subcategory (i.e. it is not extension-closed in $\cd_2(\ca)$). We thus have the following (cf. Remark \ref{rem:distinguished}).
\end{remark}

\begin{lemma}
If $\ca$ is $1$CY, the embedding $\ca \to \cd_2(\ca)$ as complexes concentrated in degree $i$, for $i \in \Z/2$, induces an algebra embedding 
$\ch_{2, \scriptscriptstyle{\mathrm{tw}}}(\ca) \hookrightarrow \cS\cd\ch_{2, \scriptscriptstyle{\mathrm{tw}}}(\ca)$.

If $\ca$ is $n$CY for some odd $n \geq 3$, such an embedding of categories does not induce an algebra embedding into $\cS\cd\ch_{2, \scriptscriptstyle{\mathrm{tw}}}(\ca)$, but only into its certain flat degeneration.
\end{lemma}

\begin{remark}
To illustrate Remark \ref{rem:exact_ext^-1}, 
we note that if $\ca$ is $1$CY, for $Z^{\bullet}, X^{\bullet}$ concentrated in degree $0$, we have $\Ext^{-1}(Z^{\bullet}, X^{\bullet}) \cong \Ext^1(Z^0, X^0)$ which is non-zero whenever $\Hom(X^0, Z^0) \neq 0$. In particular, $\Ext^{-1}(Z^{\bullet}, Z^{\bullet}) \neq 0$ for any (non-zero) $Z^{\bullet}$ concentrated in a single degree. Still, such $Z^{\bullet}$ form a proper abelian subcategory of $\cd_2(\ca)$.
\end{remark}

We are now able to formulate the main result of this subsection.

\begin{theorem}
\label{thm:root_cat_comparison}
Assume that  $\cd^b_{dg}(\ca)$ is $n$CY, for some (positive) odd $n$.
Then the classes of acyclic complexes are central in $\cS\cd\ch_{2, \scriptscriptstyle{\mathrm{tw}}}(\ca)$, and $\hallcy(\cd_2(\ca))$ is isomorphic to the central reduction by the ideal $\left\langle [K^{\bullet}] - 1 \mid K^{\bullet} \text{ is acyclic} \right\rangle$ of the algebra  $\cS\cd\ch_{2, \scriptscriptstyle{\mathrm{tw}}}(\ca)$.
\end{theorem}

Before proceeding with the proof, we give an alternative definition of the form $\langle Z^{\bullet}, X^{\bullet} \rangle_{\scriptscriptstyle{\mathrm{dw}}}$ (which uses the $n$CY property of the category $\cd_2(\ca)$).

\begin{lemma}
\label{lem:our_twist_sdh}
The form 
\[
\begin{aligned}
&\dim\Hom_{\cc_2(\ca)}(Z^{\bullet}, X^{\bullet}) + \langle Z^{\bullet}, X^{\bullet} \rangle_{\cc_2/\cd_2, > 0} \\ &- \dim\Hom_{\cc_2(\ca)}(X^{\bullet}, Z^{\bullet}) - \langle X^{\bullet}, Z^{\bullet} \rangle_{\cc_2/\cd_2, > 0} + \langle Z^{\bullet}, X^{\bullet} \rangle_{1,\ldots,n-1}
\end{aligned}
\]
where the last term is the truncated Euler pairing in the category $\cd_2(\ca)$ defined as in earlier sections,
is well-defined on $K_0(\cc_2(\ca)) \times K_0(\cc_2(\ca))$.
\end{lemma}

\begin{proof}
Suppose that we have a short exact sequence $0 \to Z_1^\bullet \to Z_2^\bullet \to Z_3^\bullet \to 0$
in $\cc_2(\ca)$; it induces a distinguished triangle in $\cd_2(\ca)$. We have a pair of long exact sequences from this short exact sequence (resp. triangle) to $X^{\bullet}$ in $\cc_2(\ca)$ (resp. in $\cd_2(\ca)$), and another pair of those from $X^{\bullet}$ to this short exact sequence (resp. triangle). Let $a^i_{\cc_2}, a^i_{\cd_2}, b^i_{\cc_2}, b^i_{\cd_2}$ be the connecting morphisms from $\Ext^{i-1}$ to $\Ext^i$ in these four long exact sequences. 
Using the first pair, we see that 
\begin{align*}
\sum_{j=1}^3(-1)^{j-1}&\left(\dim\Hom_{\cc_2(\ca)}(Z_j^{\bullet}, X^{\bullet}) + \langle Z_j^{\bullet}, X^{\bullet} \rangle_{\cc_2/\cd_2, > 0}\right) = \\
&= \rank a^1_{\cc_2} + \sum_{i=1}^{\infty} (-1)^i (\rank a^i_{\cc_2} + \rank a^{i+1}_{\cc_2} -  \rank a^i_{\cd_2} - \rank a^{i+1}_{\cd_2}) =
\rank a^1_{\cd_2}.
\end{align*}
Similarly, using the second pair, we get that 
\begin{align*}
\sum_{j=1}^3(-1)^{j-1}&\left(\dim\Hom_{\cc_2(\ca)}(X^{\bullet}, Z_j^{\bullet}) + \langle X^{\bullet},Z_j^{\bullet} \rangle_{\cc_2/\cd_2, > 0}\right) =
\rank b^1_{\cd_2}.
\end{align*}
Thanks to $n$ being odd, we find from the second sequence that 
\[
\langle Z_1^{\bullet}, X^{\bullet} \rangle_{1,\ldots,n-1} - \langle Z_2^{\bullet}, X^{\bullet} \rangle_{1,\ldots,n-1} + \langle Z_3^{\bullet}, X^{\bullet} \rangle_{1,\ldots,n-1} = - \rank a^1_{\cd_2} + \rank a^n_{\cd_2}.
\]
Finally, the $n$CY property of $\cd_2(\ca)$ implies that $b^1_{\cd_2}$ is dual to $a^n_{\cd_2}$, and so their ranks coincide. 
Combining this all together, we see that the alternating sum of the form in the statement of the lemma for $Z_i^\bullet, \, i = 1, 2, 3$ and $X^\bullet$ equals
\[\rank a^1_{\cd_2} - \rank b^1_{\cd_2} - \rank a^1_{\cd_2} + \rank a^n_{\cd_2} = 0.\]
The dual arguments show the additivity of the form on short exact sequences in the second argument.
\end{proof}

\begin{lemma}
\label{lem:forms_agree_sdh}
The form defined in Lemma \ref{lem:our_twist_sdh} coincides with the form $\langle Z^{\bullet}, X^{\bullet} \rangle_{\scriptscriptstyle{\mathrm{dw}}}$.
\end{lemma}

\begin{proof}
We need to prove that, for any pair of complexes $Z^{\bullet}, X^{\bullet} \in \cc_2(\ca),$ we have
\begin{equation}
\label{cw_agrees_with_twist}
\begin{aligned}
\langle Z^{\bullet}, X^{\bullet} \rangle_{\scriptscriptstyle{\mathrm{dw}}} =
&\left(\dim\Hom(Z^{\bullet}, X^{\bullet}) + \langle Z^{\bullet}, X^{\bullet} \rangle_{\cc_2/\cd_2, > 0}\right) \\ &- \left(\dim\Hom(X^{\bullet}, Z^{\bullet}) + \langle X^{\bullet}, Z^{\bullet} \rangle_{\cc_2/\cd_2, > 0}\right) +   \langle Z^{\bullet}, X^{\bullet} \rangle_{1,\ldots,n-1}.
\end{aligned}
\end{equation}
Both sides are well-defined on $K_0(\cc_2(\ca)) \times K_0(\cc_2(\ca)),$ and we know that $K_0(\cc_2(\ca))$ is generated by the classes of complexes concentrated in a single degree
(see e.g. \cite[Proposition 4.1.9]{Gorsky_thesis}). 
Therefore, it is sufficient to verify the identity \eqref{cw_agrees_with_twist} on such complexes. 

The formulas expressing higher extensions groups in 
$\cc_2(\ca)$ and in 
$\cD_2(\ca)$ between stalk complexes in terms of extensions in $\ca$ are obtained by straightforward calculations and can be found in \cite[Lemma 4.1.1]{Gorsky_thesis}. 
Using these, we compute all the terms in \eqref{cw_agrees_with_twist} for a pair of stalk complexes. There are two cases:

\emph{Case 1}: $Z^{\bullet}, X^{\bullet}$ are concentrated in the same degree. Without loss of generality, we assume that
\[
Z^{\bullet} = \xymatrix{0 \ar@<0.5ex>[r] & Z \ar@<0.5ex>[l]}, \qquad X^{\bullet} = \xymatrix{0 \ar@<0.5ex>[r] & X \ar@<0.5ex>[l]},
\]
for some $Z, X \in \ca$. 

We then have the following, where $k=(n-1)/2$:
\begin{align*}
\langle Z^{\bullet}, X^{\bullet} \rangle_{\scriptscriptstyle{\mathrm{dw}}} & = \langle Z, X \rangle_{\ca};\\
\dim\Hom_{\cc_2(\ca)}(Z^{\bullet}, X^{\bullet}) & = \dim\Hom(Z, X) = \dim\Ext^0 (Z, X);\\
\langle Z^{\bullet}, X^{\bullet} \rangle_{\cc_2/\cd_2, > 0} & = \sum_{i=0}^{k} i \dim\Ext^{2i+1}(Z, X) - \sum_{j=1}^{k} (j-1) \dim\Ext^{2j}(Z, X);\\
\dim\Hom(X^{\bullet}, Z^{\bullet}) & = \dim\Hom (X, Z) =  \dim\Ext^{2k+1} (Z, X) ;\\
\langle X^{\bullet}, Z^{\bullet} \rangle_{\cc_2/\cd_2, > 0} & = \sum_{j=0}^{k} (k-j) \dim\Ext^{2j}(Z, X) - \sum_{i=0}^{k-1} (k-i-1) \dim\Ext^{2i+1}(Z, X);\\
\langle Z^{\bullet}, X^{\bullet} \rangle_{1,\ldots,n-1} & = k \langle Z, X \rangle_{\ca}.
\end{align*}
Summing up, we find that the right hand side of \eqref{cw_agrees_with_twist} equals
\begin{align*}
&\dim\Ext^0 (Z, X) + \sum_{i=0}^{k} i \dim\Ext^{2i+1}(Z, X) - \sum_{j=1}^{k} (j-1) \dim\Ext^{2j}(Z, X) \\
&- \dim\Ext^{2k+1} (Z, X) - \sum_{j=0}^{k} (k-j) \dim\Ext^{2j}(Z, X) + \sum_{i=0}^{k-1} (k-i-1) \dim\Ext^{2i+1}(Z, X) \\
&+ k \langle Z, X \rangle_{\ca} \\ &= - (k-1) \langle Z, X \rangle_{\ca} + k \langle Z, X \rangle_{\ca} = \langle Z, X \rangle_{\ca},
\end{align*}
which in turn equals the left hand side of \eqref{cw_agrees_with_twist}.

\emph{Case 2}: $Z^{\bullet}, X^{\bullet}$ are concentrated in different degrees. Without loss of generality, we assume that
\[
Z^{\bullet} = \xymatrix{Z \ar@<0.5ex>[r] & 0 \ar@<0.5ex>[l]}, \qquad X^{\bullet} = \xymatrix{0 \ar@<0.5ex>[r] & X \ar@<0.5ex>[l]},
\]
for some $Z, X \in \ca$. 

We then have the following:
\begin{align*}
\langle Z^{\bullet}, X^{\bullet} \rangle_{\scriptscriptstyle{\mathrm{dw}}} & = 0;\\
\dim\Hom(Z^{\bullet}, X^{\bullet}) & = 0;\\
\langle Z^{\bullet}, X^{\bullet} \rangle_{\cc_2/\cd_2, > 0} &= \sum_{i=0}^{k} i \dim\Ext^{2i}(Z, X) - \sum_{j=0}^{k} j \dim\Ext^{2j+1}(Z, X);\\
\dim\Hom(X^{\bullet}, Z^{\bullet}) & = 0;\\
\langle X^{\bullet}, Z^{\bullet} \rangle_{\cc_2/\cd_2, > 0} & = \sum_{i=0}^{k} (k-i) \dim\Ext^{2i+1}(Z, X) - \sum_{i=0}^{k} (k-i) \dim\Ext^{2i}(Z, X);\\
\langle Z^{\bullet}, X^{\bullet} \rangle_{1,\ldots,n-1} & = - k \langle Z, X \rangle_{\ca}.
\end{align*}
Summing up, we find that the right hand side of \eqref{cw_agrees_with_twist} equals
\begin{align*}
&\sum_{i=0}^{k} i \dim\Ext^{2i}(Z, X) - \sum_{j=0}^{k} j \dim\Ext^{2j+1}(Z, X) \\ &-\sum_{i=0}^{k} (k-i) \dim\Ext^{2i+1}(Z, X) + \sum_{i=0}^{k} (k-i) \dim\Ext^{2i}(Z, X) \\
&- k \langle Z, X \rangle_{\ca}\\
&= k \langle Z, X \rangle_{\ca} - k \langle Z, X \rangle_{\ca} = 0,
\end{align*}
which in turn equals the left hand side of \eqref{cw_agrees_with_twist}.
\end{proof}

\noindent \emph{Proof of Theorem \ref{thm:root_cat_comparison}.}
By Lemma \ref{lem:forms_agree_sdh}, the multiplication in the algebra 
$\cS\cd\ch_{2, \scriptscriptstyle{\mathrm{tw}}}(\ca)$ 
reads
\begin{align}
\label{eq:twisted_sdh}
[Z^{\bullet}] \ast [X^{\bullet}] = & q^{\frac{1}{2}\diamond} \times
\sum_{\delta \in \Hom_{\cD_2(\ca)}(Z^{\bullet}[-1], X^{\bullet})} [\mt(\delta)].
\end{align} 
where
\begin{align*}
\diamond= &-\dim\Hom_{\cc_2(\ca)}(Z^{\bullet}, X^{\bullet}) - \langle Z^{\bullet}, X^{\bullet} \rangle_{\cc_2/\cd_2, > 0}\\&- \dim\Hom_{\cc_2(\ca)}(X^{\bullet}, Z^{\bullet}) - \langle X^{\bullet}, Z^{\bullet} \rangle_{\cc_2/\cd_2, > 0} +  \langle Z^{\bullet}, X^{\bullet} \rangle_{1,\ldots,n-1}.
\end{align*}
For $K^{\bullet}$ acyclic, $X^\bullet \in \cc_2(\ca)$, we have 
$\langle K^{\bullet}, X^{\bullet} \rangle_{1,\ldots,n-1} = 0 = \langle X^{\bullet}, K^{\bullet} \rangle_{1,\ldots,n-1}$ and
\begin{align*}
\left(\dim\Hom_{\cc_2(\ca)}(K^{\bullet}, X^{\bullet}) + \langle K^{\bullet}, X^{\bullet} \rangle_{\cc_2/\cd_2, > 0}\right) = \langle K^{\bullet}, X^{\bullet} \rangle_{\cc_2};\\
\left(\dim\Hom_{\cc_2(\ca)}(X^{\bullet}, K^{\bullet}) + \langle X^{\bullet}, K^{\bullet} \rangle_{\cc_2/\cd_2, > 0}\right) = \langle X^{\bullet}, K^{\bullet} \rangle_{\cc_2}.
\end{align*}
So we have
\[
[K^\bullet] \ast [X^\bullet] = q^{-\frac{1}{2} \left( \langle K^{\bullet}, X^{\bullet} \rangle_{\cc_2(
\ca)} +  \langle X^{\bullet}, K^{\bullet} \rangle_{\cc_2(\ca)}\right)} [K^\bullet \oplus X^\bullet] = [X^\bullet] \ast [K^\bullet].
\]
This shows that the classes of acyclic complexes are indeed central in $\cS\cd\ch_{2, \scriptscriptstyle{\mathrm{tw}}}(\ca)$.

We perform the central reduction by the ideal $\left\langle [K^{\bullet}] - 1 \mid K^{\bullet} \text{ is acyclic} \right\rangle$. Since $\cS\cd\ch_{2, \scriptscriptstyle{\mathrm{tw}}}(\ca)$ was a free module over the twisted group algebra of the Grothendieck group $K_0(\cc_{2, ac}(\ca))$
of acyclic complexes, the central reduction is a $\C$-vector space where each choice of representative of 
isomorphism classes in $\cd_2(\ca)$ gives a basis. It thus remains to prove that the multiplication in this central reduction agrees with the one in $\hallcy(\cd_2(\ca))$.

We rescale the basis by $q^{\frac{1}{2}\left(\dim\Hom_{\cc_2(\ca)}(M^{\bullet}, M^{\bullet}) + \langle M^{\bullet}, M^{\bullet} \rangle_{\cc_2/\cd_2, > 0}\right)}$. In the rescaled basis, the formula \eqref{eq:twisted_sdh} reads
\begin{align}
\label{eq:product_with_f(delta)}
[Z^{\bullet}] \ast [X^{\bullet}] = & q^{f(\delta)}
 \sum_{\delta \in \Hom_{\cD_2(\ca)}(Z^{\bullet}[-1], X^{\bullet})} [\mt(\delta)],
\end{align}
where 
\begin{align*}
f(\delta) \coloneqq& \frac{1}{2} (\langle Z^{\bullet}, X^{\bullet} \rangle_{1,\ldots,n-1} -\dim\Hom_{\cc_2(\ca)}(Z^{\bullet}, X^{\bullet}) - \langle Z^{\bullet}, X^{\bullet} \rangle_{\cc_2/\cd_2, > 0} \\
&-\dim\Hom_{\cc_2(\ca)}(X^{\bullet}, Z^{\bullet}) - \langle X^{\bullet}, Z^{\bullet} \rangle_{\cc_2/\cd_2, > 0} \\&- \dim\Hom_{\cc_2(\ca)}(Z^{\bullet}, Z^{\bullet}) - \langle Z^{\bullet}, Z^{\bullet} \rangle_{\cc_2/\cd_2, > 0}\\
&-\dim\Hom_{\cc_2(\ca)}(X^{\bullet}, X^{\bullet}) - \langle X^{\bullet}, X^{\bullet} \rangle_{\cc_2/\cd_2, > 0} \\&+\dim\Hom_{\cc_2(\ca)}(\mt(\delta), \mt(\delta)) + \langle \mt(\delta), \mt(\delta) \rangle_{\cc_2/\cd_2, > 0}).    
\end{align*}

As we performed the central reduction, we can now assume without loss of generality that basis elements are represented by complexes in prescribed quasi-isomorphism classes. Furthermore, we can take any representative of such a quasi-isomorphism class and the above formula \eqref{eq:product_with_f(delta)} will remain correct. In other words, $f(\delta)$ does indeed depend only on $\delta$, and not on the choice of representatives $Z^{\bullet}, X^{\bullet}$ of the respective quasi-isomorphism classes. For a fixed $\delta$, we take an arbitrary lift to a morphism $\beta: Z'^{\bullet} \to X'^{\bullet}$ in $\cc_2(\ca)$, so that the corresponding $\mt(\delta)$ is given by $Y'^\bullet\coloneqq\mbox{Cone}(\beta)$. The formula \eqref{dim_ext_and_4x4matrix} applies both in $\cc_2(\ca)$ and in $\cd_2(\ca)$ for $\beta$ and self-extensions of $Y'^\bullet$ (to directly reduce the former to the dg case, we may consider $\cc_2(\ca)$ as a heart of a bounded $t$-structure in an algebraic category $\cd^b(\cc_2(\ca))$, then use the arguments of Section \ref{subsec:hallOddCY} for its pre-triangulated dg enhancement). 
We thus find that
\begin{align*}
f(\delta) &= \frac{1}{2} (\langle Z^{\bullet}, X^{\bullet} \rangle_{1,\ldots,n-1} - (\dim\Hom_{\cc_2(\ca)}(Z'^\bullet \oplus X'^\bullet, Z'^\bullet \oplus X'^\bullet) - \dim\Hom_{\cc_2(\ca)}(Y'^\bullet, Y'^\bullet)) \\
&\quad-\sum_{i=1}^{\infty} (-1)^i (\dim\Ext^i_{\cc_2(\ca)}(Z'^\bullet \oplus X'^\bullet, Z'^\bullet \oplus X'^\bullet) - \dim\Ext^i_{\cc_2(\ca)}(Y'^\bullet, Y'^\bullet) \\
&\quad-\dim\Ext^i_{\cd_2(\ca)}(Z'^\bullet \oplus X'^\bullet, Z'^\bullet \oplus X'^\bullet) - \dim\Ext^i_{\cd_2(\ca)}(Y'^\bullet, Y'^\bullet)) 
\\
&=\frac{1}{2} (\langle Z^{\bullet}, X^{\bullet} \rangle_{1,\ldots,n-1} - 
d_{\cc_2, 1}(\beta) - \sum_{i=1}^{\infty} (-1)^i (d_{\cc_2, i}(\beta) + d_{\cc_2, i+1}(\beta) - d_{\cd_2, i}(\beta) - d_{\cd_2, i+1}(\beta)))
\\
&=\frac{1}{2} (\langle Z^{\bullet}, X^{\bullet} \rangle_{1,\ldots,n-1} - d_{\cd_2, 1}(\beta)) = \frac{1}{2} (\langle Z^{\bullet}, X^{\bullet} \rangle_{1,\ldots,n-1} - d_{\cd_2, 1}(\delta)).
\end{align*} 
Thus, the formula \eqref{eq:product_with_f(delta)} can be rewritten as 
\[
[Z^{\bullet}] \ast [X^{\bullet}] =  q^{\frac{1}{2} \left(\langle Z^{\bullet}, X^{\bullet} \rangle_{1,\ldots,n-1} - d_{\cd_2, 1}(\delta)\right)}
 \sum_{\delta \in \Hom_{\cD_2(\ca)}(Z^{\bullet}[-1], X^{\bullet})} [\mt(\delta)],
\]
which coincides with the multiplication in $\hallcy(\cd_2(\ca))$ in the basis $\{u_x\}$. So our rescaling of the basis is an algebra isomorphism. \hfill$\Box$

\medskip

\begin{remark}
By inspecting the definition of the algebra $\cS\cd\ch_{2, \scriptscriptstyle{\mathrm{tw}}}(\ca)/\left\langle [K^{\bullet}] - 1 \mid K^{\bullet} \text{ is acyclic} \right\rangle$ and the proof of Theorem \ref{thm:root_cat_comparison}, we note that the statement remains correct if $\cd_2(\ca)$ is only required to be $n$CY as a triangulated category. Under this assumption, we can thus alternatively deduce the associativity of $\hallcy(\cd_2(\ca))$ (meaning the algebra defined in the basis $\{u_x\}$ by the corresponding product formula) from that of $\cS\cd\ch_{2, \scriptscriptstyle{\mathrm{tw}}}(\ca)/\left\langle [K^{\bullet}] - 1 \mid K^{\bullet} \text{ is acyclic} \right\rangle$. For $n = 1$, which is arguably the most interesting case as will be illustrated below, $\ca$ being $1$CY immediately implies that $\cd_2(\ca)$ is $1$CY as a triangulated category. This can be checked directly using the facts that  for $\ca$ hereditary, $\cd^b(\ca)/[2]$ is itself triangulated and thus equivalent to $\cd_2(\ca)$, and every object in all the triangulated categories involved is (quasi-)isomorphic to a direct sum of its homologies. 
\end{remark}

Since the algebra $\hallcy(\cd_2(\ca))$ depends only on the triangulated category $\cd_2(\ca)$, we deduce the same invariance property for the central reduction algebra. 

\begin{corollary}
Let $\ca, \ca'$ be a pair of abelian categories
such that 
$\cd_2(\ca) \overset\sim\to \cd_2(\ca')$ is an equivalence of $n$CY categories, for some odd $n \geq 1$. Then we have an isomorphism 
\[
\cS\cd\ch_{2, \scriptscriptstyle{\mathrm{tw}}}(\ca)/\left\langle [K^{\bullet}] - 1 \mid K^{\bullet} \text{ is acyclic} \right\rangle \overset\sim\to \cS\cd\ch_{2, \scriptscriptstyle{\mathrm{tw}}}(\ca')/\left\langle [K^{\bullet}] - 1 \mid K^{\bullet} \text{ is acyclic} \right\rangle.
\]
In particular, such central reductions are isomorphic as long as  $\cd^b(\ca) \overset\sim\to \cd^b(\ca')$.
\end{corollary}

This generalizes the tilting invariance (of central reductions) of $\cS\cd\ch_{2, \scriptscriptstyle{\mathrm{tw}}}(\ca)$ proved in \cite{Gorsky2013, Gorsky_thesis, LuPeng}.

For $n = 1$, the algebra $\cS\cd\ch_{2, \scriptscriptstyle{\mathrm{tw}}}(\ca)$ is isomorphic to the Drinfeld double of the twisted extended Hall algebra of $\ca$ \cite{bridgeland_quantum, Gorsky2013, Yanagida, LuPeng}. From the formulas, it follows directly (see also \cite{BurbanSchiffmannEllipticHall}) that, under our $n$CY assumption, the central reduction $\cS\cd\ch_{2, \scriptscriptstyle{\mathrm{tw}}}(\ca)/\left\langle [K^{\bullet}] - 1 \mid K^{\bullet} \text{ is acyclic} \right\rangle$ 
is then isomorphic to the Drinfeld double (DD) of the twisted (non-extended!) Hall algebra of $\ca$. We thus get the following.

\begin{corollary}
\label{cor:DD}
Let $\ct$ be a triangulated category equivalent to the root category $\cd_2(\ca)$ of a $1$CY abelian category $\ca$.
Then $\halluncy(\cd_2(\ca))$  is isomorphic to the Drinfeld double of the twisted Hall algebra of $\ca$. 
\end{corollary}

\begin{corollary}
\label{cor:der_inv_DD}
Let $\ca$ be a $1$CY abelian category. 
Then the Drinfeld double of the twisted Hall algebra of $\ca$ is an invariant of the root category $\cd_2(\ca)$. In particular, it is invariant under arbitrary derived equivalences $\cd^b(\ca) \overset\sim\to \cd^b(\cb)$.
\end{corollary}

\begin{remark}
We note that if the twisted (extended or not) Hall algebra of $\ca$ is only a topological bialgebra, the proper definition of the Drinfeld double as a Hopf algebra involves taking a certain completion, but this is not required if one is to consider only the associative algebra structure, see \cite[Section 4]{LuPeng} for a detailed account. 
\end{remark}

\begin{remark}
\label{rem:DD_comparisons}
We briefly discuss the place of Corollary \ref{cor:der_inv_DD} in the context of earlier works.
\begin{itemize}
\item[(i)] Derived invariance of the Drinfeld double $DD(\ch_{\scriptscriptstyle{\mathrm{tw}}}^{e}(\ca))$ of the twisted extended Hall algebra of a hereditary abelian category $\ca$ was first established by Cramer \cite{Cramer} without constructing it explicitly as a Hall-like algebra of a category. We do not know if there exist equivalences of root categories of abelian categories not induced by derived equivalences (in \cite[Section 4]{chen2023derived} the non-existence of such equivalences is conjectured). In other words, Corollary \ref{cor:der_inv_DD} certainly recovers Cramer's result in the case of $1$CY categories $\ca$ and possibly generalizes it.

\item[(ii)] It was claimed by Bridgeland \cite{bridgeland_quantum} and proved by Yanagida \cite{Yanagida} that the twisted localization of $\ch(\cc_2(\proj \ca))$ is isomorphic to $DD(\ch_{\scriptscriptstyle{\mathrm{tw}}}^{e}(\ca))$ provided $\ca$ has enough projectives, and in this generality results of the first named author \cite{Gorsky2013}, where this twisted localization of $\ch(\cc_2(\proj \ca))$ was realized as $\cS\cd\ch_{2, \scriptscriptstyle{\mathrm{tw}}}(\ca)$,
imply the tilting invariance of $DD(\ch_{\scriptscriptstyle{\mathrm{tw}}}^{e}(\ca))$, partially recovering results of Cramer \cite{Cramer}. Lu and Peng \cite{LuPeng} generalized this to the case of arbitrary essentially small finitary hereditary $\ca$. In all these works, the construction of the algebra used the choice of a hereditary category $\ca$, various homological information about the category of 2-periodic complexes $\cc_2(\ca)$, and taking a non-trivial algebra localization. While we restrict ourselves to the CY generality, our algebra does not depend on the choice of $\ca$, i.e. on the choice of the triangular decomposition, in a sense. One may say that it is defined \emph{globally}, adopting the terminology used by Peng and Xiao \cite{PengXiao} in a closely related context. 

\item[(iii)] Very recently, two different new constructions of twisted Hall algebras of root categories of hereditary abelian categories appeared in \cite{zhang2022hall} (for arbitrary essentially small finitary hereditary $\ca$) and in \cite{chen2023derived} (under a weaker version of the $1$CY condition). Their advantages are in not using the categories of complexes and algebra localizations. The former uses the structure constants of $\cd^b(\ca)$ in order to define the multiplication of an algebra associated with $\cd_2(\ca)$, while the latter seems to be closer in spirit to our approach, with multiplication counting extensions in $\cd_2(\ca)$. Both definitions however are not ``global'': in both works, the structure constants do depend on the choice of a hereditary abelian category $\ca$ whose root category is equivalent to a given triangulated category, and isomorphisms between an arbitrary object in the root category $\cd_2(\ca)$ and the direct sum of its homologies with respect to $\ca$ are used prominently in constructions and proofs. The algebras with suitably twisted multiplications are shown to be isomorphic to $DD(\ch_{\scriptscriptstyle{\mathrm{tw}}}^{e}(\ca))$, resp. to $DD(\ch_{\scriptscriptstyle{\mathrm{tw}}}(\ca))$. Their derived invariance is then deduced from \cite{Cramer}, and does not immediately follow from the definitions itself, as opposed to our $\halluncy(\ca)$. \emph{A posteriori}, $\halluncy(\ca)$ is then isomorphic to the central reduction of the algebra in \cite{zhang2022hall} (with an appropriate twist) and to the algebra in \cite{chen2023derived}, since each of these three algebras is isomorphic to $DD(\ch_{\scriptscriptstyle{\mathrm{tw}}}(\ca))$.

\item[(iv)] Going beyond the CY case, one does not look for the Drinfeld double of $\ch_{\scriptscriptstyle{\mathrm{tw}}}(\ca)$, as the same Hopf algebra as in the CY case is not well-defined on the latter. Instead, one extends $\ch_{\scriptscriptstyle{\mathrm{tw}}}(\ca)$ by the (twisted) group algebra of the Grothendieck group $K_0(\ca)$ or of its numerical quotient $K_0^{num}(\ca)$, and defines a Hopf algebra on this extended version $\ch_{\scriptscriptstyle{\mathrm{tw}}}^{\scriptscriptstyle{\mathrm{ex}}}(\ca)$. Now the (reduced) Drinfeld double construction can be aplied to this algebra. This recovers, in particular, quantum groups of Kac--Moody type associated with quivers. Thus, one would like to define an extended twisted Hall algebra of the root category of an arbitrary hereditary finitary abelian category. Most constructions discussed in points (i)--(iii) produce such algebras by different means, which use different extra bits of information involving the choice of category $\ca$ in nontrivial ways. \iftoggle{arxiv}{We expect that a more natural definition could be obtained using Kontsevich's construction \ref{ex:Kontsevich} with some carefully designed twist. We plan to address this in future work}.
 
\item[(v)] Peng and Xiao \cite{PengXiao} associated a Kac--Moody Lie algebra over $\Z/(q-1)$ with an arbitrary essentially small $\Hom$-finite $\F_q$-linear 2-periodic triangulated category. For the root category of a hereditary finite-dimensional algebra $A$, they proved that such a Lie algebra associated with $\cd_2(\mod A)$ is isomorphic to the Lie algebra naturally associated with $A$ in works of Gabriel and Ringel. The latter result was significantly generalized in the recent work \cite{fang2023lie}. There, for any finite-dimensional algebra $A$ of finite global dimension, a limit of a suitable motivic version of the twisted Bridgeland--Hall algebra of $A$ was introduced and shown to contain a Lie subalgebra isomorphic to a Lie algebra associated with the category $\cK_2(\proj A) \overset\sim\to \cd_2(\mod A)$ via certain constructible functions. The latter recovers the Lie algebra of Peng--Xiao \cite{PengXiao} for this category. The definition of this Lie algebra of constructible functions is ``global'', in the sense that it depends only on the structure of the triangulated root category and not on the algebra $A$. This shows the derived invariance of a certain limit version of the twisted Bridgeland--Hall algebra of $A$. All of this suggests that our algebras $\hallcy(\ct)$ admit, in the case of $\ct$ having ind-constructible stacks of objects, motivic versions whose classical limits may contain interesting Lie algebras. We do not pursue this line of research in this article.
\end{itemize}
\end{remark}

\iftoggle{arxiv}
{
\subsection{Extended Hall algebras of CY triples and their central reductions}

In this subsection we discuss a different construction of extended Hall algebras associated with triangulated categories $\ct$ realized as certain Verdier quotients. For $\ct$ algebraic and $n$CY for some odd $n$, we show that suitable central reductions of twisted versions of such extended algebras recover $\hallcy(\ct)$. The general constructions do not require the existence of enhancements, so we present them in the triangulated setting.

\subsubsection{Extended Hall algebras via Verdier quotients}
\label{subsec:ext_ha_reduction}

Assume that we have triangulated category $\cb$ and a thick subcategory $\cn \subseteq \cb$. Under suitable finiteness conditions, we can define an extended version of a Hall algebra of the Verdier quotient $\ct := \cb / \cn$, following \cite{Gorsky_thesis, Gorsky2024, Kontsevich_private}. This is similar in spirit to the construction of algebras $\cS\cd\ch_2(\ca)$, but instead of working with an abelian category of complexes and its localization $\cd_2(\ca)$, we take a localization of triangulated categories as an input. Let $\pi: \cb \to \ct$ denote the Verdier quotient functor.

Consider the set $\Iso(\cb)$ as a commutative monoid with addition $[x] + [y] = [x \oplus y]$.

We write $[u] \sim_c [v]$ if there exists a triangle $x \to y \to z \overset{h}\to x[1]$ in $\cb$ such that $\pi h = 0$ (equivalently, such that $h$ factors through $\cn$) and we have isomorphisms $u \overset\sim\to y$ and $v \overset\sim\to x \oplus z$ in $\cb$. We define a congruence $[u] \approx_c [v]$ by taking the transitive closure of the reflexive closure of the relation $\sim_c$. 

We define a monoid $M'$ as the quotient $\Iso(\cb) /_{\approx_c}$. It can be equivalently described as the quotient of $\Iso(\cb)$ modulo the ideal of relations generated by relations of the form $[x] + [z] = [y]$ for triangles  $x \to y \to z \overset{h}\to x[1]$ with at least one of $x$ and $z$ belonging to $\cn$.

We have a commutative diagram of monoids 
\[
\begin{tikzcd}[ampersand replacement=\&]
    K_0(\cn) \arrow[r,>->] \& M' \arrow[r,->>,"p"]\arrow[d,->>,] \& \Iso(\ct) \arrow[d,->>]\\   
   \& K_0(\cb) \arrow[r,->>] \& K_0(\ct). 
\end{tikzcd}
\]

We can think of $K_0(\cn)$
as the kernel of $p$. We define $M$ as the quotient of $M'$ by the ideal generated by the kernel of the canonical composition $i_{\cn}: K_0(\cn) \to M' \to G(M')$, with the latter term being the group completion of $M'$.

\begin{remark}
The fact that $\approx_c$ is indeed a congruence is proved in \cite[Proposition 2.7]{enomoto2022grothendieck}. In fact, $M'$ 
is the Grothendieck monoid of the extriangulated structure on the underlying additive category of $\cb$ obtained by pulling back the split structure on $\ct$ via $\pi$, in the sense of \cite[Section 2, Appendix A]{CKP24}, and $\cn$ is a Serre subcategory of $\cb$ considered with this extriangulated structure.
The fact that the canonical map $M_0(\cn) = K_0(\cn) \to M'$ is injective is a special case of \cite[Proposition 37]{enomoto2022grothendieck}, and the statement that $M'/K_0(\cn) \cong  \Iso(\ct)$ is a special case of \cite[Corollary 4.28]{enomoto2022grothendieck}. Taking the quotient by the ideal generated by $\Kerr(i_{\cn})$ guarantees that the images of the classes of objects in $\cn$ act injectively on $M$.
\end{remark}

Consider a pair of elements $x, z \in \cb$. For every $\delta \in \Hom_{\ct}(z[-1], x)$, the pair $x, z$ determines a canonical lift $\mt(\delta) \in M$ of the isomorphism class of $\mbox{Cone}(\delta)$ (which we will denote by $C(\delta)$ as before) via $p$. 
Namely, if $\delta$ is represented by a roof $z[-1] \overset{\alpha[-1]}\leftarrow z'[-1] \overset{\beta}\to x $, with $C(\alpha[-1]) \in \cn$, then $\mt(\delta) := [C(\beta) \oplus C(\alpha)] = [C(\beta)] + [C(\alpha)]$. It is straightforward to check that this does not depend on the roof.

Similarly to the notation in Section  \ref{sec:root_categories}, we consider ``truncated relative Euler pairings'' for $z, x \in \cb$ and $I \subseteq \Z$:
\[
\langle z, x \rangle_{\cb/\ct, I} \coloneqq \sum_{i \in I} (-1)^i \left(\dim \Ext^i_{\cb}(z, x) - \dim \Ext^i_{\ct}(z, x)\right).
\]

We are now almost ready to present two versions of Hall algebras associated with $\cn \subseteq \cb$, well-defined under different finiteness conditions. In either of the settings, for $u \in \cn, v \in \cb$, we have 
$\Ext^i_{\cb}(u, v) = \Ext^i_{\cb}(v, u) = 0$ for $|i| \gg 0$. The Euler pairings in terms of extension groups in $\cb$ are thus well-defined on $K_0(\cn) \times M'$ and on  $M' \times K_0(\cn)$. In fact, they descend onto $(K_0(\cn)/\Kerr(i_{\cn})) \times M$ and on  $M \times (K_0(\cn)/\Kerr(i_{\cn}))$. We denote both of them by $\langle -, ? \rangle_{\cb}$. There is no clash of notation, since the restrictions of the pairings  on  $(K_0(\cn)/\Kerr(i_{\cn})) \times (K_0(\cn)/\Kerr(i_{\cn}))$ agree.

\begin{definition}
\label{def:subcat+LRF}
Assume that we have a thick subcategory $\cn \subseteq \cb$, $\cb$ is left locally homologically finite and, moreover, for each $u, v \in \cb$, the map $\Ext^i_{\cb}(u, v) \to \Ext^i_{\ct}(u, v)$ is a bijection for $i \gg 0$. The Hall algebra $\ch'(\cb, \cn)$ is a $\mathbb{Q}$-algebra with basis given by $M$ and product
\[
[z] \cdot [x] \coloneqq 
\underbrace{\left(\prod_{i = -\infty}^{0} |\Ext^i_{\cb}(z, x)|^{(-1)^{i+1}}  
\prod_{i = 1}^{\infty}  \left(\frac{|\Ext^i_{\cb}(z, x)|}{|\Ext^i_{\ct}(z, x)|}\right)^{(-1)^{i+1}}\right)}_{q^{-\langle z, x \rangle_{\cb, \leq 0} - \langle z, x \rangle_{\cb/\ct, > 0}}}
\sum_{\delta \in \Ext^1_{\ct}(z, x)} [\mt(\delta)], 
 \]
for $x, z$ arbitrary representatives in $\cb$ of elements in $M$.
\end{definition}

\begin{definition}
\label{def:subcat+RLF}
Assume that we have a thick subcategory $\cn \subseteq \cb$, $\cb$ is right locally homologically finite and, moreover, for each $u, v \in \cb$, the map $\Ext^i_{\cb}(u, v) \to \Ext^i_{\ct}(u, v)$ is a bijection for $i \ll 0$. The Hall algebra $\ch(\cb, \cn)$ is a $\mathbb{Q}$-algebra with basis given by $M$ and product
\[
[z] \cdot [x] \coloneqq
\underbrace{\left(\prod_{i = -\infty}^{0} \left(\frac{|\Ext^i_{\cb}(z, x)|}{|\Ext^i_{\ct}(z, x)|}\right)^{(-1)^i}
\prod_{i = 1}^{\infty} |\Ext^i_{\cb}(z, x)|^{(-1)^i} \right)}_{q^{\langle z, x \rangle_{\cb/\ct, \leq 0} + \langle z, x \rangle_{\cb, > 0}}}
\sum_{\delta \in \Ext^1_{\ct}(z, x)} [\mt(\delta)], 
 \]
 for $x, z$ arbitrary representatives in $\cb$ of elements in $M$.
\end{definition}

The following properties are satisfied and can be proved similarly to \cite{Gorsky_thesis}; details will appear in a more general setting in \cite{Gorsky2024}.

\begin{itemize}
\item The algebras are well-defined (under the respective finiteness conditions), that is, the products do not depend on the choice of representatives in $\cb$ of elements in $M$.
\item The algebras $\ch(\cb, \cn)$ and $\ch'(\cb, \cn)$ are associative and unital whenever they are well-defined (i.e. under assumptions of the corresponding definitions \ref{def:subcat+RLF}, resp. \ref{def:subcat+LRF}).
\item Either of the algebras is a free module over the twisted group algebra of $K_0(\cn)/\Kerr(i_{\cn})$, where the twist is given by the Euler pairing. Each choice of representatives in $M$ of isomorphism classes in $\ct$ gives a basis of this module. In fact, in the setting of $\ch'(\cb, \cn)$, we have
\[
[z] \cdot [x] = q^{-\langle z, x \rangle_{\cb}} [z \oplus x]; \qquad [x] \cdot [z] = q^{-\langle x, z \rangle_{\cb}} [z \oplus x];
\]
and in the setting of $\ch(\cb, \cn)$, we have
\[
[z] \cdot [x] = q^{\langle z, x \rangle_{\cb}} [z \oplus x]; \qquad [x] \cdot [z] = q^{\langle x, z \rangle_{\cb}} [z \oplus x],
\]
for $z \in \cn, x \in \cb$.
\item Either of the algebras is a twisted flat $K_0(\cn)/\Kerr(i_{\cn})$--deformation of $\ch(\ct)$, whenever the latter is well-defined.
\item For given $\cn, \cb$ both finiteness conditions are satisfied if and only if for all $x, y \in \cb$, we have 
$\Ext^i_{\cb}(x, y) = \Ext^i_{\ct}(x, y) = 0$ for $|i| \gg 0$.
Then the Euler pairings
$\langle -, ? \rangle_{\cb}$
and $\langle -, ? \rangle_{\ct}$ are well-defined on $K_0(\cb) \times K_0(\cb)$ and  $K_0(\ct) \times K_0(\ct)$ respectively, and the algebra $\ch(\cb, \cn)$ is isomorphic to the twist of the algebra $\ch'(\cb, \cn)$ by $q^{\langle -, ? \rangle_{\cb/\ct, \Z}+ \langle -,?\rangle_{\cb}} = q^{(2 \langle -, ? \rangle_{\cb} - \langle -,?\rangle_{\ct})}$.

\item An exact functor $F: \cb_1 \to \cb_2$ which restricts to a fully faithful functor $\cn_1 \to \cn_2$ and induces a fully faithful functor $\ct_1 \coloneqq \cb_1/\cn_1\to \cb_2 / \cn_2 =: \ct_2$ defines a homomorphism $\ch(\cb_1, \cn_1)_{\scriptscriptstyle{\mathrm{tw}}} \to \ch(\cb_2, \cn_2)$ where the twist is given by 
$q^{\langle z, x \rangle_{\cb_2/\cb_1}}$, whenever these two algebras are well-defined. A similar statement holds for $\ch'(\cb_1, \cn_1)_{\scriptscriptstyle{\mathrm{tw}}} \to \ch'(\cb_2, \cn_2)$, whenever these two algebras are well-defined. 
If the induced map of monoids $M_1 \to M_2$ is an isomorphism and the induced functor $\ct_1 \to \ct_2$ is an equivalence, then the respective algebra homomorphism is an isomorphism. If $F$ is fully faithful, then the twist is trivial.
\item Assume that we are in the setting of Definition \ref{def:subcat+LRF} and we have $[z] = [\widetilde{z}] + [z^{\cn}], [x] = [\widetilde{x}] + [x^{\cn}]$, where $z^{\cn}, x^{\cn} \in \cn$ and the map  $\Ext^i_{\cb}(\widetilde{z}, \widetilde{x}) \to \Ext^i_{\ct}(\widetilde{z}, \widetilde{x})$ is surjective for $i = 1$ and bijective for $i \geq 2$.
Then we have
\begin{equation}
\label{eq:replacements-LRHF}
[z] \cdot [x] = q^{\langle z^{\cn}, \widetilde{z} \rangle_{\cb}  + \langle x^{\cn}, \widetilde{x} \rangle_{\cb} + \langle x^{\cn}, \widetilde{z} \rangle_{\cb} - \langle \widetilde{z}, x^{\cn} \rangle_{\cb} -\langle z^{\cn}, x^{\cn} \rangle_{\cb} - \langle \widetilde{z}, \widetilde{x} \rangle_{\cb, \leq 0} } [z^{\cn} \oplus x^{\cn}] 
\sum_{\delta \in \Hom_{\cb}(\widetilde{z}[-1], \widetilde{x})} [C(\delta)].
\end{equation}
\item Assume that we are in the  setting of Definition \ref{def:subcat+RLF} and we have $[z] = [\widetilde{z}] + [z^{\cn}], [x] = [\widetilde{x}] + [x^{\cn}]$, where $x^{\ca}, z^{\cn} \in \cn$ and the map  $\Ext^i_{\cb}(\widetilde{z}, \widetilde{x}) \to \Ext^i_{\ct}(\widetilde{z}, \widetilde{x})$ is surjective for $i = 1$ and bijective for $i \leq 0$.
Then we have
\begin{equation}
\label{eq:replacements-RLHF}
[z] \cdot [x] = q^{\langle z^{\cn}, \widetilde{z} \rangle_{\cb} + \langle x^{\cn}, \widetilde{x} \rangle_{\cb} + \langle x^{\cn}, \widetilde{z} \rangle_{\cb} - \langle \widetilde{z}, x^{\cn} \rangle_{\cb} - \langle z^{\cn}, x^{\cn} \rangle_{\cb} + \langle \widetilde{z}, \widetilde{x} \rangle_{\cb, > 0}} [z^{\cn} \oplus x^{\cn}]
\sum_{\delta \in \Hom_{\cb}(\widetilde{z}[-1], \widetilde{x})} [C(\delta)].
\end{equation}
\item Classes of elements of the two-sided kernel of the Euler pairing inside $K_0(\cn)/\Kerr(i_{\cn})$ (i.e. classes $[u]$ of objects $u \in \cn$ such that $\left\langle u, v \right\rangle = \left\langle v, u \right\rangle = 0,  \quad \forall v \in \cb$) are central, and so one is allowed to take \emph{numerical} central reductions by specializing them to 1. 
\end{itemize}

\subsubsection{Calabi--Yau conditions}
\label{sec:CY_triples+Kontsevich}

\begin{definition} \label{def:CY-triples}
Following Iyama--Yang \cite{IyamaYang}, we say that categories  $(\cb, \cn, \cm)$, with $\cn, \cm \subseteq \cb$ subcategories of $\cb$, $\cn$ being thick, form a \textit{$(n+1)$-Calabi--Yau triple} for some integer $n \geq 1$ if the following conditions are satisfied, combined with our assumption on $\cb$ being $\Hom$-finite, essentially small, and idempotent complete:
\begin{itemize}
    \item There is a bifunctorial isomorphism
    \[
     \Hom(x, y)^{\vee} \cong \Hom(y, x[n+1]),
    \]
    for any $x \in \cn, y \in \cb$; 
    \item $\cm$ is a \emph{silting} subcategory of $\cb$, meaning that the minimal thick subcategory of $\cb$ containing $\cm$ is $\cb$ itself, and $\Ext^i(\cm, \cm) = 0, \forall i > 0$;
    \item $\cm$ is a \emph{dualizing $\mathbf k$-variety}, in the sense of \cite{AuslanderReiten}; equivalently, it has a Serre functor;
    \item $\cb$ admits a $t$-structure right adjacent to the weight structure \cite{Bondarko} (also known as  co-$t$-structure \cite{Pauksztello})) $\cm$,
    $(\cm[<0]^{\perp_{\cb}}, \cm[>0]^{\perp_{\cb}})$, with $\cm[>0]^{\perp_{\cb}} \subseteq \ca$.
\end{itemize}
\end{definition}

It is proved in \cite[Theorem 4.10]{IyamaYang} that these conditions imply that $\cb$ also admits a  $t$-structure left adjacent to $\cm$. Further, the quotient category $\ct := \cb/\cn$ is $\Hom$-finite and there is a bifunctorial isomorphism
    \[
    \Hom_\ct(x, y)^{\vee} \cong \Hom_\ct(y, x[n]),
    \]
for any $x, y \in \ct$, i.e. $\ct$ is a $n$-Calabi--Yau triangulated category.

\begin{proposition}
\label{prop:silting_CY_triple}
Assume that given a pair $\cn  \subseteq \cb$, there exists a subcategory $\cm \subseteq \cb$ such that $(\cb, \cn, \cm)$ is a $(n+1)$-Calabi--Yau triple as defined above. Then the pair $\cn  \subseteq \cb$ satisfies the conditions of Definition \ref{def:subcat+RLF}. In particular, the Hall algebra of the pair $\ch(\cn, \cb)$ is well-defined. Further, for each $z, x$ we can find decompositions such that the formula \eqref{eq:replacements-RLHF} applies.
\end{proposition}

\begin{proof}
The right local homological finiteness follows directly from the existence of a silting subcategory $\cm \subseteq \cb$, see \cite[Proposition 2.4]{AiharaIyama}. 

For any $x \in \cn$, $y \in \cb$, we also have $\Hom(x, y[-i])^{\vee} \cong  \Hom(y[-(n+1)], x[i])$, 
and so  $\Ext^{-i}(x, y) = 0$ for $i \gg 0$. Similarly, we have $\Hom(y, x[-i]) \cong \Hom(x([-(n+1)], y[i])^{\vee}$; therefore,  $\Ext^{-i}(y, x) = 0,$ for $i \gg 0$. To summarize, for $x \in \cn, y \in \cb$, we have 
$\Ext^i_{\cb}(x, y) = \Ext^i_{\cb}(y, x) = 0$ for $|i| \gg 0$.

To show that for each $x, y \in \cb$, the map $\Ext^i_{\cb}(x, y) \to \Ext^i_{\ct}(x, y)$ is a bijection for $i \ll 0$, we consider the following subcategory $\cf \subseteq \cb$ called the \emph{fundamental domain} \cite{IyamaYang}:
\[
\cf := \cm \ast \cm[1] \ast \cdots \ast \cm[n-1] = \cb_{\geq 1 - n} \cap \cb_{\leq 0}.
\]
It is proved in \cite[Proposition 5.9]{IyamaYang} that the canonical map $\Hom_{\cb}(x,y) \to \Hom_{\ct}(x, y)$ is a bijection for any $x \in \cb_{\leq 0}, y \in  \cb_{\geq 1 - n}$. By definition of weight structures, it follows that for such $x, y$, the map $\Hom_{\cb}(x,y[-i]) \to \Hom_{\ct}(x, y[-i])$ is a also a bijection for each $i \geq 0$. It follows that for all $x, y \in \cf$, the map $\Ext^{i}_{\cb}(x,y) \to \Ext^{i}_{\ct}(x,y)$ is bijective for all $i \leq 0$.

By \cite[Theorem 5.8.(b)]{IyamaYang}, the functor $\cb \to \ct$ restricts to an equivalence of additive categories $\cf \overset{\sim}\to \ct$. The proof of density in loc.cit. shows that each object $x$ in $\cb$ is connected to an object $x^{\cf}$ of $\cf$ via two explicitly constructed triangles with one term belonging to $\cn$ in either of them. Given a pair of objects $x, y$, it now follows from four long exact sequences, two in the first argument and two in the second, that $\Ext^{i}_{\cb}(x,y) \cong \Ext^{i}(x^{\cf}, y^{\cf})$ for $i \ll 0$. Combined with the previous paragraph, this gives the desired bijectivity of the map  $\Ext^{i}_{\cb}(x,y) \to \Ext^{i}_{\ct}(x,y)$ for $i \ll 0.$

Finally, for the decomposition sufficient for the formula \eqref{eq:replacements-RLHF} to apply, we use the same triangles, only shift those in the second argument by 1. In other words, we have $\widetilde{x} = x^{\cf}$ and $\widetilde{y} = ((y[-1])^{\cf})[1]$. 
\end{proof}

\begin{remark}
Note that only the existence of a silting subcategory $\cm$ is used, not the choice of $\cm$. In particular, the algebra $\ch(\cb, \cn)$ does not depend on the choice of $\cm$ satisfying the conditions of Definition \ref{def:CY-triples}.
\end{remark}

As explained in \cite[Remark 5.3]{IyamaYang}, $(\cd, \cn, \cm)$ is an $(n+1)$-CY triple, $\cb$ is algebraic, and $\cm$ is the additive closure of a silting object, then there exists a dg-algebra $A$ satisfying the following conditions:

\begin{itemize}
\item There is a triangulated equivalence $\cb \to \perf A$ which restricts to an equivalence $\ca \to \pvd A$, where $\pvd A$ is the full subcategory of $\cd A$ consisting of dg-modules whose underlying dg $\mathbf k$-module is perfect;
\item $H^i(A) = 0$ for $i > 0$;
\item $H^0(A)$ is finite-dimensional,
\end{itemize}
together with the bifunctorial Calabi--Yau pairing between $\perf A$ and its full subcategory $\pvd A$. Note that this last assumption on the CY pairing holds if $A$ is \emph{left} (or \emph{smooth}) Calabi--Yau.
This setup recovers the construction of generalized higher cluster categories \cite{Guo}, which in turn generalized \cite{Amiot}.

\begin{corollary}
Let $A$ be a dg-algebra satisfying the above conditions. Then the pair $\pvd A \subseteq \perf A$ satisfies the conditions of Definition~\ref{def:subcat+RLF}. In particular, the Hall algebra of the pair $\ch(\perf A, \pvd A)$ is well-defined. Further, for each $z, x$ we can find decompositions such that the formula \eqref{eq:replacements-RLHF} applies.
\end{corollary}

\begin{remark}
We note that the construction of decompositions uses objects in $\cf[1]$. It thus appears necessary to use the ambient category $\cb$ or at least its subcategory $\cf \ast \cf[1]$ and not only the $\infty$-categorical structure of the nerve of $\cf$ in order to define $\ch(\cb, \cn)$.
\end{remark}

\begin{proposition}
A quasi-isomorphism $F$ of dg-algebras $A \to B$ satisfying the above conditions induces an isomorphism of Hall algebras $\ch(\perf A, \pvd A) \overset\sim\to \ch(\perf B, \pvd B)$.
\end{proposition}

\begin{proof}
The quasi-isomorphism $F$ induces an equivalence $\perf A \to \perf B$ which restricts to an equivalence $\pvd A \to \pvd B$. The result then follows from the general functoriality property.
\end{proof}

\begin{example}
As mentioned above, the construction applies to (higher generalized) cluster categories. A prototypical case is given by the cluster category of a Jacobi-finite quiver with potential. In this case, the dg-algebra $A$ in question is the so-called Ginzburg dg-algebra $\Gamma_{Q, W}$. Such categories provide additive categorifications of cluster algebras with principal coefficients. We do not give the precise definitions, instead referring the reader to the survey \cite{Keller2012} on these and related topics. 
\end{example}

For a Ginzburg dg-algebra $\Gamma_{Q, W}$, the simple dg-modules $S_i$ are $3$-spherical. They correspond to vertices of the quiver $Q$. It is proved in \cite{KellerYang} that the mutations of quivers with potential in the sense of Derksen--Weyman--Zelevinsky \cite{DWZ} admit a categorification in terms of the derived category $\cD(\Gamma_{Q, W})$. In fact, for each mutation $\mu_i: (Q, W) \to (Q', W')$, there exists a pair of equivalences $\Phi_{\pm}: \cD(\Gamma_{Q', W'}) \overset\sim\to \cD(\Gamma_{Q, W})$ related by a spherical twist with respect to $S_i$. Further, these equivalences restrict to equivalences $\perf(\Gamma_{Q', W'}) \overset\sim\to \perf(\Gamma_{Q, W})$ 
and
$\pvd(\Gamma_{Q', W'}) \overset\sim\to \pvd(\Gamma_{Q, W})$. This immediately implies the following.

\begin{corollary}
The mutation of quivers with potentials induces an isomorphism of Hall algebras $\ch(\perf \Gamma_{Q, W}, \pvd \Gamma_{Q, W})$.
\end{corollary}

In other words, the algebra $\ch(\perf \Gamma_{Q, W}, \pvd \Gamma_{Q, W})$ is an invariant of a mutation class of (Jacobi-finite) quivers with potential.

The spherical twists with respect $S_i$'s generate a braid group action by autoequivalences on $\cD(\Gamma_{Q, W})$ which restricts onto $\perf(\Gamma_{Q, W})$ and $\pvd(\Gamma_{Q, W})$. This implies the following.

\begin{proposition}
The braid group generated by spherical twists with respect to the simple $\Gamma_{Q, W}$-modules acts on $\ch(\perf \Gamma_{Q, W}, \pvd \Gamma_{Q, W})$.
\end{proposition}

The above construction works for a pair $\cn \subseteq \cb$ with a version of the $(n+1)$-Calabi--Yau condition for some positive $n$ and a certain extra piece of data given by a silting subcategory $\cm$. There is a dual analogue of this for $n+1 \leq 0$ developed by H.~Jin in \cite{Jin}. There, instead of a silting subcategory, one requires the existence of a \emph{simply-minded collection} and a pair of weight structures on $\cb$, mutually compatible in a certain precise sense. It is proved that in presence of such data, the Verdier quotient $\ct = \cb/\cn$ is realizable, as an additive category, as an ideal quotient of a suitable subcategory of $\cb$, which is an analogue of the fundamental domain $\cf$ in the positive CY setting. By replacing the arguments of \cite{IyamaYang} used in the proof of Proposition \ref{prop:silting_CY_triple} by their counterparts in this negative CY setting, namely by \cite[Theorem 1.1 and its proof]{IyamaYang2} and \cite[Lemma 2.6.(2), Theorem 4.5.(1)--(2)]{Jin}, we get the following. 

\begin{proposition}
\label{prop:triple_negative}
Assume that given a pair $\cn  \subseteq \cb$, there exists a simple-minded collection $\cS \subseteq \cb$ such that $(\cb, \cn, \cS)$ is an $(n+1)$-Calabi--Yau triple for some $n \leq -1$, in the sense of \cite{Jin}. Then the pair $\cn  \subseteq \cb$ satisfies the conditions of Definition \ref{def:subcat+LRF}. In particular, the Hall algebra of the pair $\ch(\cb, \cn)$ is well-defined. Further, for each $z, x$ we can find decompositions such that the formula \eqref{eq:replacements-LRHF} applies.
\end{proposition}

\begin{example}
Let $n \leq -1$, $B$ be a finite-dimensional hereditary algebra over $\mathbf{k}$, and consider the trivial extension dg-algebra $A = B \oplus B^\vee[-n-1]$. Then $A$ is a non-positive proper Gorenstein dg algebra, and it is proved in \cite{Jin, Jin2020} that the triple $(\pvd A, \perf A, \cS)$ is an $(n+1)$-Calabi--Yau triple, where $\cS$ is the set of simple dg A-modules concentrated in degree $0$. Thus, Proposition \ref{prop:triple_negative} applies. Further, the category $\pvd A / \perf A$ is equivalent as (an algebraic) triangulated category to the $n$-cluster category $\cc_{n}(B) \coloneqq \cd^b(\mod B)/(\nu \circ [-n])$ of $B$, see the discussion in \cite[Section 6]{Jin2020} following \cite{Keller2005}.

More generally, for a non-positively graded finite-dimensional symmetric algebra $A$ with socle concentrated in degree $n+1$ for some $n \leq -1$, the same triple $(\pvd A, \perf A, \cS)$ is an $(n+1)$-Calabi--Yau triple  and the category $H^0(\pvd A / \perf A)$ is $n$CY as a triangulated category \cite{Brightbill2020, Jin}.
\end{example}

We now discuss 
a construction we learned from M.~Kontsevich. It is a variation of the construction of $\cS\cd\ch_2(\ca)$, but does not use abelian categories of complexes in any way. This example played an important role in our finding of a correct definition of $\hallcy(\cc)$.

\begin{example} 
\label{ex:Kontsevich}
\cite{Kontsevich_private}
Let $\cc$ be a pre-triangulated dg-category over the graded polynomial algebra $\mathbf{k}[t]$, with the generator $t$ in some positive degree. Assume that for all $x, y \in \cc$, $\Hom_{\cc}(x, y)$ is a finite rank free $\mathbf{k}[t]$-module. We consider the tensor product categories $\cc_0 \coloneqq \cc \otimes_{\mathbf{k}[t]} \mathbf{k}$ and $\cc_{gen} \coloneqq \cc \otimes_{\mathbf{k}[t]} \mathbf{k}[t^{\pm 1}]$.
We have a canonical functor $\cc \to \cc_{gen}$ and denote by $\cc_{tors} \subseteq \cc$ its kernel. It is not difficult to check that the pair $(H^0(\cc), H^0(\cc_{tors}))$ satisfies the conditions of Definition \ref{def:subcat+LRF}. Thus, the algebra $\ch'(H^0(\cc), H^0(\cc_{tors}))$ is well-defined. Moreover, for each pair $(z, x)$, we can find decompositions such that the formula \eqref{eq:replacements-LRHF} applies.

As an example, we can take $\cc = \cd^b_{dg}(\ca)\otimes_{\mathbf k}\perf \mathbf{k}[t]$ for an abelian category $\ca$ such that for all $x, y$, we have $\Ext^i(x, y) = 0$ for $i \gg 0$,  and take $t$ having degree 2. In this case, the Hall algebra $\ch'(H^0(\cc), H^0(\cc_{tors}))$ is an extended algebra of $H^0(\cd^b_{dg}(\ca) \otimes \perf \mathbf{k}[t^{\pm 1}])$, i.e. of the root category of $\ca$. This extended algebra is by construction derived invariant, since the multiplication uses $H^0(\cc_0) = \cd^b(\ca)$ as the input. 
\end{example}

We note that if $\cc_0$ is $n$CY, then so is $\cc_{gen}$, and we also have a $(n-1)$CY
pairing between $H^0(\cc_{tors})$ and $H^0(\cc)$. 

\begin{remark}
The difference between Example \ref{ex:Kontsevich} and the cases of Calabi--Yau triples involving silting subcategories or simply-minded collections is that we do not realize the underlying additive category of $H^0(\cc_{gen})$ as a subquotient of the underlying additive category of $H^0(\cc)$, i.e. we do not appeal to the existence of fundamental domains. For this reason, decompositions of $[z], [x]$ used in the formula \eqref{eq:replacements-LRHF} are defined ``locally'', i.e. they actually depend on a pair $(z, x)$, and we do not have a decomposition for a given $z$ which works for all $x$.
\end{remark}

\subsubsection{Central reduction and relation to $\hallcy$}

Here, we show that a central reduction of a suitable twist of $\ch'(\cb, \cn)$ or $\ch(\cb, \cn)$ is isomorphic to $\hallcy(\ct)$, provided $\cb$ is algebraic (and, therefore, so are $\cn$ and $\ct$ as well) and $\ct$ is $n$CY for some odd $n$. This applies to all the $(n+1)$-Calabi--Yau triples and to setting of Example \ref{ex:Kontsevich} from Section \ref{sec:CY_triples+Kontsevich}

We assume from now on that $n \geq 1$ and $\cn \subseteq \cb$ satisfy the condition of Definition \ref{def:subcat+LRF}; the same arguments apply and the similar statements hold true in the other three cases. The strategy is similar to that in Section \ref{sec:root_categories}.

\begin{lemma}
\label{lem:twist_extended_Verdier}
The form 
\[
\langle z, x \rangle_{\cb, \leq 0}  + \langle z, x \rangle_{\cb/\ct, > 0}  -\langle x, z \rangle_{\cb, \leq 0} - \langle x, z \rangle_{\cb/\ct, > 0} + \langle z, x \rangle_{1,\ldots,n-1}
\]
where the last term is the truncated Euler pairing in the category $\ct$,
is well-defined on $K_0(\cb) \times K_0(\cb)$.
\end{lemma}

\begin{proof}
The same proof as for Lemma \ref{lem:our_twist_sdh} applies here, up to replacing $\cc_2(\ca)$ by $\cb$, $\cd_2(\ca)$ by $\ct$, short exact sequences by distinguished triangles, and $\dim \Hom_{\cc_2(\ca)}(-,?)$ by $\langle -, ? \rangle_{\cb, \leq 0}$.
\end{proof}

Let $(\ch'(\cb, \cn)_{\scriptscriptstyle{\mathrm{tw}}}, \ast)$ be the following twist of the algebra $(\ch'(\cb, \cn) \otimes_{\QQ} \C, \cdot)$:
\[
[z] \ast [x] = q^{\frac{1}{2}\left(\langle z, x \rangle_{\cb, \leq 0} + \langle z, x \rangle_{\cb/\ct, > 0}  -\langle x, z \rangle_{\cb, \leq 0} - \langle x, z \rangle_{\cb/\ct, > 0} + \langle z, x \rangle_{1,\ldots,n-1}\right)} [z] \cdot [x].
\]

The twist is indeed well-defined by Lemma \ref{lem:twist_extended_Verdier}.

\begin{lemma}
The classes of elements in $\cn$ are central in  $(\ch'(\cb,\cn)_{\scriptscriptstyle{\mathrm{tw}}}, \ast)$.
\end{lemma}

\begin{proof}
By definition, for $z \in \cn$, we have 
\begin{align*}
[z] \ast [x] &= q^{-\langle z, x \rangle_{\cb} + \frac{1}{2}\left(\langle z, x \rangle_{\cb, \leq 0} + \langle z, x \rangle_{\cb/\ct, > 0}  -\langle x, z \rangle_{\cb, \leq 0} - \langle x, z \rangle_{\cb/\ct, > 0} + \langle z, x \rangle_{1,\ldots,n-1}\right)} [z \oplus x] \\
&= q^{-\frac{1}{2}\left(\langle z, x \rangle_{\cb, \leq 0}+\langle x, z \rangle_{\cb, \leq 0}\right)} [z \oplus x]\\
&= q^{-\langle x, z \rangle_{\cb} + \frac{1}{2}\left(\langle x, z \rangle_{\cb, \leq 0} + \langle x, z \rangle_{\cb/\ct, > 0}  -\langle z, x \rangle_{\cb, \leq 0} - \langle z, x \rangle_{\cb/\ct, > 0} + \langle x, z \rangle_{1,\ldots,n-1}\right)} [z \oplus x] = [x] \ast [z].
\end{align*}
\end{proof}

\begin{remark}
In case $n$ is even (and so $(n+1)$ is odd), one can check in similar way that the form
\[
\langle z, x \rangle_{\cb, \leq 0}  + \langle z, x \rangle_{\cb/\ct, > 0}  +\langle x, z \rangle_{\cb, \leq 0} + \langle x, z \rangle_{\cb/\ct, > 0} + \langle z, x \rangle_{1,\ldots,n-1}
\]
is well-defined on $K_0(\cb) \times K_0(\cb)$. However, in a suitably twisted algebra, the classes of elements in $\cn$ are not central in general, as so one cannot take a central reduction with respect to them. 
This is one of the reasons we doubt that there may exist a general construction of non-extended intrinsic Hall algebras of even-periodic evenCY categories.
\end{remark}

\begin{theorem}
\label{thm:extended_Verdier_comparison}
Assume that $\cn, \cb, \ct$ are algebraic, $\ct$ is $n$CY for some odd $n \geq 1$, and $\ch'(\cb, \cn)$ is well-defined.
Then $\hallcy(\ct)$ is isomorphic to the central reduction:
\[
\ch'(\cb,\cn)_{\scriptscriptstyle{\mathrm{tw}}} / \langle [z] - 1 \, | \, z \in \cn \rangle \overset\sim\to \hallcy(\ct).
\]
\end{theorem}

\begin{proof}
The same proof as for Theorem \ref{thm:root_cat_comparison} applies here, up to the same modification as in the proof of Lemma \ref{lem:twist_extended_Verdier}.
\end{proof}

\begin{remark}
If we rescale the basis of $\ch'(\cb,\cn)_{\scriptscriptstyle{\mathrm{tw}}}$ by setting 
\[
[[x]] = q^{-\frac{1}{2}(\langle z, x \rangle_{\cb, \leq 0} + \langle z, x \rangle_{\cb/\ct, > 0})} [x],
\]
the central reduction can be rewritten as 
\[
\ch'(\cb,\cn)_{\scriptscriptstyle{\mathrm{tw}}} / \langle [[x \oplus z]] - [[z]] \, | \, z \in \cn, x \in \cb \rangle.
\]
\end{remark}

\begin{remark}
A different construction of extended Hall algebras was introduced in \cite{Gorsky2018} as a generalization of \cite{bridgeland_quantum}.  Let $\cf$ be an essentially small, idempotent-complete, $\Hom$-finite Frobenius exact category, with $\underline{\cf}$ its stable category (which is algebraic triangulated by definition) and $\cp$ the full subcategory of projective-injective objects. It was then shown that the localization $\ch(\cf)[[p]^{-1} \, | \, p \in \cp]$, called the \emph{semi-derived Hall algebra} just as algebras in Section \ref{sec:root_categories} and denoted $\cS\cd\ch(\cf, \cp)$, satisfies many properties similar to those of algebras $\ch'(\cb, \cn)$ and $\ch(\cb, \cn)$ (a common generalization of these constructions will be introduced in \cite{Gorsky2024}). It can thus also be seen an extended Hall algebra of $\underline{\cf}$. Arguments similar to those in the proof of Theorem \ref{thm:root_cat_comparison} also show the following.
\end{remark}

\begin{proposition}
Assume that $\underline{\cf}$ is $n$CY for some odd $n$. Then we have an isomorphism 
\[
\cS\cd\ch(\cf, \cp)_{\scriptscriptstyle{\mathrm{tw}}} / \left\langle m_p - 1 \, | \, p \in \cp \right\rangle \overset\sim\to \hallcy(\underline{\cf}),
\]
where $(\cS\cd\ch(\cf, \cp)_{\scriptscriptstyle{\mathrm{tw}}}, \ast)$ is the twist of $(\cS\cd\ch(\cf, \cp) \otimes_{\QQ} \C, \cdot)$ by setting 
\[
m_z \ast m_x = \sqrt{\frac{|\Hom_{\cf}(z, x)|}{|\Hom_{\cf}(x, z)|}} q^{\langle z, x \rangle_{1,\ldots, n-1}} m_z \cdot m_x
\]
if $n > 1$, and the similar twist in case $n \leq -1$.
\end{proposition}
} 

\section{CY 2-spans and their homotopy cardinality}
\label{sec_2spans}

In this section we categorify the constructions of Section~\ref{sec_hcardcy}.
In particular, 1-categories whose morphisms are 1-spans are extended to 2-categories whose 2-morphisms are 2-spans.
In Section~\ref{subsec:cy2spans} we discuss Calabi--Yau structures on 2-spans of dg-categories.
Then, in Section~\ref{subsec_lin2cat}, we construct linearization functors $\hdff$ and $\hdffx$ categorifying $\hdf$ and $\hdfx$, respectively.

All the dg-categories in this section are assumed to be {\degfin}.

\subsection{Calabi--Yau 2-spans}
\label{subsec:cy2spans}

Let $\spancat_2$ be the (weak) 2-category whose objects are locally finite homotopy types, 1-morphisms are 1-spans of such, and 2-morphisms are equivalence classes of 2-spans as in \eqref{diag:2span}.
(An $(\infty,2)$-categorical version of $\spancat_2$ was constructed in~\cite[Chapter 10]{DK_highersegal}, though for our purposes its homotopy 2-category suffices.)
Our aim in this subsection is to give a variant of this construction for CY dg-categories.

The notion of relative Calabi--Yau structure was generalized by Christ--Dyckerhoff--Walde~\cite{CDW_complexes} to complexes and cubical diagrams of linear stable $\infty$-categories. 
Here, we will be interested in the special case of (weak) CY structures on 2-spans, and that they are the 2-morphisms of an appropriate 2-category.

Assuming $\cc$ has binary products, we can fold a diagram of the form \eqref{diag:2span} to a square
\[
\begin{tikzcd}
    \cc \arrow[r,"g_1"]\arrow[d,"g_2"'] & \cb_1 \arrow[d,"{(f_{11},f_{21})}"]\\   
   \cb_2 \arrow[r,"{(f_{12},f_{22})}"'] & \ca_1{\times} \ca_2  
\end{tikzcd}
\]
and, assuming $\cc$ has a terminal object, we can send a general coherent square to a 2-span with $\ca_2=*$.
Thus, we can often restrict to the special case of squares instead of general 2-spans.

\begin{definition}
    A \defword{(weak) $n$-Calabi--Yau structure} on a coherent square
    \begin{equation}\label{diag:CYsquare}
    \begin{tikzcd}
        \cc \arrow[r,"g_1"]\arrow[d,"g_2"'] & \cb_1\arrow[d,"f_1"] \\
        \cb_2 \arrow[r,"f_2"'] & \ca
    \end{tikzcd}
    \end{equation}
    is given by a weak right $(n-2)$CY structure on $\ca$, weak right $(n-1)$CY structures on $f_1$ and $f_2$, and a weak right $n$CY structure on $(g_1,g_2):\ce\to \cb_1{\times}_{\ca}\cb_2$, which are all compatible.
    In detail, this means we have an equivalence $\alpha:\ca[n-2]^\vee\to\ca$ and bicartesian squares of bimodules
    \begin{equation}
        \begin{tikzcd}
            \cb_i \arrow[r,"f_i"]\arrow[d] & f_i^*\ca \arrow[d,"f_i^\vee\circ(\alpha^{-1})^\vee"]\\ 
            0 \arrow[r] & \cb_i[n-2]^\vee 
        \end{tikzcd}
        \qquad
        \begin{tikzcd}
            \cc \arrow[r,"{(g_1,g_2)}"]\arrow[d] & (g_1,g_2)^*\left(\cb_1{\times}_{\ca}\cb_2\right) \arrow[d,"{(g_1^\vee,g_2^\vee)\circ(\beta^{-1})^\vee}"] \\
            0 \arrow[r] & \cc[n-1]^\vee
        \end{tikzcd}
    \end{equation}
    where the CY structure $\beta:\left(\cb_1{\times}_{\ca}\cb_2\right)[n-1]^\vee\to \cb_1{\times}_{\ca}\cb_2$ is provided by Corollary~\ref{cor:CYspancomp}.
\end{definition}

In the remainder of this subsection we will show that the CY property is preserved under vertical composition of 2-spans and composition of a 2-span and a 1-span (``whiskering'').
We note that in any weak 2-category, horizontal composition can be recovered, up to equivalence, from vertical composition and whiskering (horizontal composition with identity 2-morphisms) as follows: If $f,f':a\to b$ and $g,g':b\to c$ are 1-morphisms and $\alpha:f\to f'$ and $\beta:g\to g'$ are 2-morphisms, then the horizontal composition of $\alpha$ and $\beta$ is $\beta *\alpha\cong(\beta * 1_{f'})(1_g * \alpha)$.

\begin{lemma}\label{lem:FuseTripple}
Suppose we have a diagram of dg-categories
\[
\begin{tikzcd}
    \cb_1 \arrow[r,"f_1"] & \ca_1 & \cb_2 \arrow[l,"f_2"']\arrow[r,"f_4"] & \ca_2 \\
          & \cb_3 \arrow[u,"f_3"]
\end{tikzcd}
\]
where $\ca_1$ and $\ca_2$ are $(n-2)$CY and $f_1$, $f_3$, and $(f_2,f_4)$ are $(n-1)$CY.
Then the natural functor
\[
\cb_1{\times_{\ca_1}}\cb_2{\times}_{\ca_1}\cb_3\longrightarrow \left(\cb_1{\times}_{\ca_1}\cb_3\right)\times\left(\left(\cb_1{\times}_{\ca_1}\cb_2\right){\times}_{\ca_2}\left(\cb_2{\times}_{\ca_1}\cb_3\right)\right)
\]
is $n$CY.
\end{lemma}

\begin{remark}
The lemma has a geometric analog, which provides some intuition for why it is true.
For simplicity we consider the special case $\ca_2=*$.
Let $B_i$, $i=1,2,3$, be compact oriented (topological) manifolds with boundary with given identification $\partial B_i\cong A$ for some fixed closed oriented manifold $A$.
Construct a manifold $T$ by gluing the three manifolds with corners $B_i\times [-1,1]$ together, where $\partial B_i\times [0,1]$ is glued to $\partial B_{i+1}\times [-1,0]$, $i\in\Z/3$.
Then 
\[
\partial T=(B_1\sqcup_A B_2)\sqcup (B_1\sqcup_A B_3)\sqcup (B_2\sqcup_A B_3)
\]
and $B_1\sqcup_A B_2\sqcup_A B_3$ is a deformation retract of $T$.
\end{remark}

\begin{proof}
We will use Roman letters to denote the pullback of the various diagonal bimodules to the category of $\cb_1{\times}_{\ca_1}\cb_2{\times}_{\ca_1}\cb_3$-bimodules.
First, we claim that in the coherent diagram
\begin{equation}\label{diag:FuseTripple1}
\begin{tikzcd}
B_1{\times}_{A_1}B_2{\times}_{A_1}B_3 \arrow[r]\arrow[d] & \left(B_1{\times}_{A_1}B_2\right){\times}_{A_2}\left(B_2{\times}_{A_1}B_3\right) \arrow[d] & \\
B_1{\times}_{A_1}B_3\arrow[d]\arrow[r] & \mathrm{fib}\left(B_1{\times} B_2{\times} B_3\to A_1{\times} A_2\right)\arrow[r]\arrow[d] & B_1{\times} B_3 \arrow[d] \\
0 \arrow[r] & \mathrm{fib}(B_2\to A_2) \arrow[r] & A_1
\end{tikzcd}
\end{equation}
all squares are bicartesian.
The bottom and left rectangles (combinations of two squares) are bicartesian by definition.
That the bottom right square is bicartesian can be seen from the following diagram:
\begin{equation}\label{diag:FuseTripple2}
\begin{tikzcd}
\mathrm{fib}\left(B_1{\times} B_2{\times} B_3\to A_1{\times} A_2\right)\arrow[r]\arrow[d] & B_1{\times} B_2{\times} B_3 \arrow[r]\arrow[d] & B_1{\times} B_3 \arrow[d] \\
\mathrm{fib}\left(B_2\to A_2\right)\arrow[r]\arrow[d] & A_1{\times} B_2\arrow[r]\arrow[d] & A_1 \\
0 \arrow[r] & A_1{\times} A_2
\end{tikzcd}
\end{equation}
Finally, we have a diagram
\begin{equation}\label{diag:FuseTripple3}
\begin{tikzcd}
B_1{\times} B_2{\times} B_3 \arrow[d]\arrow[r] & A_1^3{\times} A_2 \arrow[d]\arrow[r] & A_1{\times} A_2\arrow[d] \\
0 \arrow[r] & (B_1{\times} B_2{\times} B_3)[n-2]^\vee \arrow[r] & (B_1{\times}_{A_1} B_2{\times}_{A_1} B_3)[n-2]^\vee
\end{tikzcd}
\end{equation}
where the left square comes from the assumptions of the lemma (and is bicartesian) and the right square is the shifted dual of the bicartesian square
\begin{equation}\label{diag:FuseTripple4}
\begin{tikzcd}
B_1{\times}_{A_1}B_2{\times}_{A_1} B_3 \arrow[r]\arrow[d] & A_1 \arrow[d] \\
B_1{\times} B_2{\times} B_3 \arrow[r] & A_1^3
\end{tikzcd}
\end{equation}
which comes from the definition of $B_1{\times}_{A_1}B_2{\times}_{A_1} B_3$, together with the square containing the identity map of $A_2$.
From \eqref{diag:FuseTripple3} we conclude 
\[
\mathrm{fib}(B_1{\times} B_2{\times} B_3\to A_1{\times} A_2)\cong (B_1{\times}_{A_1} B_2{\times}_{A_1} B_3)[n-1]^\vee
\]
and inserting this back into \eqref{diag:FuseTripple1} gives the required bicartesian square.
\end{proof}

\begin{proposition}\label{prop:CYSquareVComp}
Given two $n$CY squares
\[
\begin{tikzcd}
    \cc_1 \arrow[r]\arrow[d] & \cb_2 \arrow[d] \\
    \cb_1 \arrow[r] & \ca
\end{tikzcd}\qquad
\begin{tikzcd}
    \cb_2 \arrow[d] & \cc_2 \arrow[l]\arrow[d] \\
    \ca  & \cb_3 \arrow[l]
\end{tikzcd}
\]
which are compatible in the sense that the two functors $\cb_2\to\ca$ together with their CY structure are equivalent, then the (vertically) composed square
\[
\begin{tikzcd}
    \cc_1{\times}_{\cb_2}\cc_2 \arrow[r]\arrow[d] & \cb_3 \arrow[d] \\
    \cb_1 \arrow[r] & \ca
\end{tikzcd}
\]
is naturally $n$CY.
\end{proposition}

\begin{proof}
We have the following coherent ``staircase'' diagram:
\begin{equation}\label{diag:staircase}
\adjustbox{scale=.9}{
\begin{tikzcd}
B_2[-1] \arrow[r]\arrow[d] & C_1{\times}_{B_2}C_2 \arrow[r]\arrow[d] & B_1{\times}_AB_2{\times}_AB_3 \arrow[r]\arrow[d] & B_1{\times}_AB_3 \arrow[d] \\
0 \arrow[r] & C_1{\times}C_2 \arrow[r]\arrow[d] & (B_1{\times}_AB_2)\times (B_2{\times}_AB_3) \arrow[r]\arrow[d] & (B_1{\times}_AB_2{\times}_AB_3)[n-1]^\vee \arrow[d] \\
 & 0 \arrow[r] & C_1[n-1]^\vee{\times}C_2[n-1]^\vee \arrow[r]\arrow[d] & (C_1{\times}_{B_2}C_2)[n-1]^\vee \arrow[d] \\
 & & 0 \arrow[r] & B_2[n-2]^\vee 
\end{tikzcd}
}
\end{equation}
We claim that all squares in \eqref{diag:staircase} are bicartesian.
\begin{itemize}[label={}]
    \item 
    $\ydiagram{3,1+2,2+1}*[*(black)]{1,3+0}$: From the definition of $C_1{\times}_{B_2}C_2$.   
    \item 
    $\ydiagram{3,1+2,2+1}*[*(black)]{3+0,3+0,2+1}$: Shifted dual of the previous square.
    \item 
    $\ydiagram{3,1+2,2+1}*[*(black)]{2,3+0}$: From the definition of $B_1{\times}_AB_2{\times}_AB_3$.
    \item 
    $\ydiagram{3,1+2,2+1}*[*(black)]{3+0,2+1,2+1}$: Shifted dual of the previous square and using the $(n-1)$CY structures on $\cb_1{\times}_{\ca}\cb_2$ and $\cb_2{\times}_{\ca}\cb_3$.
    \item 
    $\ydiagram{3,1+2,2+1}*[*(black)]{3+0,1+1}$: From the $n$CY structure on the two squares to be composed.
    \item 
    $\ydiagram{3,1+2,2+1}*[*(black)]{2+1,3+0}$: From Lemma~\ref{lem:FuseTripple}.
\end{itemize}

We conclude that the square $\ydiagram{3,1+2,2+1}*[*(black)]{1+2,1+2}$ exists and is bicartesian, which completes the proof.
\end{proof}

\begin{proposition}\label{prop:CYSquareWhisker}
Suppose we have a diagram
\[
\begin{tikzcd}
\cc \arrow[r]\arrow[d] & \cb_2 \arrow[d] & \cb_3 \arrow[dl]\arrow[dr] \\
\cb_1 \arrow[r] & \ca_1 & & \ca_2
\end{tikzcd}
\]
where the square is $n$CY and the span $\ca_1\xleftarrow{}\cb_3\to\ca_2$ is $(n-1)$CY.
Then the ``whiskered'' square
\[
\begin{tikzcd}
    \cc{\times}_{\ca_1}\cb_3 \arrow[r]\arrow[d] & \cb_2{\times}_{\ca_1}\cb_3 \arrow[d] \\
    \cb_1{\times}_{\ca_1}\cb_3 \arrow[r] & \ca_2
\end{tikzcd}
\]
is naturally $n$CY.
\end{proposition}

\begin{proof}
The right vertical and lower horizontal arrows in the whiskered square are CY by Proposition~\ref{prop:CYSpanComp}.
It remains to show that the functor 
\[
\cc{\times}_{\ca_1}\cb_3 \longrightarrow \left(\cb_1{\times}_{\ca_1}\cb_3\right){\times}_{\ca_2}\left(\cb_2{\times}_{\ca_1}\cb_3\right)
\]
is CY.
To see this consider the coherent diagram
\begin{equation}\label{whiskerRoofDiag}
\begin{tikzcd}[column sep=tiny]
 & & \cc{\times}_{\ca_1}\cb_3 \arrow[dl]\arrow[dr] & & \\
 & \cc \arrow[dl]\arrow[dr] & & \cb_1{\times}_{\ca_1}\cb_2{\times}_{\ca_1}\cb_3 \arrow[dl]\arrow[dr] & \\
 * & & \cb_1{\times}_{\ca_1}\cb_2 & & \left(\cb_1{\times}_{\ca_1}\cb_3\right){\times}_{\ca_2}\left(\cb_2{\times}_{\ca_1}\cb_3\right)
\end{tikzcd}
\end{equation}
where the lower-left span is CY by assumption, the lower-right span is CY by Lemma~\ref{lem:FuseTripple}, and the square is cartesian.
By Proposition~\ref{prop:CYSpanComp} the composed span is CY.
\end{proof}

\begin{proposition}\label{prop:idsquare}
If $f:\cb\to\ca$ is $(n-1)$CY, then the ``identity square''
\[
\begin{tikzcd}
\cb \arrow[d,"1"']\arrow[r,"1"] & \cb \arrow[d,"f"] \\
\cb \arrow[r,"f"'] & \ca
\end{tikzcd}
\]
is $n$CY, where the connecting map $\cb[n]^\vee\to\cb$ is a zero map.
\end{proposition}

\begin{proof}
This follows from the diagram of $\cb$-bimodules
\[
\begin{tikzcd}
\cb \arrow[r]\arrow[d] & \cb{\times}_{f^*\ca}\cb \arrow[r]\arrow[d] & \cb^2 \arrow[d] \\
0 \arrow[r] & \cb[n-1]^\vee \arrow[r]\arrow[d] & \cb \arrow[d] \\
 & 0 \arrow[r] &  f^*\ca 
\end{tikzcd}
\]
where the top rectangle, and thus both top squares, are split.
\end{proof}

The above propositions show that there is a (weak) 2-category $\spancat_2^{n\mathrm{CY}}$ with
\begin{itemize}
    \item Objects: {\degfin} dg-categories over $\mathbf k$ with $(n-2)$CY structure,
    \item Morphisms: $(n-1)$CY 1-spans,
    \item 2-morphisms: 2-spans admitting an $n$CY structure, up to equivalence.
\end{itemize}

We would like to emphasize that in the definition of 2-morphisms, we only require the existence of an $n$CY structure on the underlying 2-span. The choice of such a structure is not a part of the data, cf. Remark~\ref{rem: Weinstein} below.

Given that 2-spans without CY structure form a 2-category, the main thing to check is that vertical and horizontal composition of 2-morphisms preserve the property of admitting a CY structure. 
For vertical composition this is Proposition~\ref{prop:CYSquareVComp}. 
Horizontal composition can be written in terms of vertical composition and whiskering, so one can combine Propositions~\ref{prop:CYSquareVComp} and~\ref{prop:CYSquareWhisker}.
We will not go into more detail here, since on the one hand they will not be needed for our application in the next section, and on the other because of our somewhat artificial restriction to $k$-spans with $k=2$ and homotopy 2-categories instead of $(\infty,2)$-categories.

Throughout this subsection, a \textit{2-span} has been a coherent diagram of the type~\eqref{diag:2span}. Let us refer to such diagrams as \textit{lozenge diagrams}. Instead, one could also consider ``spans of spans'', i.e. \textit{window diagrams} of the following form:
\begin{equation}\label{diag:2spanWindow}
\begin{tikzcd}
     \ca_{11} & \cb_1 \arrow[l]\arrow[r] & \ca_{12} \\
    \cb_2 \arrow[u]\arrow[d] & \cc \arrow[u]\arrow[d]\arrow[l]\arrow[r] & \cb_3\arrow[u]\arrow[d] \\
    \ca_{21} & \cb_4 \arrow[l]\arrow[r] & \ca_{22}
\end{tikzcd}
\end{equation}
Indeed, we will encounter such diagrams in Section~\ref{sec:legendrian}.
There are several ways of turning such a diagram into a lozenge diagram by partially composing the \textit{boundary 1-spans} with apex $\cb_i$, $i=1,2,3,4$.
For example, composing the top span with the right span and the left span with the bottom span yields:
\begin{equation}\label{lozengeFromWindow}
\begin{tikzcd}
      & \cb_1 \times_{\ca_{12}}\cb_3 \arrow[dl]\arrow[dr] &  \\
     \ca_{11} & \cc \arrow[u]\arrow[d] & \ca_{22} \\
     &  \cb_2\times_{\ca_{21}} \cb_4 \arrow[ul]\arrow[ur] & 
\end{tikzcd}
\end{equation}
Alternatively, one could compose the left with the top spans and the right with the bottom span, or any three of the four spans.
For our purposes, it is important that the \textit{boundary categories} of all these diagrams are the same, as is made precise in the following lemma.

\begin{lemma}\label{lem:funToBoundary}
There is a canonical equivalence
\begin{equation*}
\lim\left(\adjustbox{scale=.9}{\begin{tikzcd}
     \ca_{11} & \cb_1 \arrow[l]\arrow[r] & \ca_{12} \\
    \cb_2 \arrow[u]\arrow[d] &  & \cb_3\arrow[u]\arrow[d] \\
    \ca_{21} & \cb_4 \arrow[l]\arrow[r] & \ca_{22}
\end{tikzcd}}\right) \cong
\lim\left(\adjustbox{scale=.9}{
\begin{tikzcd}
      & \cb_1 \times_{\ca_{12}}\cb_3 \arrow[dl]\arrow[dr] &  \\
     \ca_{11} &  & \ca_{22} \\
     &  \cb_2\times_{\ca_{21}} \cb_4 \arrow[ul]\arrow[ur] & 
\end{tikzcd}}\right) 
\end{equation*}
under $\cc$, which maps to both by the universal property.
Variants of this statement hold where the right-hand-side is replaced by the other lozenges obtained by composing some of the edges of the original window.
\end{lemma}

\begin{proof}
A special case of the fact that limits can be computed iteratively.
\end{proof}

\begin{remark} \label{rem: Weinstein}
After the first version of our paper was posted on the arXiv, the work~\cite{BCS}  
appeared, where an $(\infty,k)$-category of $n$CY $k$-spans is constructed.
More precisely, these authors mainly focus on the dual notion of \textit{left} CY \textit{cospans}, while the case of right CY spans is briefly treated in \cite[Appendix B]{BCS} (in the proper and strong CY case, but as sketched in \cite[Remark B.1]{BCS}, the same construction should work in the case of weak CY structures in a larger generality). The latter should in particular recover the 2-category $\spancat_2^{n\mathrm{CY}}$;
however, we still need the proofs and diagrams in this subsection for the discussion in the following Section~\ref{subsec_lin2cat}.
\end{remark}

\subsection{Linearized 2-category}
\label{subsec_lin2cat}

The plan of this subsection is to take $\spancat_2$ as a starting point and ``linearize its 2-morphisms'', that is, construct another 2-category, $L_2\spancat_2$, whose 2-morphisms between a pair of 1-morphisms form a vector space, together with a functor
\[
\hdff:\spancat_2\to L_2\spancat_2
\]
which is the identity on objects and 1-morphisms.
A variant of this which makes use of the formalism of Section~\ref{subsec_hcpercy} yields an analogous 2-functor, for even $n$,
\[
\hdffx:\spancat_2^{n\mathrm{CY}}\to L_2\spancat_2^{n\mathrm{CY}}
\]
which categorifies $\hdfx$.
Throughout this subsection we restrict for simplicity to spaces with finite $\pi_0$ and dg-categories with a finite number of isomorphism classes of objects. 

\begin{remark}
Instead of linearizing only the 2-morphisms, one can construct a functor $\spancat_2\to \vectcat_2(\mathbf k)$ to the (strict) 2-category of $\mathbf k$-linear categories, taking a space $X$ to its category of local systems, and which linearizes both 1-morphisms and 2-morphisms. 
For an $(\infty,2)$-categorical version of this functor see~\cite[Remark 4.2.5]{LurieHopkins}, where the authors partially prove 2-functoriality. 
The restriction of the functor to 1-groupoids was constructed by Morton~\cite{morton11}.
\end{remark}

\begin{definition}
    Define the 2-category $L_2\spancat_2$ whose objects and 1-morphisms are the same as those of $\spancat_2$ and whose 2-morphisms from a 1-span $B_1\to A_1{\times} A_2$ to a 1-span $B_2\to A_1{\times} A_2$ are $\hdf(B_1{\times}_{A_1{\times} A_2} B_2)$.
    The remaining structure of a 2-category is defined below.
\end{definition}

Vertical composition in $L_2\spancat_2$ is defined as follows.
Given a triple of 1-spans $B_i\to A_1{\times} A_2\eqqcolon A$, $i\in\{1,2,3\}$, we need a linear map
\[
\hdf(B_1{\times}_{A} B_2)\otimes \hdf(B_2{\times}_{A} B_3)\longrightarrow \hdf(B_1{\times}_{A} B_3)
\]
which we take to be the image under $\hdf$ of the 1-span
\[
B_1{\times}_{A} B_2\longleftarrow B_1{\times}_{A} B_2{\times}_A B_3\longrightarrow (B_2{\times}_{A} B_3){\times}(B_1{\times}_{A} B_3).
\]
Associativity of vertical composition follows from functoriality of $\hdf$.

Let us turn to horizontal composition (of which whiskering --- horizontal composition with identity 2-morphisms --- is a special case).
For a quadruple of 1-spans 
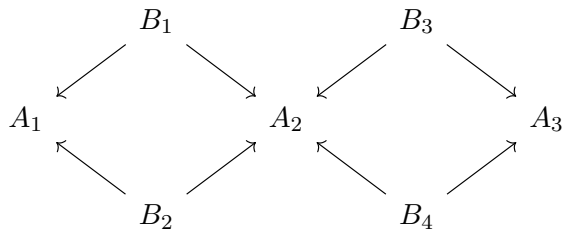
\begin{equation*}
\begin{tikzcd}
    & B_1 \arrow[dl]\arrow[dr] & & B_3 \arrow[dl]\arrow[dr] \\
    A_1 &  & A_2 & & A_3\\
    & B_2 \arrow[ul]\arrow[ur] & & B_4 \arrow[ul]\arrow[ur]
\end{tikzcd}
\end{equation*}
we need a linear map
\[
\hdf(B_1{\times}_{A_1{\times}A_2} B_2)\otimes \hdf(B_3{\times}_{A_2{\times}A_3} B_4)\longrightarrow \hdf\left((B_1{\times}_{A_2} B_3){\times}_{A_1{\times}A_3} (B_2{\times}_{A_2}B_4)\right)
\]
which comes from applying $\hdf$ to the 1-span
\[
\begin{tikzcd}[column sep=-1.5cm]
& (B_1{\times}_{A_1} B_2){\times}_{A_2} (B_3{\times}_{A_3} B_4) \arrow[dl]\arrow[dr] \\
(B_1{\times}_{A_1{\times}A_2} B_2){\times} (B_3{\times}_{A_2{\times}A_3} B_4) & & (B_1{\times}_{A_2} B_3){\times}_{A_1{\times}A_3} (B_2{\times}_{A_2}B_4)    
\end{tikzcd}
\]
or, more schematically:
\[
\lim\left(\adjustbox{scale=.4}{\begin{tikzcd}
    & \bullet \arrow[dl]\arrow[dr] & & & \bullet \arrow[dl]\arrow[dr] \\
    \bullet &  & \bullet & \bullet & & \bullet\\
    & \bullet \arrow[ul]\arrow[ur] & & & \bullet \arrow[ul]\arrow[ur]
\end{tikzcd}}\right)
\longleftarrow
\lim\left(\adjustbox{scale=.4}{\begin{tikzcd}
    & \bullet \arrow[dl]\arrow[dr] & & \bullet \arrow[dl]\arrow[dr] \\
    \bullet &  & \bullet & & \bullet\\
    & \bullet \arrow[ul]\arrow[ur] & & \bullet \arrow[ul]\arrow[ur]
\end{tikzcd}}\right)
\longrightarrow
\lim\left(\adjustbox{scale=.4}{\begin{tikzcd}
    & \bullet \arrow[dl]\arrow[r] & \bullet & \bullet \arrow[l]\arrow[dr] \\
    \bullet &  & & & \bullet\\
    & \bullet \arrow[ul]\arrow[r] & \bullet & \bullet \arrow[l]\arrow[ur]
\end{tikzcd}}\right)
\]

The 2-identity of a 1-span $B\to A_1{\times}A_2$ is the element
\[
\hdf\left(*\leftarrow B\to B{\times}_{A_1{\times} A_2} B\right)(1)\in\hdf(B{\times}_{A_1{\times} A_2} B).
\]

Next, we turn to the construction of the linearization functor.
\begin{definition}
    Define the 2-functor $\hdff:\spancat_2\to L_2\spancat_2$ to be the identity on objects and 1-morphisms and sending a 2-span
    \begin{equation}\label{2spanFor2Functor}
    \begin{tikzcd}
        & B_1 \arrow[dl,"f_{11}"']\arrow[dr,"f_{21}"] & \\
        A_1 & C \arrow[u,"g_1"]\arrow[d,"g_2"] & A_2 \\
        & B_2 \arrow[ul,"f_{12}"]\arrow[ur,"f_{22}"'] &
    \end{tikzcd}
    \end{equation}
    to 
    \[
    \hdf\left(*\leftarrow C\to B_1{\times}_{A_1{\times} A_2} B_2\right)(1)\in\hdf(B_1{\times}_{A_1{\times} A_2} B_2).
    \]
\end{definition}

\begin{proposition}\label{prop_hdff2functor}
    $\hdff$ is a 2-functor.
\end{proposition}

\begin{proof}
Compatibility with vertical composition follows from functoriality of $\hdf$ for the following composition of 1-spans, part of \eqref{diag:staircase}:
\begin{equation*}
\begin{tikzcd}
C_1{\times}_{B_2}C_2 \arrow[r]\arrow[d] & B_1{\times}_AB_2{\times}_AB_3 \arrow[r]\arrow[d] & B_1{\times}_AB_3 \\
C_1{\times}C_2 \arrow[r]\arrow[d] & (B_1{\times}_AB_2)\times (B_2{\times}_AB_3) &  \\
* 
\end{tikzcd}
\end{equation*}
Compatibility with whiskering follows in the same way from~\eqref{whiskerRoofDiag}.
2-identities are preserved by definition.
\end{proof}

We now replace $\spancat_2$ by $\spancat_2^{n\mathrm{CY}}$ and $\hdf$ by $\hdfx$ to construct a functor $\hdffx:\spancat_2^{n\mathrm{CY}}\to L_2\spancat_2^{n\mathrm{CY}}$. 

\begin{definition}
    Suppose $n$ is even. Define the 2-functor $\hdffx:\spancat_2^{n\mathrm{CY}}\to L_2\spancat_2^{n\mathrm{CY}}$ to be the identity on objects and 1-morphisms and sending a 2-span of the form~\eqref{2spanFor2Functor}
    to 
    \[
    \hdfx\left(*\leftarrow C\to B_1{\times}_{A_1{\times} A_2} B_2\right)(1)\in\hdfx(B_1{\times}_{A_1{\times} A_2} B_2).
    \]
\end{definition}

\begin{proposition}\label{prop_hdffx2functor}
    $\hdffx$ is a 2-functor.
\end{proposition}

\begin{proof}
The proof is similar to that of Proposition~\ref{prop_hdff2functor}, but replacing $\hdf$ by $\hdfx$ and using its functoriality (Theorem~\ref{thm:functoriality}). 
We also need that the 1-spans in the diagram in that proof are CY, which follows by definition of the source category and Lemma~\ref{lem:FuseTripple}.
\end{proof}

For computing vertical composition in $\spancat_2^{n\mathrm{CY}}$, the following formula will prove to be useful.

\begin{lemma}\label{lem:gammaTrippleFormula}
Suppose $n$ is even. For $b\in\cb_1{\times_{\ca}}\cb_2{\times}_{\ca}\cb_3$, let 
\begin{gather*}
\gamma_{\mathrm{triple}}\coloneqq\gamma\left(\cb_1{\times_{\ca}}\cb_2{\times}_{\ca}\cb_3  \longrightarrow \left(\cb_1{\times}_{\ca}\cb_3\right)\times\left(\cb_1{\times}_{\ca}\cb_2\right){\times}\left(\cb_2{\times}_{\ca}\cb_3\right),b\right) \\
\eta\coloneqq\delta(*{\times}_{\ca{\times}\ca}(\cb_1{\times}\cb_2{\times}\cb_3)),\qquad \beta_{ij}\coloneqq\delta(\cb_i{\times}_\ca\cb_j)
\end{gather*}
then
\begin{equation*}
\gamma_{\mathrm{triple}}=
\begin{cases}
2\eta-\beta_{12}-\beta_{13}-\beta_{23}-\frac{1}{2}\langle a,a\rangle_{0,\ldots,n-2} \, & \mbox{for} \,\, n \geq 2;\\
2\eta-\beta_{12}-\beta_{13}-\beta_{23}+\frac{1}{2}\langle a,a\rangle_{n-1,\ldots,-1} \, & \mbox{for} \,\, n \leq 0,
\end{cases}
\end{equation*}
where $a$ is the image of $b$ in $\ca$.
\end{lemma}

\begin{proof}
The coherent diagram of bimodules
\begin{equation*}
\begin{tikzcd}
B_1{\times}_{A}B_2{\times}_{A}B_3 \arrow[r]\arrow[d] & \left(B_1{\times}_{A}B_2\right){\times}\left(B_2{\times}_AB_3\right){\times}\left(B_1{\times}_{A}B_3\right) \arrow[d]\arrow[r] & (B_1{\times}B_2{\times}B_3)^2 \arrow[d] \\
0 \arrow[r]& \mathrm{fib}\left(B_1{\times} B_2{\times} B_3\to A\right)\arrow[r]\arrow[d] & B_1{\times} B_2{\times} B_3{\times} A^2 \arrow[d] \\
& 0 \arrow[r] & A^3
\end{tikzcd}
\end{equation*}
in which all squares are bicartesian gives, in case of $n \geq 2$,
\begin{equation}\label{gtflem_1}
\eta+\zeta=\gamma_{\mathrm{triple}}-\langle b,b\rangle_{0,\ldots,n-1}+\beta_{12}+\beta_{13}+\beta_{23}
\end{equation}
where $\zeta\coloneqq\delta(*{\times}_{\ca}(\cb_1{\times}\cb_2{\times}\cb_3))$.
Furthermore, the diagram~\eqref{diag:FuseTripple3} shows that
\begin{equation}\label{gtflem_2}
\zeta = \eta - \langle b,b\rangle_{0,\ldots,n-1}-\frac{1}{2}\langle a,a\rangle_{0,\ldots,n-2}
\end{equation}
using the fact that the defect of the left square in that diagram is
\[
-\frac{3}{2}\langle a,a\rangle_{0,\ldots,n-2}+\sum_{i=1}^3\langle b_i,b_i\rangle_{0,\ldots,n-2}
\]
where $b_i\in\cb_i$ is the image of $b$. 
Combining~\eqref{gtflem_1} and \eqref{gtflem_2} gives the claimed formula. For $n \leq 0$, the formula is obtained by a similar calculation.
\end{proof}

\section{Invariants of Legendrian knots and tangles}
\label{sec:legendrian}

The main goal of this section is to prove Theorem~\ref{thm:introLinvariants} expressing the graded ruling polynomial of a Legendrian knot $L$ as an invariant of the augmentation category of $L$.
In Section~\ref{subsec:ltangles} we recall elementary definitions around Legendrian tangles, in particular their ruling invariants.
In Section~\ref{subsec_augcats} we construct augmentation categories of Legendrian tangles as 2CY 2-spans.
Finally, in Section~\ref{subsec:augcat_count} we apply the formalism developed in the previous section to augmentation categories of Legendrian tangles and show that this recovers their ruling invariants.

\subsection{Legendrian tangles and their ruling polynomials}
\label{subsec:ltangles}

We fix throughout this subsection a grading group $\Z/N$ for some $N\in\Z_{\geq 0}$. 
All Legendrian curves are assumed to be $\Z/N$-graded.
(Later, when assigning dg-categories to Legendrian tangles, we will require $N$ to be even.) 
A \textbf{Legendrian tangle} is a properly embedded Legendrian curve $L\subset [0,1]\times\R^2$ with $\partial L\subset \{0,1\}\times \R^2$ and intersecting the boundary transversely.
Such a tangle is \textbf{generic} if the only singularities of the front projection are simple crossings and cusps, which moreover appear at distinct $x$-coordinates in $[0,1]$. 

\begin{definition}
The category $\ltanglecat$ of generic Legendrian tangles is defined as follows:
\begin{itemize}
    \item Objects are finite subsets $S\subset\R$ together with a function $\deg:S\to\Z/N$.
    \item Morphisms from $S_1$ to $S_0$ are represented by generic $\Z/N$-graded Legendrian tangles, $L$, with $L\cap [0,\varepsilon)\times\R^2=[0,\varepsilon)\times \{0\}\times S_0$ and $L\cap (1-\varepsilon,1]\times\R^2=(1-\varepsilon,1]\times \{0\}\times S_1$ for some $\varepsilon>0$, where the equality is as $\Z/N$-graded curves.
    Two tangles represent the same morphism if they are connected by a path of such tangles.
\end{itemize}
Composition of morphisms is concatenation of tangles.
Our conventions are chosen so that the front projection of the composed tangle $LR$ has $L$ on the left and $R$ on the right.
\end{definition}

\begin{remark}
There is a coarser and more natural equivalence relation on Legendrian tangles given by Legendrian isotopy, i.e. paths of Legendrian tangles which are not necessarily all generic.
We use the more restrictive equivalence relation only so that various functors from $\ltanglecat$ are easier to define. (We do not need to check invariance under Legendrian Reidemeister moves.)
\end{remark}

Any generic Legendrian tangle has an essentially unique factorization into \textbf{basic tangles}, i.e. those with precisely one singularity (crossing or cusp) in the front diagram, see Figure~\ref{fig:tangleFact}.
More precisely, any morphism in $\ltanglecat$ has a unique factorization into basic tangles, unique up to isomorphisms.

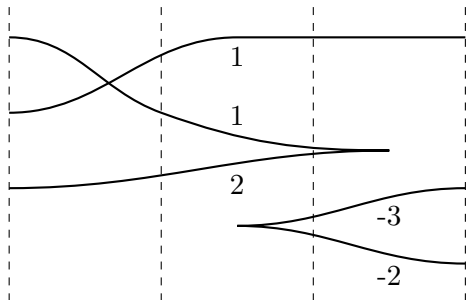
\begin{figure}[ht]
    \centering
    \begin{tikzpicture}
        \draw[thick] (-3,1) to[out=0,in=180] (0,2) to (3,2);
        \draw (0,2) node[below] {1};
        \draw[thick] (-3,2) to[out=0,in=160] (-1,1) to[out=-20,in=180] (2,.5) to[out=180,in=0] (-3,0);
        \draw (0,.7) node[above] {1};
        \draw (0,.3) node[below] {2};
        \draw[thick] (3,0) to[out=180,in=0] (0,-.5) to[out=0,in=180] (3,-1);
        \draw (2,-.1) node[below] {-3};
        \draw (2,-.9) node[below] {-2};
        \foreach \x in {-3,-1,1,3}
        {
        \draw[dashed] (\x,-1.5) to (\x,2.5);
        }
    \end{tikzpicture}
    \caption{An example of a graded Legendrian tangle and its factorization into basic tangles (by cutting along dashed lines). The numbers labeling the strands define the grading.}
    \label{fig:tangleFact}
\end{figure}

In the remainder of this subsection we will recall the definition of a certain functor, $\rpolyfun$, from $\ltanglecat$ to the category of free $\Z[z^{\pm 1}]$-modules defined by T.~Su~\cite{TaoSuThesis}, see also~\cite{HaidenFlagsTangles}.
This generalizes the graded ruling polynomial of Legendrian links, $L$, in the following sense: $\rpolyfun(\emptyset)=\Z[z^{\pm 1}]$ and $\rpolyfun(L)$ is multiplication by the ruling polynomial $P_L(z)$ of $L$.

We begin by defining rulings of finite sets and Legendrian tangles.

\begin{definition} \label{def:ruling-finite-set}
A \defword{ruling} on a graded finite subset $S\subset\R$ is a partition of $S$ into 2-element subsets $\{s,t\}$, $s<t$, so that $\deg(s)=\deg(t)+1$. 
Denote by $\rulings(S)$ the set of rulings of $S$ in this sense.
\end{definition}

\begin{definition}
Let $L$ be a generic ($\Z/N$-graded) Legendrian tangle with front $\pi(L)\subset [0,1]\times\R$.
A \defword{ruling} of $L$ is given by a collection of closed intervals $I_i\subset [0,1]$, $i=1,\ldots,n$, each with a pair of piecewise smooth functions $\alpha_i,\beta_i:I_i\to\R$, such that the following conditions hold:
\begin{enumerate}
    \item $\alpha_i(x),\beta_i(x)\in\pi(L)$ for all $x\in I_i$.
    \item $\alpha_i(x)<\beta_i(x)$ for $x\in \mathrm{int}(I_i)$.
    \item The graphs of the $\alpha_i$'s and $\beta_i$'s together cover all of $\pi(L)$ and intersect only at crossings and cusps.
    \item At a crossing, $\alpha_i$ and $\beta_i$ may ``jump'' from one branch of $\pi(L)$ to the other provided the degrees of the two branches are equal and the following condition holds.
    By the previous condition, there is a unique $j\neq i$ such that one of $\alpha_j$ or $\beta_j$ also passes through the same crossing.
    Then for $x$ slightly to the left or the right of the $x$-coordinate of the crossing, the two intervals $[\alpha_i(x),\beta_i(x)]$ and $[\alpha_j(x),\beta_j(x)]$ should be either disjoint or one contained in the other. 
    Such a crossing is called a \defword{switch} of the ruling.
    \item Let $I_i\eqqcolon[x_0,x_1]$. If $x_0\neq 0$, then $\alpha_i(x_0)=\beta_i(x_0)$ and this is a left cusp of $\pi(L)$. If $x_0=0$, then $\deg(\alpha_i(x_0))=\deg(\beta_i(x_0))+1$, where we refer to the grading on $\pi(L)$. Similarly, if $x_1\neq 1$, then $\alpha_i(x_1)=\beta_i(x_1)$ and this is a right cusp of $\pi(L)$. If $x_1=1$, then $\deg(\alpha_i(x_1))=\deg(\beta_i(x_1))+1$.
\end{enumerate}
Two rulings are identified if they differ by a permutation of the indexing set $\{1,\ldots,n\}$.
We denote the set of rulings of $L$ by $\rulings(L)$.
\end{definition}

A ruling, $\rho$, of a Legendrian tangle evidently restricts to a ruling of its left and right boundary, denoted $\partial_0\rho$ and $\partial_1\rho$ respectively.

\begin{definition}
Define the functor $\rpolyfun:\ltanglecat\to\mathrm{Mod}(\Z[z^{\pm 1}])$ on objects: $\rpolyfun(S)$ is the free $\Z[z^{\pm 1}]$-module with basis $\rulings(S)$, and on morphisms: 
\begin{equation*}
    \rpolyfun(L)(\rho_1)\coloneqq \sum_{\substack{\rho\in\rulings(L) \\ \partial_1\rho=\rho_1}} z^{\mathrm{sw}(\rho)-\mathrm{rc}(L)}\partial_0\rho
\end{equation*}
where $L$ is a tangle representing a morphism in $\ltanglecat$, $\rho_1\in\rulings(\partial_1 L)$, $\mathrm{sw}(\rho)$ is the number of switches of $\rho$, and $\mathrm{rc}(L)$ is the number of right cusps of $L$. 
\end{definition}
It is easy to see that $\rpolyfun$ is a functor. 
Less trivially, $\rpolyfun(L)$ depends only on the Legendrian isotopy class of $L$, see~\cite[Theorem 2.1.10]{TaoSuThesis}.

\subsection{Augmentation categories}
\label{subsec_augcats}

Henceforth, we assume that $N=2m$ is even ($m=0$ is also allowed) and that all dg-categories and Legendrian curves are $\Z/2m$-graded.
We also fix an arbitrary base field $\mathbf k$ for our dg-categories.

The goal of this subsection is to construct for any Legendrian tangle, $L$, a coherent diagram of dg-categories
\begin{equation}\label{diag:tanglecats}
\begin{tikzcd}
    \cc(\partial_0L) \arrow[d] & \cc(L) \arrow[l]\arrow[r]\arrow[d] & \cc(\partial_1L) \arrow[d] \\
    \mathrm{Loc}(\partial_0L) & \mathrm{Loc}(L) \arrow[l]\arrow[r] & \mathrm{Loc}(\partial_1L) 
\end{tikzcd}
\end{equation}
admitting a natural 2CY structure. More precisely, we show that the window diagram (as discussed in Section~\ref{subsec:cy2spans}) obtained by extending the above diagram by a row of three copies of the final dg-category on the top admits a 2CY structure extending the naturally defined 1CY structure on the boundary of the window diagram.

In the above, $\mathrm{Loc}(L)$ is the category of finite-rank $\mathbf k$-linear local systems on the 1-manifold with boundary $L$, $\cc(\partial_i L)$ turns out to be equivalent to the derived category of representations of the $A_n$ quiver, where $n+1=|\partial_iL|$. 
In the case when $L$ is a link, $\cc(L)$ contains, as a full subcategory, a category quasi-equivalent to the augmentation category $Aug_+(L)$ whose objects are the augmentations of the Chekanov--Eliashberg DGA of a link $L$.
Our approach for constructing the diagrams \eqref{diag:tanglecats} is to first define them for basic tangles (those with a single crossing or cusp) and then for general tangles by gluing, i.e. forming the homotopy pullback of dg-categories.
This is very similar to the ``Morse complex category'' of \cite{NRSSZ}, except that we do not assume that the front is in plat position.
The main novelty is the construction of the 2CY structures.

\subsubsection{Graded finite sets}

We begin by defining $\cc(S)$ for a finite subset $S$ of $\R$ with grading $\deg:S\to\Z/2m$.
Let $n=|S|$, then an object of $\cc(S)$ is given by an acyclic $\Z/2m$-graded complex of vector spaces over $\mathbf k$, $(C,d)$, of finite total dimension, together with an increasing filtration of graded subspaces
\begin{equation}
    0=F_0C\subseteq F_1C\subseteq \cdots \subseteq F_{n-1}C\subseteq F_nC=C
\end{equation}
preserved by the differential $d$.
The chain complex $\Hom_{\cc(S)}(C,D)$ is given by filtration preserving linear maps and with the usual differential.
We have a functor
\begin{equation}\label{grset_functor}
    [\deg]\circ\Gr\colon\cc(S)\longrightarrow\mathrm{Loc}(S)
\end{equation}
which sends a filtered acyclic complex $C$ to the local system over $S$ whose fiber over the $i$-th (smallest) point, $s_i\in S$, is the complex 
\[
\left(\Gr_iC\right)[\deg(s_i)]=\left(F_iC/F_{i-1}C\right)[\deg(s_i)].
\]

The category $\cc(S)$ is quasi-equivalent to the category $\typeacat{n-1}$ of complexes of representations of an $A_{n-1}$-type quiver.
The equivalence sends an object of $\cc(S)$ to the representation
\[
F_1C\longrightarrow F_2C \longrightarrow\cdots \longrightarrow F_{n-1}C
\]
of the $A_{n-1}$-type quiver with all arrows pointing to the right.
Conversely, given such a sequence we append any acyclic complex on the right and resolve so that all arrows become injective maps.
Thus, the functor~\eqref{grset_functor} is essentially the functor $\typeacat{n-1}\to\typeacat{1}^n$ from Example~\ref{ex:an1cy} and admits a 1CY structure.
Some care must be taken with the signs: We give the summand of $\mathrm{Loc}(S)$ corresponding to $s\in S$ the 0CY structure $(-1)^{\deg(s)}$.
Because of the shift in the definition of~\eqref{grset_functor}, these combine to a 1CY structure on the functor as in Example~\ref{ex:an1cy}.

There is a full subcategory $\mathrm{Loc}_1(S)$ of $\mathrm{Loc}(S)$ of those objects $E$ with $H^\bullet(E_s)\cong\mathbf{k}$ for all $s\in S$, i.e. rank 1 local systems in degree 0, and a corresponding full subcategory $\cc_1(S)$ of $\cc(S)$ of those objects which map to $\mathrm{Loc}_1(S)$ by the functor \eqref{grset_functor}.
The following proposition, due to Barannikov~\cite{barannikov94}, defines a one-to-one correspondence between rulings of $S$ and equivalence classes of objects in $\cc_1(S)$. 

\begin{proposition}\label{prop:rulingsClass}
    Let $S$ be a graded finite subset of $\R$.
    For each ruling, $\rho$, of $(S,\deg)$ let $C_\rho\in\cc_1(S)$ be the object with underlying chain complex $C_\rho=\bigoplus_{s\in S}(\mathbf k s)[-\deg(s)]$, $F_iC_\rho\subset C_\rho$ is the subspace generated by the $i$ smallest points in $S$, and differential with non-zero terms $d(t)=s$ if $s<t$ and $\{s,t\}$ belongs to $\rho$.
    Then the construction $\rho\mapsto C_\rho$ induces a bijection $\theta_S:\rulings(S)\to\pi_0(\cc_1(S)^\sim)$ between the set of rulings of $S$ and equivalence classes of objects of $\cc_1(S)$.
\end{proposition}

\begin{remark}
A variant of the above construction where we do not impose the acyclicity condition on the total complex $C$ leads to a category $\ce(S)$ which is quasi-equivalent to $\typeacat{n}$, $n=|S|$.
This is in general a more natural construction, though for our purposes of making comparisons with counts of rulings it suffices to restrict to the subcategory $\cc(S)\subset \ce(S)$.    
\end{remark}

\subsubsection{Vertex}

Instead of Legendrian links and tangles one could, more generally, consider Legendrian \textit{graphs}.
Here, we will not systematically work in this generality, however it will be useful (to simplify proofs) to consider one particular case: a basic ``tangle'' which has, instead of a cusp or crossing, a single trivalent vertex, see Figure~\ref{fig:vertextangle}. 

\begin{figure}[ht]
\centering
\begin{tikzpicture}[scale=1]
\draw[thick] (-2,1.5) to (2,1.5);
\node at (0,1.2) {\vdots};
\draw[thick] (-2,.65) to (2,.65);
\draw[thick] (-2,.25) to [out=-10,in=180] (0,0) to (2,0);
\draw[thick] (-2,-.25) to [out=10,in=180] (0,0);
\draw[thick] (-2,-.65) to (2,-.65);
\node at (0,-1) {\vdots};
\draw[thick] (-2,-1.5) to (2,-1.5);
\end{tikzpicture}
\caption{Front of a Legendrian graph with a single 3-valent vertex.}
\label{fig:vertextangle}
\end{figure}
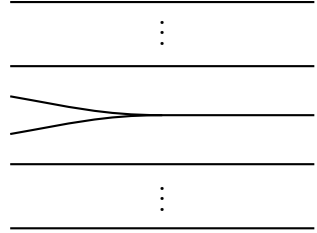

Suppose $L$ is such a Legendrian graph with a single 3-valent vertex which connects two points in $\partial_0L$ with a single point in $\partial_1 L$.
(There is no condition on the grading of the three edges which meet at the vertex, i.e. they are independent.)
In this case we define $\cc(L)\coloneqq \cc(\partial_0L)$ and the functor $\cc(L)\to\cc(\partial_1L)$ forgets the piece $F_iC$ of the filtration, where $i-1\geq 0$ is the number of strands of $L$ below the vertex.
Instead of $\mathrm{Loc}(L)$ it is more appropriate to consider the following category $\cc_0(L)$: An object of $\cc_0(L)$ is an object $E$ of $\mathrm{Loc}(\partial_1 L)$ together with a 2-step filtration (subspace) of $E_i$, the fiber over the $i$-th point of $\partial_1L$.
The functor $\cc(L)\to\cc_0(L)$ takes the associated graded with respect to the filtration $F_jC$, $j\neq i$, and remembers $F_iC\subset F_{i+1}C/F_{i-1}C$, as well as taking the appropriate shifts.
The functor $\cc_0(L)\to\mathrm{Loc}(\partial_1 L)$ forgets the 2-step filtration of $E_i$, while the functor $\cc_0(L)\to\mathrm{Loc}(\partial_0 L)$ takes the associated graded with respect to the 2-step filtration (and possibly shifts the two pieces).
Thus we obtain a diagram
\begin{equation}\label{diag:vertexcats}
\begin{tikzcd}
    \cc(\partial_0L) \arrow[d] & \cc(L) \arrow[l]\arrow[r]\arrow[d] & \cc(\partial_1L) \arrow[d] \\
    \mathrm{Loc}(\partial_0L) & \cc_0(L) \arrow[l]\arrow[r] & \mathrm{Loc}(\partial_1L) 
\end{tikzcd}
\end{equation}
of dg-categories and functors.

\begin{lemma}
\label{lem:vertex2cy}
    The diagram \eqref{diag:vertexcats} assigned to a Legendrian graph, $L$, with a single 3-valent vertex admits a natural 2CY structure.
\end{lemma}

\begin{proof}
Note that the right square in~\eqref{diag:vertexcats} is cartesian, since an $(n+1)$-step filtration is the same thing as an $n$-step filtration together with a $2$-step filtration of one of its subquotients.
Thus, the relevant square, obtained by composing the bottom and right 1-spans of the original diagram, is
\begin{equation*}
\begin{tikzcd}
    \cc(\partial_0L) \arrow[d] & \cc(L) \arrow[l]\arrow[d] \\
    \mathrm{Loc}(\partial_0L) & \cc(L) \arrow[l]
\end{tikzcd}
\end{equation*}
where both the top and right arrow are isomorphisms and the left and bottom arrows are the same 1CY functor. 
Thus the square is 2CY by Proposition~\ref{prop:idsquare}.
\end{proof}

\subsubsection{Cusps}

Let $L$ be a basic tangle with a right cusp, so $\partial_0L=\{s_1,\ldots,s_n\}$, $s_1<\cdots<s_n$, and $\partial_1L$ is in natural bijection with $\partial_0L\setminus\{s_i,s_{i+1}\}$ for some $1\leq i<n$.
We define $\cc(L)$ to be the full subcategory of $\cc(\partial_0L)$ of those objects, $C$, such that $F_{i+1}C/F_{i-1}C$ is acyclic.
The functor $\cc(L)\to\cc(\partial_0L)$ is then just inclusion, while the functor $\cc(L)\to\cc(\partial_1L)$ forgets $F_iC$ and $F_{i+1}C$.
To define a functor $\cc(L)\to\mathrm{Loc}(L)$ we need to give a quasi isomorphism 
\[
(\Gr_{i+1}C)[\deg(s_{i+1})]\longrightarrow(\Gr_iC)[\deg(s_i)]
\]
where $C\in\cc(L)$, which plays the role of parallel transport along the component of $L$ mapping to the cusp in the front.
In view of $\deg(s_i)=\deg(s_{i+1})+1$, this amounts to giving a degree 1 map $\Gr_{i+1}C\to\Gr_iC$.
Such a map arises as the connecting chain map of the short exact sequence of complexes
\[
0\to F_iC/F_{i-1}C\longrightarrow F_{i+1}C/F_{i-1}C \longrightarrow F_{i+1}C/F_iC\to 0
\]
which is well-defined up to homotopy and is a quasi-isomorphism since the middle term is acyclic.

\begin{lemma}
    The diagram \eqref{diag:tanglecats} assigned to a right cusp admits a natural 2CY structure.
\end{lemma}

\begin{proof}
The diagram \eqref{diag:tanglecats} for a cusp is obtained from the diagram \eqref{diag:vertexcats} of a 3-valent vertex by passing to full subcategories (of those objects where $F_{i+1}C/F_{i-1}C$ is acyclic), thus inherits the 2CY structure.
\end{proof}

The case of a left cusp is completely analogous and obtained by switching the roles of $\partial_0L$ and $\partial_1L$.

\subsubsection{Crossings}

Let $L$ be a basic tangle with $n$ strands and a single crossing of the $k$-th and $(k+1)$-st strand.
We define $\cc(L)$ to be the dg-category whose objects are acyclic chain complexes $C$, of finite total dimension, together with a \textit{pair} of increasing filtrations
\begin{gather*}
    0=F_0C\subseteq F_1C\subseteq \cdots \subseteq F_{n-1}C\subseteq F_nC=C \\
    0=F_0'C\subseteq F_1'C\subseteq \cdots \subseteq F_{n-1}'C\subseteq F_n'C=C
\end{gather*}
with
\begin{equation}\label{filtMutation}
F_i'C=F_iC\text{ for } i\neq k,\qquad F_kC\cap F_k'C=F_{k-1}C,\qquad F_kC+F_k'C=F_{k+1}C.   
\end{equation}
Morphisms in $\cc(L)$ are those maps which preserve both filtrations.
The functors $\cc(L)\to\cc(\partial_0L)$ and $\cc(L)\to\cc(\partial_1L)$ are given by forgetting the filtrations $F_\bullet'C$ and $F_\bullet C$, respectively.
As a consequence of \eqref{filtMutation} we have canonical isomorphisms
\[
\Gr_kC\to\Gr_{k+1}'C,\qquad \Gr_{k+1}C\to\Gr_k'C
\]
where $\Gr_i'C\coloneqq F_i'C/F_{i-1}'C$, and thus a functor $\cc(L)\to\mathrm{Loc}(L)$.

\begin{lemma}
The diagram \eqref{diag:tanglecats} assigned to a crossing admits a natural 2CY structure.
\end{lemma}

\begin{proof}
Let $L'$ be a Legendrian graph as in Figure~\ref{fig:vertextangle} with a single 3-valent vertex, and let $L''$ be the Legendrian graph obtained from $L'$ by reflection on the vertical axis.
To the composed graph $L'L''$ we can assign the diagram
\begin{equation*}
\begin{tikzcd}
    \cc(\partial_0L') \arrow[d] & \cc(L')\times_{\cc(\partial_1 L')}\cc(L'') \arrow[l]\arrow[r]\arrow[d] & \cc(\partial_1L'') \arrow[d] \\
    \mathrm{Loc}(\partial_0L') & \cc_0(L')\times_{\mathrm{Loc}(\partial_1 L')}\cc_0(L'') \arrow[l]\arrow[r] & \mathrm{Loc}(\partial_1L'') 
\end{tikzcd}
\end{equation*}
which has a 2CY structure by Lemma~\ref{lem:vertex2cy} and gluing of 2CY spans.
An object of $\cc(L')\times_{\cc(\partial_1 L')}\cc(L'')$ is and acylic chain complex with a pair of filtrations of the same length which are identical except possibly at the $k$-th step.
Thus, $\cc(L)$ is a full subcategory of this category, and similarly $\mathrm{Loc}(L)$ is a full subcategory of $\cc_0(L')\times_{\mathrm{Loc}(\partial_1 L')}\cc_0(L'')$.
Thus we can give the diagram \eqref{diag:tanglecats} assigned to a crossing the restricted 2CY structure.
\end{proof}

As a counterpart to Proposition~\ref{prop:rulingsClass} we have the following classification result.

\begin{proposition}\label{prop:CRrulingsClass}
Let $L$ be a basic tangle with a crossing, then there is a commutative diagram of sets
\begin{equation}\label{diag:CRulingsClass}
\begin{tikzcd}
\rulings(\partial_0L) \arrow[d,"\theta_{\partial_0L}"]& \rulings(L) \arrow[l,"\partial_0"]\arrow[r,swap,"\partial_1"]\arrow[d,"\cong"] & \rulings(\partial_1L)\arrow[d,"\theta_{\partial_1L}"]  \\
\pi_0(\cc_1(\partial_0L)^\sim)  & \pi_0(\cc_1(L)^\sim) \arrow[l,swap,"\partial_0"] \arrow[r,"\partial_1"] & \pi_0(\cc_1(\partial_1L)^\sim)
\end{tikzcd}
\end{equation}
where all vertical arrows are bijections and the left and right vertical arrows come from Proposition~\ref{prop:rulingsClass}.
\end{proposition}

\begin{proof}
As before, $L$ is a basic tangle with $n$ strands and a crossing of the $k$-th and $(k+1)$-st strand.
Let $l\coloneqq\deg(s_k)-\deg(s_{k+1})$ where $s_i$ is the $i$-th point in $\partial_0L$.

If $|l|\geq 2$, then we claim that $\cc_1(L)\cong\cc_1(\partial_0 L)\cong\cc_1(\partial_1 L)$.
To see this, note that if $F_\bullet C$ and $F_\bullet' C$ are complete flags on some acyclic complex $C$ defining an object in $\cc_1(L)$, then those filtrations induce the same complete flags on the various spaces of $i$-chains, $C^i$, of $C$ and are compatible with the same differentials, since they necessarily satisfy $d(F_{k+1}C)\subseteq F_{k-1}C$.
Since also $\rulings(L)\cong\rulings(\partial_0 L)\cong\rulings(\partial_1 L)$, the statement of the proposition follows in this case.

If $|l|=1$, then we claim that $\partial_i:\cc_1(L)\to\cc_1(\partial_i L)$ is fully faithful for $i=0,1$. 
Indeed, both flags $F_\bullet C$ and $F_\bullet' C$ still induce the same flags on the $C^i$'s, but they are compatible with distinct sets of differentials. 
Looking at the normal form from Proposition~\ref{prop:rulingsClass}, we see that $\cc_1(\partial_iL)$ realizes exactly those $\rho\in\rulings(\partial_iL)\cong\pi_0(\cc_1(\partial_iL)^\sim)$ where $k$ and $k+1$ are \textit{not} paired. (Thus, unless the grading group is $\Z/2$, one of the $\cc_1(\partial_i L)$ actually coincides with $\cc_1(L)$.)
This shows $\rulings(L)\cong\pi_0(\cc_1(L)^\sim)$ and this identification fits into a commutative diagram~\eqref{diag:CRulingsClass}.

Finally, we turn to the case $l=0$, which is the most interesting one.
The graded sets $\partial_i L$ are then isomorphic as objects of $\ltanglecat$ and we let $\rulings\coloneqq\rulings(\partial_0L)=\rulings(\partial_1L)$.
Partition $\rulings=\rulings_+\sqcup\rulings_-$ where $\rulings_+$ is the subset of those rulings where, if $p,q$ are such that $s_k$ is paired with $s_p$ and $s_{k+1}$ is paired with $s_q$, then the intervals $[s_k,s_p]$ and $[s_{k+1},s_q]$ are either disjoint or one is contained in the other.
Also write $\tau:\rulings\to\rulings$ for the transposition of $s_k$ and $s_{k+1}$. 
The map $\tau$ induces a bijection between $\rulings_+$ and $\rulings_-$.
The restriction to the boundary $(\partial_0,\partial_1):\rulings(L)\to\rulings^2$ is injective, so we may identify $\rulings(L)$ with its image which is
\[
\rulings(L)=\{(\rho,\rho)\mid\rho\in\rulings_+\}\sqcup\{(\tau(\rho),\rho)\mid\rho\in\rulings_+\}\sqcup\{(\rho,\tau(\rho))\mid\rho\in\rulings_+\}.
\]
Borrowing terminology from~\cite[Definition 2.1]{HR15}, rulings belonging to these three subsets are said to have a \textit{switch}, \textit{return}, or \textit{departure}, respectively, at the crossing.
Also, injectivity of $\rulings(L)\to\rulings^2$ implies that the middle vertical map in~\eqref{diag:CRulingsClass} is unique if it exists.

How can we detect if a given complete flag $F_\bullet C$ on an acyclic complex $C$ belongs to the class $\rulings_+$ or $\rulings_-$?
The plane $F_{k+1}C/F_{k-1}C$ contains a line $L_F\coloneqq F_kC/F_{k-1}C$, but it also contains a line $L_d$ which is the unique line found among the subspaces
\begin{gather*}
    \Ker\left(d:F_{k+1}C/F_{k-1}C\longrightarrow C/(F_rC+d(F_{k-1}C))\right),\qquad r<k-1 \\
    \Imm\left(d:F_sC\cap d^{-1}(F_{k+1}C)\longrightarrow F_{k+1}C/F_{k-1}C\right),\qquad s>k+1.
\end{gather*}
This can be seen by inspecting the normal form of a differential corresponding to a ruling as in Proposition~\ref{prop:rulingsClass}.
Furthermore, one finds that $L_F=L_d$ if and only if the corresponding ruling is in $\rulings_-$.

For an object in $\cc_1(L)$ given by a pair of complete flags $F_\bullet C$ and $F_\bullet' C$ on an acyclic complex $C$ we then have three lines, $L_F$, $L_{F'}$, and $L_d$, in $F_{k+1}C/F_{k-1}C$, and $L_F\neq L_{F'}$ by definition. 
There are thus three possible cases: 1) all three lines are distinct, 2) $L_F=L_d$, 3) $L_{F'}=L_d$, which match up with switches, returns, and departures, respectively.
We can argue as follows: An object $(C,F_\bullet C,F_\bullet'C)$ of $\cc_1(L)$ is determined by the object $(C,F_\bullet C)$ of $\cc_1(\partial_0 L)$ together with the choice of line $L_{F'}$. 
If $L_F=L_d$, then all choices of $L_{F'}$ give isomorphic objects, and if $L_F\neq L_d$ then there are two essentially inequivalent choices: either $L_{F'}=L_d$ or $L_{F'}\neq L_d$.
Here we are using the fact that $GL(F_{k+1}C/F_{k-1}C)$ acts transitively on triples of distinct lines.
\end{proof}

\subsubsection{Generic tangles}
\label{subsec:generic tangles}

For a general tangle, $L$, let
\[
L=L_1\cdots L_n
\]
be the (unique) factorization into basic tangles, $S_i\coloneqq\partial_1(L_i)$ and define $\cc(L)$ as the homotopy limit
\[
\cc(L)\coloneqq \cc(L_1)\times_{\cc(S_1)}\cc(L_2)\times_{\cc(S_2)}\cdots\times_{\cc(S_{n-1})}\cc(L_n).
\]
Since
\[
\mathrm{Loc}(L)\cong\mathrm{Loc}(L_1)\times_{\mathrm{Loc}(S_1)}\mathrm{Loc}(L_2)\times_{\mathrm{Loc}(S_2)}\cdots\times_{\mathrm{Loc}(S_{n-1})}\mathrm{Loc}(L_n).
\]
the universal property supplies a functor $\cc(L)\to\mathrm{Loc}(L)$.
We thus obtain a diagram of the form~\eqref{diag:tanglecats}.
A proof by induction using Proposition~\ref{prop:CYSquareVComp} and Proposition~\ref{prop:CYSquareWhisker} shows that the diagram admits a natural 2CY structure.

Recall that $\mathrm{Loc}_1(L)$ is the full subcategory of $\mathrm{Loc}(L)$ consisting of those complexes of local systems whose fiber over any point has cohomology concentrated in degree 0 and of rank 1, i.e. flat line bundles. 
We denote by $\cc_1(L)$ the full subcategory of $\cc(L)$ of those objects mapping to $\mathrm{Loc}_1(L)$.
Then there is a restricted diagram
\begin{equation}\label{diag:tanglecats1}
\begin{tikzcd}
    \cc_1(\partial_0L) \arrow[d] & \cc_1(L) \arrow[l]\arrow[r]\arrow[d] & \cc_1(\partial_1L) \arrow[d] \\
    \mathrm{Loc}_1(\partial_0L) & \mathrm{Loc}_1(L) \arrow[l]\arrow[r] & \mathrm{Loc}_1(\partial_1L) 
\end{tikzcd}
\end{equation}
with the 2CY structure induced from \eqref{diag:tanglecats} by restriction.
The dg-category $\cc_1(L)$ is called $MC(L)$ in~\cite[Section 7]{NRSSZ} where it shown to be quasi-equivalent to the augmentation category $Aug_+(L)$ of $L$. This is generalized to graphs in~\cite{ABS}.

\subsection{Counting in augmentation categories}
\label{subsec:augcat_count}

Throughout this subsection the base field is a finite field, $\F_q$.
We will define a functor $\mathscr Z:\ltanglecat\to\vectcat(\C)$ by, roughly speaking, counting objects in $\cc_1(L)$.
More precisely, we apply the formalism introduced in Section~\ref{subsec_lin2cat} to the diagram constructed for any Legendrian tangle in Section~\ref{subsec_augcats}.
We then show that $\mathscr Z$ is naturally isomorphic to the functor $\rpolyfun$ counting rulings of tangles after setting $z=q^{\frac{1}{2}}-q^{-\frac{1}{2}}$. 

Let $L$ be a Legendrian tangle, then, as discussed at the end of Section~\ref{subsec:cy2spans}, we can turn its diagram~\eqref{diag:tanglecats1}, first turned into a window diagram by adding a row of three copies of the final dg-category at the top, into a 2-span (lozenge diagram) in several ways.
The images of these 2-spans under $\hdffx$ all yield the same element 
\[
a_L\in\hdfx\left(\cc_1(\partial_0L){\times}_{\mathrm{Loc}_1(\partial_0 L)}\mathrm{Loc}_1(L){\times}_{\mathrm{Loc}_1(\partial_1 L)}\cc_1(\partial_1L)\right)
\]
given by
\begin{align*}
a_L \coloneqq &\hdfx\left(*\leftarrow \cc_1(L)\to \cc_1(\partial_0L){\times}_{\mathrm{Loc}_1(\partial_0 L)}\mathrm{Loc}_1(L){\times}_{\mathrm{Loc}_1(\partial_1 L)}\cc_1(\partial_1L)\right)(1)\\
=&\hdfx\left(*\leftarrow \cc_1(L)\to \left(\cc_1(\partial_0L){\times}_{\mathrm{Loc}_1(\partial_0 L)}\mathrm{Loc}_1(L)\right){\times}_{\mathrm{Loc}_1(\partial_1 L)}\cc_1(\partial_1L)\right)(1)\\
=&\hdfx\left(*\leftarrow \cc_1(L)\to \cc_1(\partial_0L){\times}_{\mathrm{Loc}_1(\partial_0 L)}\left(\mathrm{Loc}_1(L){\times}_{\mathrm{Loc}_1(\partial_1 L)}\cc_1(\partial_1L)\right)\right)(1).
\end{align*}
In the above, we identify iterated pullbacks and limits by the canonical equivalences.

Next, let $p_L$ be the projection functor
\[
p_L:\cc_1(\partial_0L){\times}_{\mathrm{Loc}_1(\partial_0 L)}\mathrm{Loc}_1(L){\times}_{\mathrm{Loc}_1(\partial_1 L)}\cc_1(\partial_1L)\longrightarrow \cc_1(\partial_0L){\times}\cc_1(\partial_1L).
\]
We are now ready to define $\mathscr Z$.

\begin{definition}
Define a functor (of 1-categories) $\mathscr Z:\ltanglecat\to\vectcat(\C)$ as follows.
For a graded finite subset $S\subset \R$, let
\[
\mathscr Z(S)\coloneqq \hdfx(\cc_1(S))={\C}\pi_0(\cc_1(S)^\sim).
\]
For a Legendrian tangle $L$ let
\[
\mathscr Z(L)\coloneqq q^{-\frac{1}{4}|\partial_1L|}\left(\tau_{\leq 1}(p_L)^\sim\right)_!a_L:\hdfx(\cc_1(\partial_1L))\to\hdfx(\cc_1(\partial_0L))
\]
where $\left(\tau_{\leq 1}(p_L)^\sim\right)_!$ is the push-forward as defined in Section~\ref{subsec_htpycard} and we use the isomorphism
\[
\hdfx(\cc_1(\partial_1L){\times}\cc_1(\partial_0L))\cong\Hom\left(\hdfx(\cc_1(\partial_1L)),\hdfx(\cc_1(\partial_0L))\right)
\]
coming from the standard inner product.
\end{definition}

\begin{proposition}\label{prop:Zfunctor}
    $\mathscr Z$ is a functor.
\end{proposition}

\begin{proof}
We first check that $\mathscr Z$ preserves identities.
Let $L$ be an identity tangle (no crossings or cusps), then
\[
\cc_1(L)=\cc_1(\partial_0 L)=\cc_1(\partial_1 L),\qquad \mathrm{Loc}_1(L)=\mathrm{Loc}_1(\partial_0 L)=\mathrm{Loc}_1(\partial_1 L).
\]
The composition
\[
\cc_1(L)\longrightarrow\cc_1(L){\times}_{\mathrm{Loc}_1(L)}\cc_1(L)\xrightarrow{p_L} \cc_1(L){\times}\cc_1(L)
\]
is the diagonal functor, $\Delta$, thus 
\[
\hdfx\left(*,\tau_{\leq 1}\Delta^\sim\right)(1)\in\hdfx\left(\cc_1(L){\times}\cc_1(L)\right)
\]
corresponds to the identity map on $\hdfx(\cc_1(L))$.
It remains to show that
\[
\gamma(\cc_1(L)\longrightarrow\cc_1(L){\times}_{\mathrm{Loc}_1(L)}\cc_1(L),x)=\frac{n}{2}
\]
for any $x\in\cc_1(L)$, where $n\coloneqq|\pi_0(L)|$ is the number of strands of $L$.
Because of the vanishing of the CY structure on identity squares (see Proposition~\ref{prop:idsquare}) we have
\[
\gamma(\cc_1(L)\longrightarrow\cc_1(L){\times}_{\mathrm{Loc}_1(L)}\cc_1(L),x)=\langle x,x\rangle_{0,1}
\]
and the long exact sequence of the CY structure on $c:\cc_1(L)\to\mathrm{Loc}_1(L)$ shows that
\begin{equation}\label{Zfunctor_eq1}
2\langle x,x\rangle_{0,1}=\dim\Ext^0(c(x),c(x))=n,
\end{equation}
where we use $\Ext^\bullet(c(x),c(x))=\Ext^0(c(x),c(x))=\F_q^n$.

Next we check compatibility of $\mathscr Z$ with composition.
Let $L$ and $R$ be Legendrian tangles with $\partial_1 L=\partial_0 R$, so that the composed tangle $LR$ is defined.
To improve readability we set
\begin{gather*}
    \cb_0=\cc_1(\partial_0L),\qquad\cb_1=\cc_1(\partial_1L),\qquad\cb_2=\cc_1(\partial_1R), \\
    \cb_{01}=\cc_1(L),\qquad \cb_{12}=\cc_1(R),\qquad \cb_{02}=\cc_1(LR),\\
    \ca_0=\mathrm{Loc}_1(\partial_0L),\qquad\ca_1=\mathrm{Loc}_1(\partial_1L),\qquad \ca_2=\mathrm{Loc}_1(\partial_1R), \\
    \ca_{01}=\mathrm{Loc}_1(L),\qquad\ca_{12}=\mathrm{Loc}_1(R),\qquad\ca_{02}=\mathrm{Loc}_1(LR).
\end{gather*}
Let $\kappa$ be the 2CY 1-span
\[
\begin{tikzcd}[column sep=-2cm]
    & \left(\cb_0{\times}_{\ca_0}\ca_{01}\right){\times}_{\ca_1}\cb_1{\times}_{\ca_1}\left(\ca_{12}{\times}_{\ca_2}\cb_2\right) \arrow[dl]\arrow[dr]  \\
    \left(\cb_0{\times}_{\ca_0}\ca_{01}{\times}_{\ca_1}\cb_1\right)\times\left(\cb_1{\times}_{\ca_1}\ca_{12}{\times}_{\ca_2}\cb_2\right) & & \cb_0{\times}_{\ca_0}\ca_{02}{\times}_{\ca_2}\cb_2
\end{tikzcd}
\]
used in the definition of vertical composition of the 2-spans corresponding to $L$ and $R$.
The compatibility of $\hdffx$ with vertical composition (Proposition~\ref{prop_hdff2functor}) then implies that
\begin{equation*}
\hdfx(\kappa)(a_L\otimes a_R)=a_{LR}.
\end{equation*}
Furthermore, we claim that the square
\begin{equation}\label{ZfuncDiag1}
\begin{tikzcd}
\hdfx(\cb_0{\times}_{\ca_0}\ca_{01}{\times}_{\ca_1}\cb_1)\otimes\hdfx(\cb_1{\times}_{\ca_1}\ca_{12}{\times}_{\ca_2}\cb_2)\arrow[d,"\left(\tau_{\leq 1}(p_L)^\sim\right)_!a_L\otimes \left(\tau_{\leq 1}(p_R)^\sim\right)_!a_R"]\arrow[r,"\hdf(\tau_{\leq 1}\kappa^\sim)"] & \hdfx(\cb_0{\times}_{\ca_0}\ca_{02}{\times}_{\ca_2}\cb_2) \arrow[d,"\left(\tau_{\leq 1}(p_{LR})^\sim\right)_!a_{LR}"] \\
\hdfx(\cb_0{\times}\cb_1)\otimes\hdfx(\cb_1{\times}\cb_2) \arrow[r,swap,"\text{convolution}"] & \hdfx(\cb_0{\times}\cb_2)
\end{tikzcd}    
\end{equation}
commutes.
To see this, consider the coherent diagram
\[
\begin{tikzcd}
    \cb_0{\times}_{\ca_0}\ca_{02}{\times}_{\ca_2}\cb_2 \arrow[r] & \cb_0{\times}\cb_2 \\
    \left(\cb_0{\times}_{\ca_0}\ca_{01}\right){\times}_{\ca_1}\cb_1{\times}_{\ca_1}\left(\ca_{12}{\times}_{\ca_2}\cb_2\right) \arrow[d]\arrow[r]\arrow[u] & \cb_0{\times}\cb_1{\times}\cb_2 \arrow[u]\arrow[d] \\
    \left(\cb_0{\times}_{\ca_0}\ca_{01}{\times}_{\ca_1}\cb_1\right)\times\left(\cb_1{\times}_{\ca_1}\ca_{12}{\times}_{\ca_2}\cb_2\right)\arrow[r]\arrow[d] & \cb_0{\times}\cb_1{\times}\cb_1{\times}\cb_2 \arrow[d] \\
    \ca_{01}{\times}\ca_{12} \arrow[r] & \ca_0{\times}\ca_1{\times}\ca_1{\times}\ca_2
\end{tikzcd}
\]
in which the lower two squares are cartesian.
Since $\Ext^{-1}$ vanishes in $\ca_i$, $i\in\{1,2,3\}$, the defects of the lower square and the square which is the composition of the lower two squares vanish.
By Lemma~\ref{lem:DefectAdditivity} the defect of the middle square also vanishes, thus yields a cartesian square of 1-groupoids after truncation. Proposition~\ref{prop:PushPull} then implies commutativity of \eqref{ZfuncDiag1}.
To complete the proof we need to show that
\begin{equation*}
\hdfx(\kappa)=q^{-\frac{1}{4}|\partial_1 L|}\hdf(\tau_{\leq 1}\kappa^\sim)
\end{equation*}
i.e. that $\gamma(\kappa,x)=-\frac{1}{2}|\partial_1 L|$ for any $x$, but this follows immediately from Lemma~\ref{lem:gammaTrippleFormula}, since $\Ext^{-1}$ vanishes in $\ca_1$, so all defect terms $\eta$, $\beta_{ij}$ in the formula in that lemma vanish.
\end{proof}

\begin{theorem}\label{thm:tangleFunctorIso}
There is a natural isomorphism
\[
\lambda:\mathscr Z\longrightarrow\left(\_\otimes_{\Z[z^{\pm 1}]}\C\right)\circ\rpolyfun
\]
where the homomorphism $\Z[z^{\pm 1}]\to\C$ is given by $z\mapsto q^{\frac{1}{2}}-q^{-\frac{1}{2}}$ and
\[
\lambda_S\coloneqq (q-1)^{-\frac{|S|}{4}}.
\]
More precisely, $\lambda_S$ is $(q-1)^{-\frac{|S|}{4}}$ times the isomorphism $\mathscr  Z(S)\to\rpolyfun(S)\otimes_{\Z[z^{\pm 1}]}\C$ induced by the bijection $\theta_S^{-1}:\pi_0(\cc_1(S)^\sim)\to\rulings(S)$ from Proposition~\ref{prop:rulingsClass}.
\end{theorem}

\begin{proof}
It suffices to check that naturality squares commute for basic tangles, since these generate all morphisms in $\ltanglecat$.

First, let $L$ be a basic tangle with a right cusp.
Since $L$ is a disjoint union of the component of $L$ with the cusp, denoted $\succ$, and $|\partial_1 L|$ other strands, we have $\mathrm{Loc}_1(L)=\mathrm{Loc}_1(\partial_1 L)\times\mathrm{Loc}_1(\succ)$ and hence also $\cc_1(L)=\cc_1(\partial_1 L)\times\cc_1(\succ)$, using the fact that the right square in~\eqref{diag:vertexcats} is cartesian (as was noted in the proof of Lemma~\ref{lem:vertex2cy}).
Note also that $\cc_1(\succ)\cong\mathrm{Loc}_1(\succ)$ has a single object, up to isomorphism, and $\cc_1(L)$ is a full subcategory of $\cc_1(\partial_0 L)$ by definition.
As a consequence we get a commutative diagram in the category of sets
\begin{equation*}
\begin{tikzcd}
\pi_0(\cc_1(\partial_0L)^\sim) \arrow[d,"\cong"] & \pi_0(\cc_1(L)^\sim) \arrow[hookrightarrow,l] \arrow[r,"\cong"]\arrow[d,"\cong"] & \pi_0(\cc_1(\partial_1L)^\sim)\arrow[d,"\cong"] \\
\rulings(\partial_0L) & \rulings(L) \arrow[hookrightarrow,l]\arrow[r,"\cong"] & \rulings(\partial_1L),
\end{tikzcd}
\end{equation*}
where the vertical isomorphisms come from Proposition~\ref{prop:rulingsClass}.
If $\sigma$ is the 1-span $\cc_1(\partial_0 L)\leftarrow \cc_1(L)\to\cc_1(\partial_0 L)$, this implies 
\[
\hdf(\tau_{\leq 0}\sigma^\sim)=z\rpolyfun(L),
\]
which is the map induced by the injection $\rulings(\partial_1L)\to\rulings(\partial_0L)$.

If $z\in\cc_1(L)$, then $\aut(z)=\aut(\partial_0 z)=\aut(\partial_1 z)\times\mathbf k^\times$, thus
\[
\hdf(\tau_{\leq 1}\sigma^\sim)=(q-1)^{-\frac{1}{2}}\hdf(\tau_{\leq 0}\sigma^\sim).
\]
To finish the computation of $\mathscr Z(L)$, we need to compute the $\gamma$-term of the square
\begin{equation*}
\begin{tikzcd}
\cc_1(\partial_0 L) \arrow[d] & \cc_1(L) \arrow[hookrightarrow,l] \arrow[d] \\
\mathrm{Loc}_1(\partial_0 L) & \mathrm{Loc}(L){\times}_{\mathrm{Loc}(\partial_1 L)}\cc_1(\partial_1 L)=\cc_1(L). \arrow[l]
\end{tikzcd}
\end{equation*}
Since this is, after passing to a full subcategory at the upper-left vertex, the same as the identity square on the functor $\cc_1(L)\to\mathrm{Loc}_1(\partial_0 L)$, we conclude that
\[
\gamma(\cc_1(L)\to\cc_1(\partial_0 L){\times}_{\mathrm{Loc}(\partial_0 L)}\mathrm{Loc}(L){\times}_{\mathrm{Loc}(\partial_1 L)}\cc_1(\partial_1 L),z)=\frac{|\partial_0 L|}{2}
\]
as in the proof of Proposition~\ref{prop:Zfunctor}, thus
\[
\mathscr{Z}(L)=q^{\frac{1}{2}}\hdf(\tau_{\leq 1}\sigma^\sim)
\]
using that $|\partial_0 L|=|\partial_1 L|+2$.
Putting all of the above together, we see that the naturality square for $\lambda$ commutes in the case of a right cusp.
The case of a basic tangle with a left cusp is very similar.

Next, we consider the case where $L$ is a basic tangle with $n$ strands and a crossing of the $k$-th and $(k+1)$-st strand.
Let $l\coloneqq\deg(s_k)-\deg(s_{k+1})$ where $s_i$ is the $i$-th point in $\partial_0L$.
We will make use of Proposition~\ref{prop:CRrulingsClass} and ideas and terminology contained in its proof.

If $|l|\geq 2$, then $\cc_1(L)\cong\cc_1(\partial_0 L)\cong\cc_1(\partial_1 L)$ as noted in the proof of Proposition~\ref{prop:CRrulingsClass} and the square assigned to $L$ is essentially an identity square.
Both $\rpolyfun(L)$ and $\mathscr Z(L)$ are the linear map induced by the bijection $\rulings(\partial_1 L)\to\rulings(\partial_0 L)$ coming from the bijection $\partial_0 L\cong \partial_1 L$.

If $|l|=1$, then $\cc_1(L)$ is essentially a full subcategory of $\cc_1(\partial_0 L)$ and $\cc_1(\partial_1 L)$ corresponding to those rulings where $s_k$ and $s_{k+1}$ are not paired (again, see the proof of Proposition~\ref{prop:CRrulingsClass}).
Both $\rpolyfun(L)$ and $\mathscr Z(L)$ are the linear map which is the orthogonal projection to the subspace generated by those rulings.

Finally, suppose that $l=0$, then we will show $\rpolyfun(L)=\mathscr Z(L)$ by showing that
\begin{equation}\label{CRcalc1}
\frac{|\aut(\partial_0x)|^\frac{1}{2}|\aut(\partial_1x)|^\frac{1}{2}}{|\aut(x)|}=\begin{cases} q-1 & \text{switch} \\ q^{\frac{1}{2}} & \text{return/departure}, \end{cases} 
\end{equation}
\begin{equation}\label{CRcalc2}
\langle x,x\rangle_{0,1}-\frac{n}{2}=-1,
\end{equation}
\begin{equation}\label{CRcalc3}
\gamma\left(f:\cc_1(L)\to\cc_1(\partial_0L){\times}_{\mathrm{Loc}_1(L)}\cc_1(\partial_1L),x\right)=\langle x,x\rangle_{0,1},
\end{equation}
since the overall weight factor is then $(q-1)q^{-\frac{1}{2}}=z$ in the case of a switch and $1$ in the case of a return/departure, which agrees with the definition of $\rpolyfun(L)$.

The key to showing the above identities is to note that, as graded vector spaces,
\begin{equation}
\Hom(\partial_0x,\partial_0x)=\Hom(x,x)\oplus\mathbf k,\qquad \Hom(\partial_1x,\partial_1x)=\Hom(x,x)\oplus\mathbf k,
\end{equation}
where $\mathbf k$ is placed in degree 0 and we denote the additional basis element of $\Hom(\partial_ix,\partial_ix)$ by $e_i$.
Since the differentials strictly decrease the filtrations, there can be no boundary involving $e_i$, however we will see that $e_i$ may or may not lead to additional cycles, depending on $x$.

If $x$ has a switch, then $\aut(\partial_ix)/\aut(x)=\F_q^\times$ for $i=0,1$. 
This is because elements in $\aut(\partial_0x)$ (resp. $\aut(\partial_1x)$) are required to fix only the two lines $L_d$ and $L_F$ (resp. $L_{F'}$) in the plane $F_{k+1}C/F_{k-1}C$, while elements of $\aut(x)$ need to fix all three lines $L_d$, $L_F$, and $L_{F'}$.
In the case of a return ($L_F=L_d\neq L_{F'}$) we have instead $\aut(\partial_0x)/\aut(x)=\F_q$, since elements of $\aut(\partial_0x)$ need to fix only $L_F=L_d$, while elements of $\aut(x)$ need to also fix $L_{F'}$. Moreover, $\aut(\partial_1x)=\aut(x)$, since both fix the same pair of distinct lines.
Similarly, with 0 and 1 exchanged, in the case of a departure.
This implies~\eqref{CRcalc1}.

To see~\eqref{CRcalc2}, note that by~\eqref{Zfunctor_eq1}
\[
\langle x,x\rangle_{0,1}-\frac{n}{2}=\langle x,x\rangle_{0,1}-\langle\partial_ix,\partial_ix\rangle_{0,1}
\]
for $i=0,1$, and $e_i$ either leads to additional cycles in degree 0 or additional boundary in degree 1, so in either case the difference is $-1$.

Finally,~\eqref{CRcalc3} means that the rank term from the CY structure vanishes.
This is because $f$ is injective on $\Ext^0$, which in turns follows from injectivity on $\Ext^0$ of the composed functor
\[
\cc_1(L)\longrightarrow\cc_1(\partial_0L){\times}_{\mathrm{Loc}_1(L)}\cc_1(\partial_1L)\longrightarrow\cc_1(\partial_0L){\times}\cc_1(\partial_1L).
\]
Indeed, we already know that $\Ext^0(x,x)\to\Ext^0(\partial_ix,\partial_ix)$ is injective for $i\in\{0,1\}$.
\end{proof}

\begin{figure}[ht]
    \centering
    \begin{tikzpicture}[scale=1.5]
        \begin{scope}[shift={(1,0)}]
        \draw[thick] (-4,0) to[out=0,in=180] (-3,.6) to[out=0,in=180] (-2,0) to[out=180,in=0] (-3,-.6) to[out=180,in=0] (-4,0);
        \draw (-3,.6) node[above] {$-1$};
        \draw (-3,-.6) node[below] {$0$};
        \end{scope}
        
        \draw[thick] (2,.5) to[out=0,in=180] (3,1.1) to[out=0,in=180] (4,.5) to[out=180,in=0] (3,-.1) to[out=180,in=0] (2,.5);
        \draw (3,1.1) node[above] {$-1$};
        \draw (3,-0.1) node[below] {$0$};

        \draw[thick] (2,-.5) to[out=0,in=180] (3,.1) to[out=0,in=180] (4,-.5) to[out=180,in=0] (3,-1.1) to[out=180,in=0] (2,-.5);
        \draw (3,.1) node[above] {$k$};
        \draw (3,-1.1) node[below] {$k+1$};
    \end{tikzpicture}
    \caption{Front projections of a graded Legendrian unknot (left) and a Hopf link (right).}
    \label{fig:knotexamples}
\end{figure}
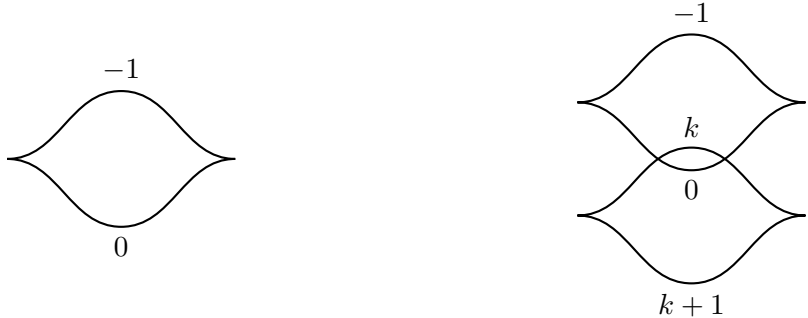

\begin{example}[Unknot]
Consider the representative of the Legendrian unknot, $L$, shown on the left of Figure~\ref{fig:knotexamples}, for arbitrary grading group $\Z/2m$.
There is a unique ruling and the ruling polynomial is $\rpoly_L(z)=z^{-1}$.
The unique object $x$, up to isomorphism, of $\cc_1(L)$ is given by the elementary acyclic filtered complex $\mathbf k\oplus \mathbf k[1]$.
We have 
\[
\dim\Ext^i(x,x)=\begin{cases} 1 & i=0 \mod 2m \\ 0 & \text{else} \end{cases}
\]
so $\gamma(x)=1$ and 
\[
\sum_{x\in\pi_0(\cc_1(L)^\sim)}\frac{q^{\frac{1}{2}\gamma(x)}}{|\mathrm{Aut}(x)|}=\frac{q^{\frac{1}{2}}}{q-1}=\rpoly_L\left(q^{\frac{1}{2}}-q^{-\frac{1}{2}}\right).
\]
\end{example}

\begin{example}[Hopf link]
Consider the representative of the Legendrian Hopf link, $L$, shown on the right of Figure~\ref{fig:knotexamples}, where the grading depends on $k\in\Z/2m$. 
In all cases there is a ruling without switches which contributes a summand $z^{-2}$ to $\rpoly_L(z)$.
If $k=0$, then there is a second ruling where both crossings are switches and which contributes a summand $1$ to $\rpoly_L(z)$.

\underline{Case $k\neq 0\mod 2m$}:
We claim that there is a unique object $x$ in $\cc_1(L)$, up to isomorphism, and $\Ext^\bullet(x,x)$ is 4-dimensional and given by the path algebra of the graded quiver
\[
\begin{tikzcd}
    \bullet \arrow[r,bend left=10,"1+k"] & \bullet \arrow[l,bend left=10,"1-k"]
\end{tikzcd}
\]
with relations that both compositions of the two arrows are zero.
If $k\neq\pm 1$, this gives 
\[
\frac{q^{\frac{1}{2}\gamma(x)}}{|\mathrm{Aut}(x)|}=\frac{q^{\frac{1}{2}(2-0+0)}}{(q-1)^2}=\left(q^{\frac{1}{2}}-q^{-\frac{1}{2}}\right)^{-2}.
\]
If $k=\pm 1$ and $m\neq 1$ this gives
\[
\frac{q^{\frac{1}{2}\gamma(x)}}{|\mathrm{Aut}(x)|}=\frac{q^{\frac{1}{2}(3-0+1)}}{(q-1)^2q}=\left(q^{\frac{1}{2}}-q^{-\frac{1}{2}}\right)^{-2}
\]
where $\mathrm{rk}\left(\Ext^0(x,x)^\vee\to\Ext^2(x,x)\right)=1$ by inspection of the long exact sequence.
Finally, if $k=\pm 1$ and $m=1$ then
\[
\frac{q^{\frac{1}{2}\gamma(x)}}{|\mathrm{Aut}(x)|}=\frac{q^{\frac{1}{2}(4-0+2)}}{(q-1)^2q^2}=\left(q^{\frac{1}{2}}-q^{-\frac{1}{2}}\right)^{-2}.
\]
where $\mathrm{rk}\left(\Ext^0(x,x)^\vee\to\Ext^2(x,x)\right)=2$. 

It remains to prove the claim about the augmentation category, $\cc_1(L)$.
We sketch how this can be done using our definition.
First, note that $\cc_1(L)$ has a unique object, $x$, up to isomorphism, which is the direct sum of two object $x^\sharp$ and $x^\flat$ in $\cc(L)$ corresponding to the upper and lower component of $L$, respectively.
We want to show that
\[
\Ext^\bullet(x^\sharp,x^\flat)=\mathbf k[-(1+k)],\qquad \Ext^\bullet(x^\flat,x^\sharp)=\mathbf k[-(1-k)].
\]
To see this, cut $L$ into two tangles $L_-$ and $L_+$ along the central vertical line intersecting $L$ in the 4-element set $L_0$.
The categories $\cc(L_-)$ and $\cc(L_+)$ are equivalent and by definition full subcategories of the category assigned to a basic tangle with a single crossing.
Denote the images of $x^\sharp$ and $x^\flat$ in $\cc(L_?)$ by $x^\sharp_?$ and $x^\flat_?$, where $?\in\{-,0,+\}$.
Then,
\begin{gather*}
\Ext^\bullet(x^\sharp_-,x^\flat_-)=\Ext^\bullet(x^\sharp_0,x^\flat_0)=\Ext^\bullet(x^\sharp_+,x^\flat_+)=\mathbf k[-(1+k)] \\
\Ext^\bullet(x^\flat_-,x^\sharp_-)=\Ext^\bullet(x^\flat_+,x^\sharp_+)=0,\qquad \Ext^\bullet(x^\flat_0,x^\sharp_0)=\mathbf k[-k].
\end{gather*}
The claim then follows using the long exact sequence of the defining cartesian square of $\cc(L)$ in terms of $\cc(L_\pm)$. 

\underline{Case $k=0\mod 2m$}:
We claim that $\cc_1(L)$ has the following numerics.
\begin{center}
{\renewcommand{\arraystretch}{1.4}\setlength{\tabcolsep}{8pt}
\begin{tabular}{c|cc}
     &  $1$ isoclass & $q$ isolasses \\
     \hline
 $|\aut(x)|$ & $(q-1)^2$ & $q-1$ \\
 $\dim\Ext^0(x,x)$ & $2$ & $1$ \\
 $\dim\Ext^1(x,x)$ & $2$ & $1$ \\
 $\dim\Ext^{\neq 0,1}(x,x)$ & $0$ & $0$ \\
 $\gamma(x)$ & $0$ & $0$
\end{tabular}}    
\end{center}
We compute
\begin{align*}
\sum_{x\in\pi_0(\cc_1(L)^\sim)}\frac{q^{\frac{1}{2}\gamma(x)}}{|\mathrm{Aut}(x)|}&=\frac{1}{(q-1)^2}+\frac{q}{q-1} \\ 
&=\frac{q^2-q+1}{(q-1)^2} \\
&=\frac{q}{(q-1)^2}+1 \\
&=\rpoly_L\left(q^{\frac{1}{2}}-q^{-\frac{1}{2}}\right)
\end{align*}
as expected.
We remark that $q^2-q+1$ is the number of $\F_q$-points of the augmentation variety of $L$ which is the cluster variety of type $A_1$ with one invertible frozen variable 
$\{(x,y)\in\mathbb A^2\mid xy\neq 1\}$ (see e.g. \cite{GaoShenWeng}).

Let us sketch how to verify the claims about the structure of $\cc_1(L)$.
Let $E$ be a 2-dimensional vector space over $\mathbf k$ and $L_i\subset E$, $i=1,2,3,4$ four 1-dimensional subspaces in $E$ such that
\[
L_1\neq L_0,\qquad L_3\neq L_0,\qquad L_1\neq L_2,\qquad L_2\neq L_3.
\]
Consider the acyclic chain complex $L_0\to E\to E/L_0$, with $E$ in degree zero, together with the triple of complete flags
\[
0\subset E/L_0[-1]\subset E/L_0[-1]\oplus L_i\subset E/L_0[-1]\oplus E\subset E/L_0[-1]\oplus E\oplus L_0[1]
\]
for $i=1,2,3$.
This defines an object, $x$, of $\cc_1(L)$, and we obtain all objects, up to isomorphism, in such a way.
If $L_2=L_0$ and $L_1=L_3$, then, as before, $x=x^\sharp\oplus x^\flat$ in $\cc(L)$ and $\Ext^\bullet(x,x)$ can be computed in the same way.
In the remaining cases, we have at least three distinct lines in $E$ and $x$ is thus indecomposable.
There is one isoclass with $L_2=L_0$, $L_1\neq L_3$, and $(q-1)$ isoclasses with $L_2\neq L_0$, since we have three distinct lines $L_0,L_1,L_2$ and a fourth line $L_4\neq L_0,L_2$, i.e. a choice in $\mathbb F_q^\times$. 
For all these $q$ isoclasses one can show that $\Ext^{0,1}(x,x)=\mathbf k$ and $\Ext^{\neq 0,1}(x,x)=0$. We omit the details.
\end{example}

\bibliographystyle{plain} 
\bibliography{refs}

\Addresses

\end{document}